\documentclass{article}
\usepackage{arxiv}
\usepackage[utf8]{inputenc}
\usepackage{geometry}
\usepackage{float}
\usepackage{graphicx,xcolor}
\usepackage{amsmath,amssymb,amsfonts,amsthm,mathtools}
\usepackage{stmaryrd}
\usepackage{multirow}
\usepackage{siunitx}
\usepackage{booktabs,makecell}
\usepackage{pgfplots}
\usepackage[export]{adjustbox}
\usepackage{titlesec}
\usepackage{pifont}
\usepackage{caption}
\usepackage{hyperref}
\usepackage{nccmath}
\definecolor{review1}{RGB}{80,80,179}
\definecolor{review2}{RGB}{217,95,2}
\definecolor{review3}{RGB}{27,158,119}
\definecolor{allreviews}{RGB}{231,41,138}
\definecolor{todiscuss}{RGB}{127,150,127}

\newcommand{\ra}{\textcolor{black}}
\newcommand{\rb}{\textcolor{black}}

\newcommand{\rall}{\textcolor{black}}

\usetikzlibrary{math}
\pgfplotsset{compat=newest}
\usepackage{comment}
\titleformat{\paragraph}[runin]{\normalfont\bfseries}{\theparagraph}{1em}{}
\titleformat{\section}
{\normalfont\fontsize{14}{17}\sffamily\bfseries}
{\thesection}
{1em}
{}
\titleformat{\subsection}
{\normalfont\fontsize{11}{15}\sffamily\bfseries\slshape}
{\thesubsection}
{1em}
{}

\newcommand{\Vone}{{\boldsymbol{e}}}
\newcommand{\mycomment}[1]{}
\newcommand{\Id}{\boldsymbol{I}_\textrm{d}}

\newcommand{\trans}{{\textrm{\rm t}}}

\newcommand{\ba}{{\boldsymbol{a}}}
\newcommand{\bb}{{\boldsymbol{b}}}
\newcommand{\be}{{\boldsymbol{e}}}

\newcommand{\balpha}{{\boldsymbol{\alpha}}}
\newcommand{\bbeta}{{\boldsymbol{\beta}}}
\newcommand{\bchi}{{\boldsymbol{\chi}}}
\newcommand{\bvert}{{\big\vert}}
\newcommand{\bVert}{{\big\Vert}}
\def\real{\mathop{\rm Re}\nolimits}

\newcommand{\smallBB}{{\hskip 0.5em $\scriptstyle\blacksquare$}}

\newcommand{\bint}[1]{{\,\in\llbracket{#1}\rrbracket}}
\newcommand{\bbint}[1]{{\llbracket{#1}\rrbracket}}

\newcommand{\ddf}[2]{{\texttt{df}^{(#1)}\big[{#2}\big] }}
\newcommand{\PE}[1]{{{\texttt{PE}\raisebox{-0.3em}{\texttt{\,\scriptsize {#1}}}}}}

\newcommand{\SEE}[3]{{{\texttt{SE}\raisebox{-0.3em}{\texttt{\,\scriptsize {#1}\,{#2}}}}\raisebox{+0.5em}{\texttt{\scriptsize{\hskip-1.1em  {#3} }}}\hskip 0.2em}}
\newcommand{\SKfoot}[2]{{\texttt{Skm[{\scriptsize #1,#2}]}}}
\newcommand{\SK}[2]{{\texttt{Skm[{\small #1,#2}]}}}
\newcommand{\blackline}{
{\raisebox{1.5pt}{\tikz{\draw[black,line width = 0.5pt](0.,0.8mm) -- (5mm,0.8mm)}}}}
\newcommand{\blueline}{
{\raisebox{1.5pt}{\tikz{\draw[blue,line width = 0.5pt](0.,0.8mm) -- (5mm,0.8mm)}}}}
\newcommand{\redline}{
{\raisebox{1.5pt}{\tikz{\draw[red,line width = 0.5pt](0.,0.8mm) -- (5mm,0.8mm)}}}}
\newcommand{\greenline}{
{\raisebox{1.5pt}{\tikz{\draw[green,line width = 0.5pt](0.,0.8mm) -- (5mm,0.8mm)}}}}

\newcommand{\resizeboxlarger}[1]{
\resizebox{\ifdim\width>\textwidth\textwidth\else\width\fi}{!}{#1}}
\newfont{\notapolice}{cmss11}
\newfont{\captionpolice}{cmss9}
\newfont{\tablepolice}{cmtt10}
\newfont{\smalltablepolice}{cmtt8}
\newfont{\defpolice}{cmssi11}
\theoremstyle{plain}
\newtheorem{theorem}{Theorem}[section]
\newtheorem{lemma}[theorem]{Lemma}
\newtheorem{prop}[theorem]{Proposition}
\newtheorem{cor}[theorem]{Corollary}
\theoremstyle{myDefinition}
\newtheorem{defn}{Definition}[section]
\theoremstyle{myRemark}
\newtheorem*{rem}{Remark}
\newtheorem{example}{Example}
\makeatletter
\def\fixedlabel#1#2{%
\@bsphack%
\protected@write\@auxout{}%
{\string\newlabel{#1}{{#2}{\thepage}}}%
\@esphack}
\makeatother
\RequirePackage{luatex85}
\usepackage{pgfplots}
\pgfplotsset{compat=newest}
\usepgfplotslibrary{groupplots}
\usepgfplotslibrary{polar}
\usepgfplotslibrary{smithchart}
\usepgfplotslibrary{statistics}
\usepgfplotslibrary{dateplot}
\usepgfplotslibrary{ternary}
\usepackage{tabularray}
\definecolor{oiBlue}{HTML}{0072B2}
\definecolor{oiOrange}{HTML}{E69F00}
\definecolor{oiGreen}{HTML}{009E73}
\definecolor{oiRed}{HTML}{D55E00}
\definecolor{oiPurple}{HTML}{CC79A7}
\definecolor{oiSky}{HTML}{56B4E9}
\definecolor{oiYellow}{HTML}{F0E442}
\definecolor{oiBlack}{HTML}{000000}
\pgfplotscreateplotcyclelist{okabe-ito}{%
{oiBlue,   mark=square*},
{oiOrange, mark=square*},
{oiGreen,  mark=square*},
{oiRed,    mark=square*},
{black, mark=square*}
}
\def\emdash{\text{\normalfont ---}}
\newcommand{\specialcell}[2]{\begin{tabular}[t]{@{}#1@{}}#2\end{tabular}}
\BeforeBeginEnvironment{tabular}{\tablepolice\renewcommand{\arraystretch}{0.7}}
\BeforeBeginEnvironment{tikzpicture}{\tablepolice\renewcommand{\arraystretch}{0.7}}
\title{The structural method for ordinary differential equations:\\ toward extreme-orders of accuracy}
\author{
    \href{https://orcid.org/0000-0003-2295-5118}{\includegraphics[scale=0.06]{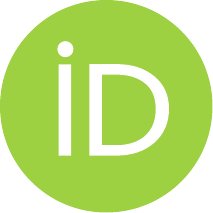}\hspace{1mm}S. Clain} \\
	Centre of Mathematics, Coimbra University, Largo D. Dinis, 3000-143 Coimbra, Portugal \\
	\texttt{clain@mat.uc.pt} \\
	\And
	\href{https://orcid.org/0000-}{\includegraphics[scale=0.06]{orcid.pdf}\hspace{1mm}M.T. Malheiro} \\
	Centre of Mathematics, Campus de Gualtar, 4710-057 Braga, Portugal \\
	\texttt{mtm@math.uminho.pt} \\
	\And
	\href{https://orcid.org/0000-}{\includegraphics[scale=0.06]{orcid.pdf}\hspace{1mm}G.J. Machado} \\
	Centre of Mathematics, Campus de Gualtar, 4710-057 Braga, Portugal \\
	\texttt{gjm@math.uminho.pt} \\
    \And
	\href{https://orcid.org/0000-}{\includegraphics[scale=0.06]{orcid.pdf}\hspace{1mm}R. Costa} \\
	Institute for Polymers and Composites, Campus de Azur\'em, University of Minho,\\ 4080-058 Guimar\~aes, Portugal \\
	\texttt{rcosta@dep.uminho.pt}\\
    }
\renewcommand{\headeright}{Article}
\renewcommand{\undertitle}{Article}
\renewcommand{\shorttitle}{The structural method for Ordinary Differential Equations}
\hypersetup{
pdftitle={The structural method for ordinary differential equations:\\ toward extreme-orders of accuracy},
pdfsubject={math.NA},
pdfauthor={Stephane Clain, Maria Teresa Malheiro, Gaspar J. Machado, Ricardo Costa},
pdfkeywords={structural method, compact scheme, very high order, spectral resolution},
}
\begin{document}
\maketitle
\begin{abstract}
We design and analyse a new numerical method to solve ODE system based on the structural method. We compute approximations of solutions together with its derivatives up to order $K$ by solving an entire block corresponding to $R$ time steps. We build the physical relations that connect the function and derivative approximations at each time step by using the ODE and its derivatives and develop the structural equations that establish linear relations between the function and its derivative over the whole block of $R$ times steps. The non-linear system is solved and provide very accurate approximations with nice spectral resolution properties.
\end{abstract}
\keywords{structural method, very high-order, extreme-order, Compact scheme, EDO, stability, multi-derivatives}

\section{Introduction}
First-order ordinary differential equations (ODEs), or systems of the form $\phi'(t)=f(\phi(t),t)$, with $t\mapsto\phi(t)\in\mathbb{R}^I$, arise in most non-stationary models to characterise the evolution of processes over time. The design of accurate and stable numerical schemes for solving these ODEs is a fundamental challenge for improving computational efficiency in critical settings, such as stiff problems or chaotic differential equations.

In the early 1970s, researchers began using the second derivative as an additional variable to improve solution accuracy and reduce the number of substeps in numerical schemes. This approach is rooted in the methodology of compact schemes, which exploit more local information to minimise stencil size~\cite{Pettigrew1989,Lele1992}. Early studies on multi-derivative methods include works by Urabe~\cite{U70}, Shintani~\cite{Hi71,Hi72}, Kastlunger, Hairer and Wanner~\cite{KW72a,KW72b,HW73}, and Cash~\cite{C78} within the Runge--Kutta (RK) framework. In the multistep setting, early contributors include Makinson~\cite{M68}, Reimer~\cite{R68}, Liniger and Willoughby~\cite{LW70}, Enright~\cite{E74a,E74b}, Genin~\cite{G74}, Makela et al.~\cite{MNS74}, Jeltsch~\cite{J76a,J76b,J77a,J77b,J79}, Brown~\cite{B77}, Gupta~\cite{G78}, and Cash~\cite{C81a,C81b}. Lambert's book~\cite{L73} discusses both approaches. References~\cite{GW86,GJ00}, together with Butcher's textbook~\cite{Butcher2008}, highlight Obrechkoff's contribution~\cite{O40}, with particular emphasis on the two-derivative case, as further explored in~\cite{BH05,CT10,TCW14}.

Following a productive decade in multi-derivative ODE integration, research activity declined between the 1980s and the 2010s, with only a limited number of publications. For RK methods, the main difficulty lies in the complex and highly nonlinear order conditions~\cite{GW86,G88,GJ00,ChanChan2004}, which complicate the construction of fourth- or higher-order schemes with more than two or three substeps. Several papers address this issue by imposing additional conditions to reduce complexity. These include explicit schemes~\cite{OY04,AOS05}, methods with few substeps~\cite{BT97}, and two-derivative cases~\cite{II99,BH05,HA06,CT10}. \\

\ra{Alternative approaches have also been proposed to achieve very high-order approximations, in which the algebraic order conditions are replaced by a few simple linear constraints. The collocation method~\cite{GS69,W70} provides approximations of arbitrary order and has been identified as a particular case of a class of implicit RK methods~\cite{CAP11}. Given the collocation points in a time interval, the algebraic order conditions become linear conditions, as expressed by relations (2.7) and (2.8) in~\cite{CAP11}. The method is not compact in the sense considered here, since it involves implicit linear combinations of the function and its derivatives at the collocation points. Nevertheless, we would like to mention the recent works of~\cite{FH20} and~\cite{LHM26}, where the authors explore a novel approach by introducing second-order derivatives to reduce the stencil size.
\\
The spectral deferred correction method is another class of high-order schemes based on successive approximations of the solution over a time interval $[t_n,t_{n+1}]$, as shown in~\cite{DGR00,M03}; see~\cite{OS20} for an extensive review of the SDC method. This method also avoids the introduction of complicated nonlinear algebraic conditions~\cite{COQ09}. A $p$-subdivision $t_{n,m}$, $m=0,\ldots,p$, together with local approximations $\phi_{n,m}$, is introduced to design an iterative scheme that provides a sequence $\phi_n^{[k]}$ of approximations of order $k+1$. The upgrade operator $\phi_n^{[k]}\to \phi_n^{[k+1]}$, defined on the interval $[t_n,t_{n+1}]$, makes it possible to increase the order of accuracy using very simple numerical tools, such as first-order Euler schemes and quadrature formulae for integration. Very recently, a multi-derivative version was proposed by introducing the second-order derivative~\cite{ZSS22}.
}\\

\rb{
The structural method is a general technique that provides very high-order approximations by introducing derivatives up to degree $K$ at each node $t_n$. Moreover, it clearly separates the physical equations, which represent the underlying mathematical model, from the structural equations, which are independent of the problem under consideration and depend only on the grid/mesh structure. The combination of these two sets of equations then yields a global (non-)linear system that provides highly accurate approximations of the function and its derivatives~\cite{CPPL23}. }

\rb{This method can be regarded as a generalisation of compact schemes in the sense of Lele~\cite{Lele1992}, applied to the ODE setting, where implicit relations between the function and its derivatives are introduced over a prescribed stencil. In our case, the stencil is referred to as an $R$-block and consists of $R$ successive time steps. As an example, the combined compact difference scheme for elliptic equations in space proposed in~\cite{Chu1998,Chu1999,Chu2000} can be reformulated as a structural method for ODEs in time; see~\cite{Clain2023}. It is worth noting that the same structural equations that yield a sixth-order accurate scheme in time for ODEs coincide with the Chu and Fan relations within the compact finite difference framework~\cite{Chu1998}.}\\

\rb{
The structural method proposed in the present work corresponds to a particular case in which, by taking advantage of all the physical relations, as explained in Sections 2 and 3, the scheme is reinterpreted as a multi-derivative RK (MDRK) scheme. The key point is that MDRK schemes are difficult to construct because of the large number of nonlinear algebraic order conditions involved. In summary, handling the order conditions for $K$-derivative methods of order $p$ is challenging. For instance, reference~\cite{GG26} lists 33 strongly nonlinear order conditions for a two-derivative RK method with $s$ substeps to provide a sixth-order scheme ($p=6$; see also~\cite{KW72a,HMS94,NBK09}). These conditions involve $2s(s+1)$ coefficients, arising from combinations of first and second derivatives at the different substeps. For example, incorporating third- or higher-order derivatives with $s$ substeps in order to achieve very high accuracy appears to be intractable, potentially involving thousands of nonlinear relations to be satisfied~\cite{KW72a,GW86,F12,MF21,QJY23}.
}

\rb{
The main advantage of structural schemes is that there is no need to verify nonlinear algebraic order conditions in order to achieve the prescribed order. Instead, the coefficients of the method are obtained from vectors belonging to the kernel of a matrix, leading to a highly efficient computation of the structural equations. Recently, the MDRK literature has focused on techniques to alleviate this major computational drawback. Explicit formulae for three-derivative and higher-derivative RK schemes have been proposed in~\cite{EH12,WA13}, while the development of new multistep schemes has been studied in~\cite{MF21,MA22}, with some improvements in accuracy reported in~\cite{NB09,AH11,OI15}. Applications include time discretisation for non-stationary partial differential equations (PDEs)~\cite{SG14,TC14,SS17,JS24} and the construction of strong stability preserving (SSP) schemes of RK type~\cite{CG16,MA20,GG22,AA22,QJ23} or multistage type~\cite{MF19,GG19,MA23} for conservation laws. A recent review~\cite{GG26} discusses these developments. However, all these methods address only second- or third-derivative schemes with a small number of substeps, owing to the large number of nonlinear algebraic order conditions involved, whereas the structural method efficiently produces $K$-derivative schemes over an $R$-block at any order $p$ by solving a simple linear problem.}\\

This paper introduces a generalisation of the structural method for ODE systems using $K$ derivatives, addressing key challenges in existing RK methods. The main advantages are the following:
\begin{itemize}
\item The principal contributions are as follows: The method corresponds to a $K$-derivative implicit RK method when all possible physical equations are used.
\item The intermediate steps are not auxiliary variables, but genuine approximations that enjoy the optimal order of accuracy of the method.
\item Unconditional stability is proved together with a formal assessment of the consistency and {\it a priori} errors.
\item The coefficients of the structural equations are computed from the kernel of a matrix, avoiding an intractable non-linear system to solve.
\item The method covers all possible meta-parameters $K$ (number of derivatives), $R$ (number of intermediate steps), and $p$ (order of the method) using the same algorithm.
\item Excellent spectral resolution is obtained when $K$ increases. In particular, only two or three grid points are necessary to perfectly recover a full revolution; hence, the method is highly effective at capturing high-frequency relative to the grid characteristic size.
\end{itemize}
\rall{The document is divided into two parts. The first is devoted to the design of the method, its implementation, comparisons with existing methods, and numerical simulations. The second part addresses the theoretical aspects of convergence, stability, and spectral resolution, together with a detailed study and justification of our implementation for solving the nonlinear global system. }

\vskip 2em
\centerline{\includegraphics[width=0.30\textwidth]{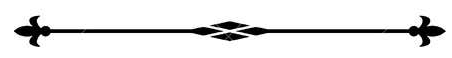}}
\centerline{\Large \sl Part I: Method design and benchmarks}
\vskip 0.5em

{\sl \small The first part focuses on the general construction of the structural equations and on the coupling between the structural and physical equations. In Section~2, we detail the technical aspects of the method, leading to different types of ODE solvers, while comparisons with classical high-order methods for ODEs are provided in Section~3. Section~4 is devoted to the presentation of a simple algorithm for solving the nonlinear systems, whereas Sections~5 and 6 present numerical evidence concerning the convergence and efficiency of the structural method for ODEs. Annexes A and B include analytical formulae for the implicit structural equations on a uniform grid, together with explicit structural equations for post-processing. We note that a numerical study for the specific case of Hamiltonian systems was carried out in~\cite{CFM25}; however, our goal here is to generalize the approach to a generic ODE system.}<.

\section{General design of the structural method}

Consider the generic first-order scalar ODE problem
\begin{equation}
\label{eq::ODE}
\phi'(t)=f(\phi(t),t),\,t\in[0,T],
\end{equation}
together with the initial condition $\phi(0)=\phi_0$, with $T>0$ the final time and $f\equiv f(z,t)$ a regular function in $\mathbb{R}\times[0,T]$.

Let $N\in\mathbb{N}$, and define a uniform grid $t_n = n\Delta t$ for $n\bint{0,N}$, such that $T = N\Delta t$. Let $\phi^{(k)}\equiv\phi^{(k)}(t)$ denote the $k$th derivative of the solution, with $\phi^{(0)}(t)\equiv\phi(t)$ as in Eq.~\eqref{eq::ODE}. The approximations $\phi^{(k)}_n$ represent $\phi^{(k)}(t_n)$ at each time grid point.

The structural method divides the problem into two distinct subsystems. The first subsystem contains the physical equations, which are derived directly from the generic ODE and its time derivatives. In contrast, the second subsystem comprises structural equations, defined by linear relations that depend only on the time grid structure. Notably, these structural equations do not involve the specific ODE being solved, emphasising their generality.

\subsection{Physical equations}

The Eq.~\eqref{eq::ODE}, which relates the first derivative $\phi^{(1)}(t)$ with function $\phi^{(0)}(t)$ and time $t$, reads
\begin{equation}\label{eq::ODE:K1}
\phi^{(1)}(t)=f\left(\phi^{(0)}(t),t\right)
\equiv\ddf{0}{\phi^{(0)},t}.
\end{equation}
The time derivative of Eq.~\eqref{eq::ODE:K1} provides an additional physical equation, that reads
\begin{equation}\label{eq::ODE_K2}
\phi^{(2)}(t)=\partial_z f\left(\phi^{(0)}(t),t\right)\phi^{(1)}(t)+\partial_t f\left(\phi^{(0)}(t),t\right)
\equiv\ddf{1}{\phi^{(0)},\phi^{(1)},t}
\end{equation}
where $\partial_z f$ and $\partial_t f$ denote the derivatives with respect to the first and second arguments, respectively. Once again, the time derivative of Eq.~\eqref{eq::ODE_K2} provides a relation between the function $\phi(t)$ and all its derivatives up to the third order, that is,
\begin{multline}\label{eq::ODE_K3}
\phi^{(3)}(t)
=\partial_z f\left(\phi^{(0)}(t),t\right)\phi^{(2)}(t)+\partial_{zz} f\left(\phi^{(0)}(t),t\right)\left[\phi^{(1)}(t)\right]^2\\
+ 2\partial_{zt} f\left(\phi^{(0)}(t),t\right)\phi^{(1)}(t)+\partial_{tt} f\left(\phi^{(0)}(t),t\right)
\equiv\ddf{2}{\phi^{(0)},\phi^{(1)},\phi^{(2)},t}.
\end{multline}
By systematically applying the chain rule $K-1$ times, we obtain a sequence of $K$ physical relations. These can be represented in a generic form as
\[
\phi^{(k+1)}(t)=\frac{\textrm{d}}{\textrm{d}t}\ddf{k-1}{\phi^{(0)},\phi^{(1)},\ldots,\phi^{(k-1)},t}
\equiv\ddf{k}{\phi^{(0)},\phi^{(1)},\ldots,\phi^{(k)},t}
,\quad k\bint{0,K-1}.
\]

\begin{rem}
Note that, by construction, the application
\[
(z_0,z_1,\ldots,z_k,t)\to\ddf{k}{z_0,z_1,\ldots,z_k,t}
\]
is a polynomial function with respect to $z_1,\ldots,z_k$ with non-linear coefficients solely depending on $z_0$ and $t$ corresponding to the partial derivatives in $z$ and time.
\hfill\smallBB
\end{rem}

We can now introduce the notion of physical equation, which just consists of substituting the exact values with the approximations $\phi^{(k)}_n(t)$ in the previous equations. Therefore, the first physical equation, associated with time $t_n$ and denoted as $\PE{1}(n)$, derives from the relation~\eqref{eq::ODE:K1} and reads
\[
\phi^{(1)}_n=f(\phi^{(0)}_n,t_n)
\equiv\ddf{0}{\phi_n^{(0)},t_n}.
\]
Similarly, the second physical equation $\PE{2}(n)$ is deduced from relation~\eqref{eq::ODE_K2}, that is
\[
\phi^{(2)}_n=\partial_z f(\phi^{(0)}_n,t_n)\phi^{(1)}_n+\partial_t f(\phi^{(0)}_n,t_n)\equiv\ddf{1}{\phi_n^{(0)},\phi_n^{(1)},t_n},
\]
and the third physical equation $\PE{3}(n)$ from relation~\eqref{eq::ODE_K3}, that is
\begin{multline*}
\phi^{(3)}_n
=\partial_z f\left(\phi^{(0)}_n,t_n\right)\phi^{(2)}_n+\partial_{zz} f\left(\phi^{(0)}_n,t_n\right)\left[\phi^{(1)}_n\right]^2
+ 2\partial_{zt} f\left(\phi^{(0)}_n,t_n\right)\phi^{(1)}_n+\partial_{tt} f\left(\phi^{(0)}_n,t_n\right)\\
\equiv\ddf{2}{\phi_n^{(0)},\phi_n^{(1)},\phi_n^{(2)},t_n}.
\end{multline*}
The physical equations associated with any derivative of order four or more are obtained in a similar way.

\begin{rem}
Each physical equation $\PE{k}(n)$ only involves numerical approximations of the solution and its derivatives at the same time grid point, $t_n$. However, to connect the values of the solution and its derivatives at different nearby time points, we introduce a second set of equations. These are called structural equations and will be defined next.
\hfill\smallBB
\end{rem}

\subsection{Structural Equations}

Given $K$ as the maximum derivative order, we construct $K$ physical equations, denoted $\PE{1}$ through $\PE{K}$. Each equation involves the numerical values $\phi^{(k)}_n$, $k\bint{0,K}$, at each time $t_n$. Once these values at $t_n$ are available, the next task is to compute the solution and its derivatives over a group of $R$ consecutive time steps, called a block. This means assembling a $(K+1)\times R$ matrix containing the values of the solution and its derivatives at each time point within this block, that is
\[
\Phi_n=\left(\phi^{(k)}_{n+r}\right)_{k,r},\quad k\bint{0,K},\,r\bint{1,R},
\]
as depicted in Fig.~\ref{fig:grid_KR}.

By evaluating the $K$ physical equations at each of the $R$ time points $t_{n+r}$, $r\bint{1,R}$, we obtain $R\times K$ relations. To uniquely determine all variables, we need $R$ additional equations, which are drawn from the structure of the uniform time grid used in the discretisation.

\begin{figure}
\centering
\includegraphics{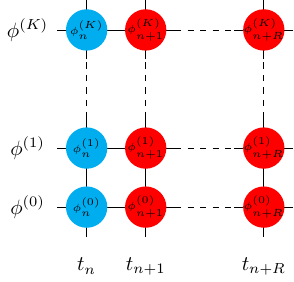}
\caption{Grid of points for a structural scheme. The given values at time $t_{n}$ are represented is blue, while the unknown values that constitute the block are marked in red.}
\label{fig:grid_KR}
\end{figure}

\subsubsection{General design}

Structural equations are linear relations involving the approximations $\phi^{(k)}_{n+r}$, where $k$ denotes the derivative order for $k\bint{0,K}$ and $r$ the time step for $r\bint{0,R}$. More specifically, structural equation takes the form
\[
\sum_{k=0}^K\sum_{r=0}^R a_{k,r}\,\Delta t^k\,\phi^{(k)}_{n+r}=0,
\]
where $\ba=\left(a_{k,r}\right)_{k,r}$, $k\bint{0,K}$, $r\bint{0,R}$, is the $(K+1)\times(R+1)$ matrix of coefficients, and $\Delta t^k\equiv(\Delta t)^k$ denotes powers of the time increment.


In order to determine the coefficients $a_{k,r}$ of the structural equation, we introduce the functional
\[
E(\ba;\pi)=\sum_{k=0}^K\sum_{r=0}^R a_{k,r}\,\Delta t^k\,\pi^{(k)}(t_{n+r})
\]
for any function $\pi=\pi(t)$, regular enough.

We represent the matrix $\ba$ as a vector by mapping pairs $(k,r)$ to a single index $(k,r)\to\ell=1+k(R+1)+r$, from $\llbracket 0,K\rrbracket\times\llbracket 0,R\rrbracket$ onto $\llbracket 1,M\rrbracket$, $M=(K+1)(R+1)$. For simplicity, we still denote this vector as $\ba$. Using this mapping, the above functional can be rewritten as
\[
E(\ba;\pi)=\sum_{k=0}^K\sum_{r=0}^R a_{\ell}\,\Delta t^k\,\pi^{(k)}(t_{n+r}),\quad\ell=\ell(k,r).
\]

To construct high-order accurate schemes, we require the functional to hold exactly for all polynomials up to a certain degree. For this purpose, consider the following polynomial functions
\[
\pi_m(t;t_n)=\left(\frac{t-t_n}{\Delta t}\right)^{m-1},\quad m\bint{1,M}.
\]
Setting $E(\ba;\pi_m)=0$ for all $m\bint{1,M}$, a $M\times M$ linear system $\mathcal{M}\ba=0$ is obtained with entries
\[
\mathcal{M}[m,\ell]=\pi^{(k)}_m(t_{n+r};t_n)=C^m_k r^{m-1-k},\quad\ell=\ell(k,r),
\]
with
\[
C^m_k=
\begin{cases}
\frac{(m-1)!}{(m-1-k)!} & \text{if $m > k$},\\
0 & \text{otherwise},
\end{cases}
\]
for $k\bint{0,K}$ and $r\bint{0,R}$. Note that the matrix $\mathcal{M}$ does not depend on the mesh parameter $\Delta t$ or on the initial time $t_n$. This independence results from the use of a uniform mesh and the definition of the functions $\pi_m(\cdot;t_n)$.

\begin{rem}
Illustrating for $K=2$ and generic $R$, the matrix $\mathcal{M}$ is given as
\begin{scriptsize}
\[
{\mathcal{M}}=
\left(
\begin{array}{cccc|cccc|cccc}
1 & 1 & \cdots & 1 & 0 & 0 & \cdots & 0 & 0 & 0 & \cdots & 0\\
0 & 1 & \cdots & R^{1} & 1 & 1 & \cdots & 1 & 0 & 0 & \cdots &0\\
0 & 1 & \cdots & R^2 & 0 & 2 & \cdots & 2\times R^{1} & 2 & 2 &\cdots& 2\\
0 & 1 & \cdots & R^3 & 0 & 3 & \cdots & 3\times R^2 & 0 & 3\times2 &\cdots &3\times 2\times R^{1}\\
\vdots & \vdots & \ddots & \vdots & \vdots & \vdots & \ddots & \vdots & \vdots & \vdots & \ddots & \vdots\\
0 & 1 & \cdots & R^{M-1} & 0 & (M-1) & \cdots & (M-1)\times R^{M-2}& 0 & (M-1)(M-2) & \cdots & (M-1)(M-2)R^{M-3}\\
\end{array}
\right).
\]
\end{scriptsize}
\hfill\smallBB
\end{rem}

We now analyse the properties of the linear system $\mathcal{M}\ba=0$.

\begin{defn}
A set of functions $\left(\nu_m(t)\right)_{m=1,\ldots,M}$ is resolvent of $E$ if, for any $m\bint{1,M}$, we have that, if $E(\ba;\nu_{m})=0$, then $\ba=0$.
\end{defn}

\begin{prop}
The set of $M$ polynomial functions of degree less than or equal to $M-1$
\[
\tau_{k,r}(t):=\tau_{\ell(k,r)}(t)=
\left(\frac{t-t_{n+r}}{\Delta t}\right)^k
\prod_{\substack{s\bint{0,R}\\s\neq r}}
\left(\frac{t-t_{n+s}}{\Delta t}\right)^{K+1},\quad k\bint{0,K},\,r\bint{0,R},
\]
is resolvent.
\end{prop}

\begin{proof}
First, observe that the degree of all these polynomial functions is strictly less than $M$. Second, these polynomials are linearly independent and therefore form a basis for $\mathbb{P}_{M-1}$.
Note that for any $k'\leqslant K$ and $r'\neq r$, we have $\tau_{k,r}^{(k')}(t_{n+r'})=0$, then
\begin{align*}
E(\ba;\tau_{k,r})
&=\sum_{k'=0}^K\sum_{r'=0}^R a_{k',r'}\,\Delta t^{k'}\,\tau_{k,r}^{(k')}(t_{n+r'})
=\sum_{k'=0}^K a_{k',r}\,\Delta t^{k'}\,\tau_{k,r}^{(k')}(t_{n+r})\\
&=\sum_{k'=k}^K a_{k',r}\,\Delta t^{k'}\,\tau_{k,r}^{(k')}(t_{n+r})=\sum_{k'=k}^K a_{k',r}\,\Delta t^{k'-k}\,b_{k'-k,r},
\end{align*}
where for $k'\geqslant k$ we have
\[
b_{k'-k,r}=\left[\prod_{s\neq r}\left(\frac{t-t_{n+s}}{\Delta t}\right)^{K+1}\right]^{(k'-k)}_{\boldsymbol{\big |t=t_{n+r}}}.
\]
Notice that for $k=k'$, we obtain
\[
b_{0,r}=\left[\prod_{s\neq r}\left(\frac{t-t_{n+s}}{\Delta t}\right)^{K+1}\right]_{\boldsymbol{\big |t=t_{n+r}}}
=\prod_{s\neq r}\left(r-s\right)^{K+1}\neq 0.
\]
The rest of the proof is achieved by induction. Let $r\bint{0,R}$. For $k=K$, we just have
\[
0=E(\ba;\tau_{K,r})=a_{K,r} b_{0,r};
\]
hence $a_{K,r}=0$. Now, take $k=K-1$, then we have
\[
0=E(\ba;\tau_{{K-1},r})=a_{K,r}\,\Delta t\,b_{1,r}+a_{K-1,r} b_{0,r}=a_{K-1,r} b_{0,r}.
\]
Therefore $a_{K-1,r}=0$. Similarly, we successively prove that $a_{K-2,r}=0$, $a_{K-3,r}=0$, $\ldots$, $a_{0,r}=0$. In conclusion, all the coefficients of $\ba$ are null, so the set of polynomial functions is resolvent.
\end{proof}

\begin{cor}
The matrix $\mathcal{M}$ is non-singular.
\end{cor}

\begin{proof}
The sets of polynomial functions $\tau_{k,r}$ and $\pi_m$ both define a basis for $\mathbb{P}_{M-1}$. Since the set $\tau_{k,r}$ is resolvent, the set of functions $\pi_m$ is also resolvent. Thus, the matrix $\mathcal{M}$ is non-singular.
\end{proof}

\begin{defn}
Let $S$ be the number of structural equations we intend to produce. We denote by $\mathcal{K}^{S}\subset\mathbb{R}^M$ the kernel of the rectangular $(M-S)\times M$ matrix $\mathcal{M}[1:M-S,:]$, with the last $S$ rows eliminated.
\end{defn}

Since $\mathcal{M}$ is non-singular, the matrix $\mathcal{M}[1:M-S,:]$ enjoys the maximum rank property, therefore the dimension of $\mathcal{K}^{S}$ is $S$ and contains a basis of $S$ linear independent vectors of $\mathbb{R}^M$ we shall denote $\ba^s$, $s\bint{1,S}$.
\begin{defn}
Given $K$ the number of derivatives, $R$ the size of the block, and $S$ the total number of structural equations, we denote by $\SEE{K}{R}{S}(s)$, $s\bint{1,S}$, the structural equation number $s$ corresponding to the vector $\ba^s\in\mathcal{K}^{S}$.
\end{defn}
\begin{defn}
We denote by $\boldsymbol{A}$ the $S\times M$ matrix composed of the line vectors $(\ba^s)^\trans$, $s\bint{1,S}$.
\end{defn}

\subsubsection{Choice of the structural equations}

The first observation is that, although the kernel $\mathcal{K}^{S}$ is unique as a vectorial subspace, we do not have the uniqueness of the basis, that is, the uniqueness of the set of $S$ structural equations. We list several methods to produce the $S$ basis vectors.

\paragraph{I. Analytical method.}
Analytical expressions of the coefficients of the system $E(\ba;\pi_{m})=0$, $m=\bint{1,M-S}$, can be derived by symbolic computation assuming a uniform time subdivision. Such relations avoid the computation of matrix $\mathcal{M}$ and are mandatory for stability analysis. In~\ref{analytical_expression} we report several sets of structural equations $\SEE{K}{R}{S}(s)$ for useful values of $K$ and $R$.

\begin{rem}
The analytical expressions of vectors $(\ba^s)_{s=1}^S$ given in~\ref{analytical_expression} do not necessarily correspond to an orthogonal basis of $\mathcal{K}^{S}$, that is, we do not necessarily have $\langle\ba^s,\ba^m\rangle=\delta_{sm}$.
\hfill\smallBB
\end{rem}

\paragraph{II. Numerical method.} For larger values of $K$ and $R$ or when dealing with a non-uniform grid, we compute an orthogonal basis of the kernel $\mathcal{K}^{S}$ using the complete SVD decomposition
\[
\mathcal{M}[1:M-S,1:M]=U\Sigma V^\trans
\]
with $\boldsymbol{U}$ a rectangular $(M-S)\times M$ matrix of orthogonal line vectors, $\boldsymbol{\Sigma}$ the diagonal $M\times M$ matrix, and $\boldsymbol{V}^\trans$ the square orthogonal $M\times M$ matrix. Since the $S$ last diagonal entries of matrix $\boldsymbol{\Sigma}$ are null, we deduce that the $S$ last column vectors of matrix $\boldsymbol{V}$ define an orthogonal basis of $\mathcal{K}^S$.

\paragraph{III. Hierarchical basis.} We want the first vector $\ba^1$ to provide the largest order of consistency error. The second vector $\ba^2$ reduces the consistency order of one and so forth. To provide such a property, we propose an algorithm to set a numerical hierarchical basis of the kernel $\mathcal{K}^{S}$.

\begin{enumerate}
\item Firstly compute the sole vector $\ba^1$ of the kernel $\mathcal{K}^{1}$.
\item Then, determine a basis of the kernel $\mathcal{K}^{2}$. Noting that $\ba^1$ automatically belongs to $\mathcal{K}^{2}$, we determine a normalised vector $\ba^2\in\mathcal{K}^{2}$ orthogonal to $\mathcal{K}^{1}$ using an orthogonalisation procedure. Moreover, the consistency error associated with $\ba^2$ is one order lower than $\ba^1$ as the number of constraints $E(\ba,\pi_m)=0$ has been reduced.
\item For a given $S'<S$, assume that the set of vectors $\displaystyle\left(\ba^s\right)_{s=1}^{s=S'}\in\mathcal{K}^{S'}$ is orthogonal. Then, determine the kernel $\mathcal{K}^{S'+1}$ and extract a new orthogonal vector $\ba^{S'+1}\in\mathcal{K}^{S'+1}$ orthogonal to $\mathcal{K}^{S'}$. By construction, $\ba^{S'+1}$ provides a reduced consistency of one order with respect to $\ba^{S'}$.
\item The construction ends when $S'=S$ is reached and $\{\ba^1,\ldots,\ba^S\}$ is a hierarchical basis for $\mathcal{K}^{S}$.
\end{enumerate}

\subsubsection{Equivalent set of structural equations}
\label{sec.EquStrEq}

As we pointed out in the previous section, one can determine different representations/bases of the kernel $\mathcal{K}^{S}$, that is, different sets of structural equations. This question is important because of the impact of the basis choice. Indeed, we can consider two different bases $\big(\ba^s)_{s=1}^{S}$ and $\big(\widetilde\ba^s)_{s=1}^{S}$ of $\mathcal{K}^{S}$ from which we build two different sets of structural equations. Consequently, there exists a change of basis given by a non-singular $S\times S$ matrix $\boldsymbol{\Xi}=[\xi_{ij}]_{S\times S}$ such as $\widetilde{\boldsymbol{A}}=\boldsymbol{\Xi}\boldsymbol{A}$, where
\begin{equation}\label{eq_basis}
\widetilde{\boldsymbol{A}}=\left[\big(\widetilde{\ba}^s\big)^\trans\right]
\quad\text{and}\quad
\boldsymbol{A}=\left[\big(\ba^s\big)^\trans\right]
\end{equation}
are two $S\times M$ matrices.
Conversely, any non-singular matrix $\boldsymbol{\Xi}$ transforms a basis of $\mathcal{K}^{S}$ into another basis of $\mathcal{K}^{S}$.

\begin{rem}
By construction, any set of structural equations provides the same global consistency error order $M-S$ since the equations are exact for any $\pi_m$, $m\bint{1,M-S}$.
\hfill\smallBB
\end{rem}

\subsection{Design of the schemes}
\label{design_schemes}

To build a complete numerical scheme to solve a first-order ODE, the set of physical and structural equations to be used must be determined. Several (meta)parameters control the scheme's complexity and accuracy, particularly the number of physical equations per grid point ($K$) and the number of grid points per block ($R$).

\subsubsection{Notations}

Given $K$ and $R$, the data are arranged into two vectors.
\begin{itemize}
\item The approximations regarding the same derivation order $k$ are gathered into vectors
\[
\Phi_n^{(k)}=\Big(\phi^{(k)}_{n+1},\ldots,\phi^{(k)}_{n+R}\Big)^\trans\in\mathbb{R}^R
\]
while $T_n=(t_{n+1},\ldots,t_{n+R})^\trans$ stands for the time-step vector.
\item The approximations for the same time step $t_{n+r}$ are gathered into vectors
\[
\Psi_{n+r}=\big(\phi^{(0)}_{n+r},\ldots,\phi^{(K)}_{n+r}\Big)^\trans\in\mathbb{R}^{K+1}.
\]
\end{itemize}
The generic formulation then is, given the initial vector $\Psi_{n}$ at time $t_n$, we seek the matrix $\displaystyle\Phi_{n}=\left(\phi^{(k)}_{n+r}\right)_{k,r}=\big [\Psi_{n+1},\ldots,\Psi_{n+R}\big ]\in\mathbb{R}^{(K+1)\times R}$, solution of the following system:
\begin{itemize}
\item $K\times R$ physical equations $\PE{k}(r)$, $k\bint{1,K}$, $r\bint{1,R}$, and
\item $R$ structural equations $\SEE{K}{R}{R}(s)$, $s\bint{1,R}$.
\end{itemize}

\begin{rem}
Note that, since we want to consider the maximum number of physical equations, we need $R$ structural equations.
\hfill\smallBB
\end{rem}

\begin{defn}
Functions $\ddf{k-1}{z_0,z_1,\ldots,z_{k-1},t}$, defined for real value entries, are extended to their vectorial form
\[
\Phi^{(k)}_{n}=\ddf{k-1}{\Phi^{(0)}_{n},\Phi^{(1)}_{n},\ldots,\Phi^{(k-1)}_{n},T_n}
\]
by applying the function component-wise, that is
\[
\phi^{(k)}_{n+r}=\ddf{k-1}{\phi^{(0)}_{n+r},\phi^{(1)}_{n+r},\ldots,\phi^{(k-1)}_{n+r},t_{n+r}},\quad r\bint{1,R}.
\]
\end{defn}

\subsubsection{Structural schemes}
\label{Structural_schemes}

\paragraph{Structural scheme $\SK{1}{R}$.}
\label{sk1R}

The structural scheme $\SK{1}{R}$ involves $2R$ unknowns per block, namely $\phi_{n+r}^{(0)}$ and $\phi_{n+r}^{(1)}$, $r\bint{1, R}$, given $\phi_{n}^{(0)}$ and $\phi_{n}^{(1)}$.
Therefore, it requires $R$ physical equations and $R$ structural equations. The scheme then reads
\[
\left\{
\begin{array}{l}
\left\{\begin{array}{l}
\phi^{(1)}_{n+r}=\ddf{0}{\phi^{(0)}_{n+r},t_{n+r}},\quad r\bint{1,R},
\end{array}\right.\\
\left\{\begin{array}{l}\displaystyle
\sum_{k=0}^1\sum_{r=0}^Ra_{k,r}^s\Delta t^k\phi^{(k)}_{n+r}=0,\quad s\bint{1,R}.
\end{array}\right.
\end{array}
\right.
\]
The analytical expressions of the coefficients of the structural equations are given in~\ref{analytical_expression}.
We rewrite the system in a more compact way
\[
\left\{
\begin{array}{l}
\Phi_{n}^{(1)}=\ddf{0}{\Phi_{n}^{(0)},T_n},\\
\boldsymbol{A}^{(0)}\Phi^{(0)}_{n}+\Delta t\boldsymbol{A}^{(1)}\Phi^{(1)}_{n}+\widehat{\boldsymbol{A}}_n\Psi_{n}=0,
\end{array}
\right.
\]
where $\ddf{0}{\Phi_{n}^{(0)},T_n}$ applies the first physical equation all over the block, while the structural equations provide a linear system with
\[
\boldsymbol{A}^{(0)}=\Big(a_{0,r}^s\Big)_{r,s=1}^R\in\mathbb{R}^{R\times R},\,
\boldsymbol{A}^{(1)}=\Big(a_{1,r}^s\Big)_{r,s=1}^R\in\mathbb{R}^{R\times R},
\]
and
\[
\widehat{\boldsymbol{A}}_n=\Big(a_{k,0}^s\,\Delta t^k\,\Big)_{s=1,\ldots,R\atop k=0,1}\in\mathbb{R}^{R\times 2},\quad\Psi_{n}=\big(\phi_{n}^{(0)},\phi_{n}^{(1)}\big)^\trans.
\]

\paragraph{Structural scheme $\SK{2}{R}$.}
\label{sk2R}
From the first and second physical equations, the $\SK{2}{R}$ scheme reads
\[
\left\{
\begin{array}{l}
\left\{\begin{array}{l}
\phi^{(1)}_{n+r}=\ddf{0}{\phi^{(0)}_{n+r},t_{n+r}},\quad r\bint{1,R},\\
\phi^{(2)}_{n+r}=\ddf{1}{\phi^{(0)}_{n+r},\phi^{(1)}_{n+r},t_{n+r}},\quad r\bint{1,R},
\end{array}\right.\\
\left\{\begin{array}{l}\displaystyle
\sum_{k=0}^2\sum_{r=0}^R a_{k,r}^s\,\Delta t^k\,\phi^{(k)}_{n+r}=0,\quad s\bint{1,R}.
\end{array}\right.
\end{array}
\right.
\]
We rewrite the equations in a more compact way
\[
\left\{
\begin{array}{l}
\Phi_{n}^{(1)}=\ddf{0}{\Phi_{n}^{(0)},T_n},\\
\Phi_{n}^{(2)}=\ddf{1}{\Phi_{n}^{(0)},\Phi_{n}^{(1)},T_n},\\
\boldsymbol{A}^{(0)}\Phi^{(0)}_{n}+\Delta t\boldsymbol{A}^{(1)}\Phi^{(1)}_{n}+\Delta t^2\boldsymbol{A}^{(2)}\Phi^{(2)}_{n}+\widehat{\boldsymbol{A}}_n\Psi_{n}=0.
\end{array}
\right.
\]

\begin{rem}
The structural method is easily extended to higher derivatives. For example, scheme $\SK{R}{3}$, using the compact notation, reads
\[
\left\{
\begin{array}{l}
\Phi_{n}^{(1)}=\ddf{0}{\Phi_{n}^{(0)},T_n},\\
\Phi_{n}^{(2)}=\ddf{1}{\Phi_{n}^{(0)},\Phi_{n}^{(1)},T_n},\\
\Phi_{n}^{(3)}=\ddf{2}{\Phi_{n}^{(0)},\Phi_{n}^{(1)},\Phi_{n}^{(2)},T_n,},\\
\boldsymbol{A}^{(0)}\Phi^{(0)}_{n}+\Delta t\boldsymbol{A}^{(1)}\Phi^{(1)}_{n}+\Delta t^2\boldsymbol{A}^{(2)}\Phi^{(2)}_{n}+\Delta t^3\boldsymbol{A}^{(3)}\Phi^{(3)}_{n}+\widehat{\boldsymbol{A}}_n\Psi_{n}.
\end{array}
\right.
\]
Note that the matrices $\boldsymbol{A}^{(0)},\boldsymbol{A}^{(1)},\ldots,\boldsymbol{A}^{(K)}$ and $\widehat{\boldsymbol{A}}$ have to be recomputed with the kernel method for each scheme.
\hfill\smallBB
\end{rem}

\section{Relation with other methods and novelty}
This section was initially motivated by a private communication from Professor Robert McLachlan concerning the interpretation of the structural scheme as an MDRK scheme when the maximum number of physical equations is used. We further extend the comparison to other traditional, well-established methods. Indeed, the $R$ structural equations, combined with the physical equations, correspond to particular cases of Multi-Derivative Runge--Kutta methods. We prove this equivalence for $K=1$ corresponding to the classical RK scheme, and for $K=2$ corresponding to the two-derivative RK scheme.).

\subsection{Structural method with $K=1$}
\subsubsection{An implicit RK reformulation}
Let $R$ be given, we just use the ODE as physical equations at point $t_{n+r}$, $r\bint{1,R}$, and the structural equations read
\[
\sum_{r=0}^Ra_{0,r}^s\phi^{(0)}_{n+r}+\Delta t\sum_{r=0}^Ra_{1,r}^s\phi^{(1)}_{n+r}=0,\quad s\bint{1,R},
\]
that are rewritten in the form
\[
\sum_{r=1}^Ra_{0,r}^s\phi^{(0)}_{n+r}+\Delta t\sum_{r=0}^Ra_{1,r}^s\phi^{(1)}_{n+r}=- a_{0,0}^s\phi^{(0)}_{n},\quad s\bint{1,R}.
\]
Noting that the constant function satisfies the structural equation, we deduce that
$\displaystyle-a_{0,0}^s=\sum_{r=1}^Ra_{0,r}^s$; hence,
\[
\sum_{r=1}^Ra_{0,r}^s\phi^{(0)}_{n+r}+\Delta t\sum_{r=0}^Ra_{1,r}^s\phi^{(1)}_{n+r}=\phi^{(0)}_{n}\sum_{r=1}^Ra_{0,r}^s,\quad s\bint{1,R}.
\]
Let $\boldsymbol{A}_0=(a_{0,r}^s)_{r,s}$ be the $R\times R$ matrix for the function approximations, $\boldsymbol{A}_1= (a_{1,r}^s)_{s,r}$ the $R\times (R+1)$ matrix for the derivative approximations, and $\Vone$ the vector of entries equal to $1$. Then, we have
\[
\boldsymbol{A}_0\Phi^{(0)}+\Delta t\boldsymbol{A}_1\Phi^{(1)}=\boldsymbol{A}_0\Vone\phi^{(0)}_{n}.
\]
Assuming that $\boldsymbol{A}_0$ is non-singular, then
\[
\Phi^{(0)}+\Delta t\boldsymbol{A}_0^{-1}\boldsymbol{A}_1\Phi^{(1)}=\Vone\phi^{(0)}_{n}
\]
and the system becomes
\[
\Phi^{(0)}=\Vone\phi^{(0)}_{n}-\Delta t\boldsymbol{A}_0^{-1}\boldsymbol{A}_1\Phi^{(1)},
\]
which corresponds to the $(R+1)$-stage RK scheme, given as
\begin{align*}
\phi^{(0)}_{n}&=\phi^{(0)}_{n}+\Delta t\sum_{r=0}^R 0f\big(\phi^{(0)}_{n+r},t_{n+r}\big),\\
\phi^{(0)}_{n+s}&=\phi^{(0)}_{n}+\Delta t\sum_{r=0}^R\alpha_{s,r}f\big(\phi^{(0)}_{n+r},t_{n+r}\big),\quad s\bint{1,R},
\end{align*}
where the coefficients $\alpha_{s,r}$ are the entries of the $R\times (R+1)$ matrix $\boldsymbol{A}_0^{-1}\boldsymbol{A}_1$ and intermediate uniform sub-steps~\cite{MK2024} over the large time-step $\Delta T=R\Delta t$.
Using the RK formalism, we have $c_s=s/R$, $s\bint{0,R}$, and
\begin{gather*}
Y_s=\phi^{(0)}_{n}+\Delta T\sum_{r=0}^Ra_{s,r}f\big(Y_{r},t_{n}+c_r\Delta T\big),\\
\phi^{(0)}_{n+R}=\phi^{(0)}_{n}+\Delta T\sum_{r=0}^R b_{r}f\big(Y_{r},t_{n}+c_r\Delta T\big),
\end{gather*}
with $a_{s,r}=\alpha_{s,r}/R$ and $b_{r}=\alpha_{Rr}/R$, $s\bint{1,R}$, $r\bint{0,R}$, and $a_{0r}=0$.
\begin{rem}
An important property of our construction is that the intermediate steps yield $(R+~\!\!1)$~th-order accurate approximations. Consequently, the intermediate approximations are no longer considered auxiliary calculus but are effective approximations.
\hfill\smallBB
\end{rem}

\begin{example}
For $R=2$, the two structural equations read
\begin{gather*}
0=\frac{\phi^{(0)}_{n+2}-\phi^{(0)}_{n}}{2}-\Delta t\frac{\phi^{(1)}_{n+2}+4\phi^{(1)}_{n+1}+\phi^{(1)}_{n}}{6},\\
0=\phi^{(0)}_{n+2}-2\phi^{(0)}_{n+1}+\phi^{(0)}_{n}-\Delta t\frac{\phi^{(1)}_{n+2}-\phi^{(1)}_{n}}{2},
\end{gather*}
which can be rewritten with $\Delta T=2\Delta t$ as
\begin{gather*}
\phi^{(0)}_{n+2}=\phi^{(0)}_{n}+\Delta T\frac{\phi^{(1)}_{n}+4\phi^{(1)}_{n+1}+\phi^{(1)}_{n+2}}{6},\\
\phi^{(0)}_{n+1}=\phi^{(0)}_{n}+\Delta T\frac{5\phi^{(1)}_{n}+8\phi^{(1)}_{n+1}-\phi^{(1)}_{n+2}}{24},
\end{gather*}
which corresponds to the fourth-order Lobatto IIIA method, given by the Butcher tableau
\[
\begin{tblr}{c|rrr}
0 & 0 & 0 & 0\\
\frac{1}{2} & \frac{5}{24} & \frac{8}{24} & -\frac{1}{24}\\
1 & \frac{1}{6} & \frac{4}{6} & \frac{1}{6}\\\hline
& \frac{1}{6} & \frac{4}{6} & \frac{1}{6}
\end{tblr}
\]
\end{example}

\ra{
\subsubsection{Collocation method}
One advantage of the structural scheme approach is that we provide the Butcher tableau without using the Algebraic Order Conditions~\cite{Butcher2008}, Sec. 23. Indeed, the number of conditions increases with the order, and the relations turn more non-linear by involving a larger number of multiplications and summations~\cite{MK2024} for $p=5,6$. The RK scheme derived from the structural equations does not require any complex analytic calculation of the coefficients, since they are determined by the kernel basis. Consequently, the structural method easily provides an RK scheme of arbitrary order.
\\
The collocation method \cite{GS69,W70,CAP11} enjoys the same property and enables to easily design a subclass of high-order implicit RK method with simple linear conditions (the collocation point being given \cite{CAP11}), to provide the expected order. In short, the method consists in seeking a polynomial function $\pi$ of degree $m$ such that
\begin{eqnarray*}
\pi(t_n)&=&\phi_n\\
\pi'(t_n+c_i\Delta t)&=&f\big(\pi'(t_n+c_i\Delta t),t_n+c_i\Delta t\big ), i=1,\ldots,m,
\end{eqnarray*}
with $c_i$ the collocation parameters. We then take $\phi_{n+1}=\pi(t_n+\Delta t)$ as an approximation at time $t_{n+1}$. Comparing with the structural method, we do not introduce intermediate point between $t_n$ and $t_{n+1}$.
Nevertheless, we can artificially built similar conditions of a collocation method by taking $\Delta T=R\Delta t$ as a super time step and $t_{n+r}=t_n+c_r\Delta R$ with $c_r=\frac{r}{R}$ as collocation coefficients uniformly distributed. The point is that the structural method  uses the information $\phi_n^1=f(\phi_n^0,t_n)$ and  $\phi_{n+1}^1=f(\phi_{n+1}^0,t_{n+1})$ in the approximation whereas the collocation method does not. Consequently, the equivalent Butcher tableau of the structural method will be different of the one associated to the collocation method.
\begin{rem}
We have presented the $\SK{1}{R}$ with a uniform grid for the sake of simplicity to provide a $(R+1)$th-order approximations, but it is perfectly possible to consider a block of $R$ times with a non-uniform distribution (Gauss Lobatto collocation points for instance). Questions then arise about possible improvement of accuracy and the stability of the method.
\end{rem}
At last, the structural method enjoys the A-stability property whereas the stability of the Collocation method (A- or L-stability) depends on the choice of the coefficients $c_i$ (Legendre, Gauss, Radeau, Lobatto) \cite{CAP11}.
\\
The multistep collocation method extension \cite{LN89,L90} used additional log-data calculated from the past to improve the accuracy. The structural method only considers the approximations after the last time step $t_n$, therefore the two methods clearly differ.
}

\subsection{Structural equations with $K=2$}
\subsubsection{A Two-Derivative RK (TDRK) reformulation}
We proceed with the second derivative ($K=2$) structural method, for which the structural equations read
\[
\sum_{r=0}^Ra_{0,r}^s\phi^{(0)}_{n+r}+\Delta t\sum_{r=0}^Ra_{1,r}^s\phi^{(1)}_{n+r}+
\Delta t^2\sum_{r=0}^Ra_{2,r}^s\phi^{(2)}_{n+r}=0,\quad s\bint{1,R},
\]
that are rewritten in the form
\[
\sum_{r=1}^Ra_{0,r}^s\phi^{(0)}_{n+r}+
\Delta t\sum_{r=0}^Ra_{1,r}^s\phi^{(1)}_{n+r}+
\Delta t^2\sum_{r=0}^Ra_{2,r}^s\phi^{(2)}_{n+r}
=- a_{0,0}^s\phi^{(0)}_{n},\quad s\bint{1,R}.
\]
Once again, the constant function satisfies the structural equation, then
\[
\sum_{r=1}^Ra_{0,r}^s\phi^{(0)}_{n+r}+
\Delta t\sum_{r=0}^Ra_{1,r}^s\phi^{(1)}_{n+r}+
\Delta t^2\sum_{r=0}^Ra_{2,r}^s\phi^{(2)}_{n+r}=\phi^{(0)}_{n}\sum_{r=1}^Ra_{0,r}^s,\quad s\bint{1,R}.
\]
Assuming that $\boldsymbol{A}_0=(a_{0,r}^s)_{r,s}$ is a non-singular $R\times R$ matrix and defining the $R\times (R+1)$ matrices $\boldsymbol{A}_1= (a_{1,r}^s)_{s,r}$ and $\boldsymbol{A}_2=(a_{2,r}^s)_{s,r}$, the structural equations are reformulated in the compact form as
\[
\Phi^{(0)}=\Vone\phi^{(0)}_{n}-\Delta t\,\boldsymbol{A}_0^{-1}\boldsymbol{A}_1\Phi^{(1)}-\Delta t^2\,\boldsymbol{A}_0^{-1}\boldsymbol{A}_2\Phi^{(2)}.
\]
Denote $\alpha^{1}_{s,r}$ and $\alpha^{2}_{s,r}$ the entries of the $R\times (R+1)$ matrices $-\boldsymbol{A}_0^{-1}\boldsymbol{A}_1$ and $-\boldsymbol{A}_0^{-1}\boldsymbol{A}_2$, respectively. The scheme is written in the form
\begin{gather*}
\phi^{(0)}_{n+s}=\phi^{(0)}_{n}+\Delta t\sum_{r=0}^R\alpha^{1}_{s,r}\,\phi^{(1)}_{n+r}
+\Delta t^2\sum_{r=0}^R\alpha^{2}_{s,r}\,\phi^{(2)}_{n+r},\\
\phi^{(1)}_{n+s}=f\big(\phi^{(0)}_{n+s},t_{n+s}\big),\\
\phi^{(2)}_{n+s}=\partial_zf\big(\phi^{(0)}_{n+s},t_{n+s}\big)\phi^{(1)}_{n+s}+\partial_tf\big(\phi^{(0)}_{n+s},t_{n+s}\big),\\
\end{gather*}
with $s\bint{1,R}$.
Reshaping the scheme in the RK framework yields
\begin{gather*}
Y^{0}_s=\phi^{(0)}_{n}+\Delta T\sum_{r=0}^R a_{s,r}^1 Y^{1}_r+\Delta T^2\sum_{r=0}^R a_{s,r}^2 Y^{2}_r,\\
Y^{1}_s=f\big(Y^{0}_{s},t_{n}+c_s\Delta T\big),\\
Y^{2}_s=\partial_zf\big(Y^{0}_{s},t_{n}+c_s\Delta T\big)Y^{1}_{s}+\partial_tf\big(Y^{0}_{s},t_{n}+c_s\Delta T\big),\\
\phi^{(0)}_{n+R}=\phi^{(0)}_{n}+\Delta T\sum_{r=0}^R b_{r}^1 Y^{1}_r+\Delta T^2\sum_{r=0}^R b_{r}^2 Y^{2}_r,
\end{gather*}
with $a^o_{s,r}=\alpha^o_{s,r}/R$ and $b^o_{r}=\alpha^o_{R,r}/R$, $s\bint{1,R}$, $r\bint{0,R}$, and $a^o_{0,r}=0$, $o=1,2$, which corresponds to the TDRK scheme with $R+1$ stages~\cite{TCW14}.

\begin{example}
For $R=2$, the two structural equations read
\begin{align*}
0=&
+24\phi^{(0)}_{n}+9\Delta t\phi^{(1)}_{n}+\Delta t^2\phi^{(2)}_{n}
-48\phi^{(0)}_{n+1}+0\Delta t\phi^{(1)}_{n+1}-8\Delta t^2\phi^{(2)}_{n+1}\\
&+24\phi^{(0)}_{n+2}-9\Delta t\phi^{(1)}_{n+2}+\Delta t^2\phi^{(2)}_{n+2},\\
0=&
-15\phi^{(0)}_{n}-7\Delta t\phi^{(1)}_{n}-\Delta t^2\phi^{(2)}_{n}
+0\phi^{(0)}_{n+1}-16\Delta t\phi^{(1)}_{n+1}+0\Delta t^2\phi^{(2)}_{n+1}\\
&+15\phi^{(0)}_{n+2}-7\Delta t\phi^{(1)}_{n+2}+\Delta t^2\phi^{(2)}_{n+2}=0,
\end{align*}
which can be rewritten with $\Delta T=2\Delta t$ as
\begin{gather*}
\phi^{(0)}_{n+2}=\phi^{(0)}_{n}+\Delta T\frac{7\phi^{(1)}_{n}+16\phi^{(1)}_{n+1}+7\phi^{(1)}_{n+2}}{30}
+\Delta T^2\frac{\phi^{(2)}_{n}-\phi^{(2)}_{n+2}}{60},\\
\phi^{(0)}_{n+1}=\phi^{(0)}_{n}+\Delta T\frac{303\phi^{(1)}_{n}+284\phi^{(1)}_{n+1}+33\phi^{(1)}_{n+2}}{1440}
+\Delta T^2\frac{39\phi^{(2)}_{n}-120\phi^{(2)}_{n+1}-3\phi^{(2)}_{n+2}}{2800},
\end{gather*}
given by the extended Butcher tableau
\[
\begin{tblr}{c|rrr|rrr}
0 &0 &0 & 0 &0 &0 & 0\\
\frac{1}{2}&\frac{303}{1440}&\frac{284}{1440}&\frac{33}{1440} &\frac{39}{2880}&-\frac{120}{2880}&-\frac{3}{2880}\\
1 &\frac{7}{30}&\frac{16}{30}&\frac{7}{30} &\frac{1}{60}&0&-\frac{1}{60}\\\hline
&\frac{7}{30}&\frac{16}{30}&\frac{7}{30} &\frac{1}{60}&0&-\frac{1}{60}
\end{tblr}
\]
\end{example}
\subsubsection{A non-TDRK scheme}

The structural equations provide an MDRK only when the maximum number of physical equations ($K$) is used. For example, taking $K=2$ and $R=1$, the following scheme is composed of one physical equation and two structural equations
\begin{gather*}
0=12(\phi^{(0)}_{n+1}-\phi^{(0)}_{n})-\Delta t\,6(\phi^{(1)}_{n+1}+\phi^{(1)}_{n})+\Delta t^2\,(\phi^{(2)}_{n+1}-\phi^{(2)}_{n}),\\
0=6(\phi^{(0)}_{n+1}-\phi^{(0)}_{n})-\Delta t\,(4\phi^{(1)}_{n+1}+2D_{n})+\Delta t^2\,\phi^{(2)}_{n+1},\\
\phi^{(1)}_{n+1}=f(\phi^{(0)}_{n+1},t_{n+1}),
\end{gather*}
has no equivalent formulation in the TDRK framework. The second physical equation has been replaced by a second structural equation; therefore, the scheme does not require the analytical expression of the derivative of $f(z,t)$.

The above scheme is third-order accurate for the function and derivative, whereas the approximation of the second derivative is second-order accurate. Moreover, stability is no longer unconditional. \ra{A restrictive condition on the time step $\Delta t$ is required} to guarantee the convergence of the approximation, as shown in reference~\cite{Clain2023}. In general, configurations involving fewer physical equations and more structural equations are possible. Nevertheless, a loss of the unconditional stability property, together with lower accuracy, leads us to disregard such schemes in the present work.

\ra{
\subsubsection{Multi-derivative collocation method}
Recently, a two-derivative collocation method~\cite{FH20} has been proposed. In fact, the authors also introduce the additional term $\phi_{n-1}$, which gives the method a two-step character. The main idea consists of introducing a Hermite interpolation polynomial $\pi$ such that
\begin{eqnarray*}
\pi(t_n)&=&\phi_n\\
\pi'(t_n+c_i\Delta t)&=&\ddf{0}{\pi(t_n+c_i\Delta t),t_n+c_i\Delta t}, i=1,\ldots,m,\\
\pi''(t_n+c_i\Delta t)&=&\ddf{0}{\pi(t_n+c_i\Delta t),\pi'(t_n+c_i\Delta t),t_n+c_i\Delta t}, i=1,\ldots,m,
\end{eqnarray*}
where $c_i$ are the collocation parameters. Once again, we take $\phi_{n+1}=\pi(t_n+\Delta t)$ as an approximation at time $t_{n+1}$, and the authors prove A-stability. To the best of our knowledge, and owing to the very limited literature on this technique, we do not know whether the method is equivalent to a two-derivative RK scheme.
}

\rb{
\subsection{Obrechkoff-type method}
Obrechkoff methods~\cite{O40} compute approximations of the function and its derivative at time $t_{n+1}$ using information from the function and its derivatives at time $t_n$. An extension was proposed in~\cite{U70,C78}, using two time steps and two derivatives; see also the Enright schemes~\cite{E74a,E74b}. The two-derivative multistep Obrechkoff-type method was studied in~\cite{KG82,II99,HA06}. More generally, multistep multi-derivative methods compute the next time step using backlog data from the previous $R$ nodes~\cite{GS64,RA97,NF03,S16}; see also~\cite{G64,SS13,AI19} for hybrid versions designed to enhance the stability of the method.
\\
Using our notations, the generalized Obrechkoff method reads
\[
\sum_{k=0}^K \alpha_k \phi_{n+1}^{(k)}=\sum_{r=0}^R \sum_{k=0}^K \beta_{kr} \phi_{n-r}^{(k)},
\]
with
\[
\phi_{n+1}^{(k)}=\ddf{k-1}{\Phi_{n+1}^{(0)},\cdots,\Phi_{n+1}^{(k-1)},t_{n+1}},\quad k=1,\ldots,K.
\]
The scheme is implicit with respect to the current time step $n+1$, but explicit in terms of the previous time steps. The improvement in order is achieved by exploiting the information computed at the $R$ earlier times. This class of schemes differs from the structural method proposed in the present work, since the resulting system is fully implicit with respect to all $R$ time steps, except for the time $t_n$.
\
We would like to emphasise that the structural equations technique can also be applied to determine the coefficients of Obrechkoff methods. Indeed, considering the functional
\[
E(\balpha\bbeta;\pi)=\sum_{k=0}^K \alpha_k \phi_{n+1}^{(k)}-\sum_{r=0}^R \sum_{k=0}^K \beta_{r,k} (\Delta t)^k\phi_{n-r}^{(k)},
\]
the conditions $E(\balpha\bbeta;\pi_\ell)=0$ with $\displaystyle \pi_\ell(t)=\left ( \frac{t-t_n}{\Delta t}\right )^\ell$, $\ell=0,\ldots,p$ provide the $p+1$ linear system
\[
A_1\balpha-\sum_{r=0}^{R} B_{r}\bbeta_{r}=0,\quad  \textrm{with } \bbeta_r=(\beta_{r,k})_{k=0}^{k=K}
\]
with $A_1$ a $p\times (K+1)$ matrix while $B_r$ are $p\times (K+1)$ matrices given by the polynomial derivative at $t_{n-r}$. If $p<(R+1)(K+1)$ and the system has maximal rank, then any vector in the kernel provides an Obrechkoff relation. Since there are only $K+1$ unknowns, namely the solution and its derivatives at time $t_{n+1}$, and $K$ physical equations at time $t_{n+1}$, a single vector from the kernel is sufficient to close the system.
\\
We note that our approach avoids the need to determine explicit relations. One only has to construct the matrices and extract a vector from the kernel. Moreover, the vector can be chosen so as to ensure the best stability properties.
}

\section{Numerical solver}
\label{sec::numerical_solver}

An efficient iterative algorithm is proposed in this section for solving the(non-)linear system of equations arising from the structural schemes $\SK{1}{R}$, $\SK{2}{R}$, and $\SK{3}{R}$.

\subsection{Algorithm for the structural schemes}

\paragraph{Structural scheme $\SK{1}{R}$.}
The compact matrix form of the structural scheme $\SK{1}{R}$ reads
\[
\left\{
\begin{array}{l}
\Phi_{n}^{(1)}=\ddf{0}{\Phi_{n}^{(0)},T_n},\\
\boldsymbol{A}^{(0)}\Phi^{(0)}_{n}+\Delta t\boldsymbol{A}^{(1)}\Phi^{(1)}_{n}+\widehat{\boldsymbol{A}}_n\Psi_{n}=0,
\end{array}
\right.
\]
for which a fixed-point iterative sequence of $2\times R$ matrices is introduced, given as
\[
\Phi_{n}[\ell]=\left(\phi^{(k)}_{n+r}[\ell]\right)_{k,r},k=0,1,\,r\bint{1,R},
\]
where $\ell$ is the iteration number. The algorithm is as follows.

\begin{itemize}
\item \textbf{Iterative procedure:} Assume that approximations of the function and derivative in vector $\Phi_{n}[\ell]$ at iteration $\ell$ are known. The iterative process then proceeds in two main stages:
\begin{enumerate}
\item[\ding{192}] Compute approximations of the derivative in vector $\displaystyle\Phi^{(1)}_{n}[\ell+1]$ using the first physical equation $\PE{1}(r)$ at time steps $t_{n+r}$, $r\bint{1,R}$, that is
\[
\Phi_{n}^{(1)}[\ell+1]=\ddf{0}{\Phi_{n}^{(0)}[\ell],T_n}.
\]
\item[\ding{193}] Compute approximations of the function in vector $\displaystyle\Phi^{(0)}_{n}[\ell+1]$ by solving a $R\times R$
linear system in the form
\[
\boldsymbol{A}^{(0)}\Phi^{(0)}_{n}[\ell+1]=-\Delta t\boldsymbol{A}^{(1)}\Phi^{(1)}_{n}[\ell+1]-\widehat{\boldsymbol{A}}_n\Psi_{n}
\]
Note that matrix $\boldsymbol{A}^{(0)}$ must be non-singular to determine vector $\Phi^{(0)}_{n}[\ell+1]$.
\end{enumerate}
The iterative procedure stops when the residual between successive approximations is below a specified tolerance.

\item \textbf{Initialisation procedure:} At the beginning of the iterative procedure, vector $\Psi_{n}$ is known but an approximation to vectors $\Psi_{n+r}[0]$, $r\bint{1,R}$, is desirable. To this end, a first-order accurate Taylor expansion can be employed, yielding
\[
\phi^{(0)}_{n+r}[0]=\phi^{(0)}_{n+r-1}[0]+\Delta t\phi^{(1)}_{n+r-1}[0],\quad
\phi^{(1)}_{n+r}[0]=\ddf{1}{\phi^{(0)}_{n+r}[0],t_{n+r}},\,r\bint{1,R},
\]
where $\phi^{(0)}_{n}[0]=\phi^{(0)}_{n}$ and $\phi^{(1)}_{n}[0]=\phi^{(1)}_{n}$.
\begin{rem}
A good initialisation is essential for reducing the number of iterations and ensuring that the initial point lies within the convergence basin of the fixed-point iteration.
\hfill\smallBB
\end{rem}
\end{itemize}

\paragraph{Structural scheme $\SK{2}{R}$.}
The procedure for the structural scheme $\SK{2}{R}$ is analogous, introducing a fixed-point iterative sequence of $3\times R$ matrices, given as
\[
\Phi_{n}[\ell]=\left(\phi^{(k)}_{n+r}[\ell]\right)_{k,r},\,k=0,1,2,\,r\bint{1,R}.
\]
The algorithm is as follows.

\begin{itemize}
\item \textbf{Iterative procedure:} Assume that approximations of the function and derivatives in vector $\Phi_{n}[\ell]$ at iteration $\ell$ are known. The iterative process then proceeds in two main stages:
\begin{enumerate}
\item[\ding{192}] Compute approximations of the derivatives in vectors $\Phi^{(1)}_{n}[\ell+1]$ and $\Phi^{(2)}_{n}[\ell+1]$ using the first and second physical equations, respectively, at time steps $t_{n+r}$, $r\bint{1,R}$, that is
\begin{gather*}
\Phi_{n}^{(1)}[\ell+1]=\ddf{0}{\Phi_{n}^{(0)}[\ell],T_n},\\
\Phi_{n}^{(2)}[\ell+1]=\ddf{1}{\Phi_{n}^{(0)}[\ell],\Phi_{n}^{(1)}[\ell+1],T_n}.
\end{gather*}
\item[\ding{193}] Compute approximations of the function in vector $\displaystyle\Phi^{(0)}_{n}[\ell+1]$ by solving a $R\times R$
linear system in the form
\[
\boldsymbol{A}^{(0)}\Phi^{(0)}_{n}[\ell+1]=-\Delta t\boldsymbol{A}^{(1)}\Phi^{(1)}_{n}[\ell+1]-\Delta t^2\boldsymbol{A}^{(2)}\Phi^{(2)}_{n}[\ell+1]-\widehat{\boldsymbol{A}}_n\Psi_{n}
\]
assuming that matrix $\boldsymbol{A}^{(0)}$ is non-singular.
\end{enumerate}

\item \textbf{Initialisation procedure:} An approximation to vectors $\Psi_{n+r}[0]$, $r\bint{1,R}$, is prescribed from a second-order accurate Taylor expansion, given as
\begin{gather*}
\phi^{(0)}_{n+r}[0]=\phi^{(0)}_{n+r-1}[0]+\Delta t\phi^{(1)}_{n+r-1}[0]+\frac{\Delta t^2}{2}\phi^{(2)}_{n+r-1}[0],\\
\phi^{(1)}_{n+r}[0]=\ddf{0}{\phi^{(0)}_{n+r}[0],t_{n+r}},\\
\phi^{(2)}_{n+r}[0]=\ddf{1}{\phi^{(0)}_{n+r}[0],\phi^{(1)}_{n+r}[0],t_{n+r}},
\end{gather*}
where $\phi^{(0)}_{n}[0]=\phi^{(0)}_{n}$, $\phi^{(1)}_{n}[0]=\phi^{(1)}_{n}$, $\phi^{(2)}_{n}[0]=\phi^{(2)}_{n}$.
\end{itemize}

\paragraph{Structural scheme $\SK{3}{R}$.}
A fixed-point iterative sequence of $4\times R$ matrices is introduced for the structural scheme $\SK{3}{R}$, given as
\[
\Phi_{n}[\ell]=\left(\phi^{(k)}_{n+r}[\ell]\right)_{k,r},\,k=0,\ldots,3,\,r\bint{1,R}.
\]
The algorithm is as follows.

\begin{itemize}
\item \textbf{Iterative procedure:} Assume that approximations of the function and derivatives in vector $\Phi_{n}[\ell]$ at iteration $\ell$ are known. The iterative process then proceeds in two main stages:
\begin{enumerate}
\item[\ding{192}] Compute approximations of the derivatives in vectors $\Phi^{(1)}_{n}[\ell+1]$, $\Phi^{(2)}_{n}[\ell+1]$, and $\Phi^{(3)}_{n}[\ell+1]$ using the first, second, and third physical equations, respectively, at time steps $t_{n+r}$, $r\bint{1,R}$, that is
\begin{gather*}
\Phi_{n}^{(1)}[\ell+1]=\ddf{0}{\Phi_{n}^{(0)}[\ell],T_n},\\
\Phi_{n}^{(2)}[\ell+1]=\ddf{1}{\Phi_{n}^{(0)}[\ell],\Phi_{n}^{(1)}[\ell+1],T_n},\\
\Phi_{n}^{(3)}[\ell+1]=\ddf{2}{\Phi_{n}^{(0)}[\ell],\Phi_{n}^{(1)}[\ell+1],\Phi_{n}^{(2)}[\ell+1],T_n}.
\end{gather*}
\item[\ding{193}] Compute approximations of the function in vector $\displaystyle\Phi^{(0)}_{n}[\ell+1]$ by solving a $R\times R$
linear system in the form
\[
\boldsymbol{A}^{(0)}\Phi^{(0)}_{n}[\ell+1]=-\Delta t\boldsymbol{A}^{(1)}\Phi^{(1)}_{n}[\ell+1]-\Delta t^2\boldsymbol{A}^{(2)}\Phi^{(2)}_{n}[\ell+1]-\Delta t^3\boldsymbol{A}^{(3)}\Phi^{(3)}_{n}[\ell+1]-\widehat{\boldsymbol{A}}_n\Psi_{n}.
\]
\end{enumerate}

\item \textbf{Initialisation procedure:} An approximation to vectors $\Psi_{n+r}[0]$, $r\bint{1,R}$, is prescribed from a third-order accurate Taylor expansion, given as
\begin{gather*}
\phi^{(0)}_{n+r}[0]=\phi^{(0)}_{n+r-1}[0]+\Delta t\phi^{(1)}_{n+r-1}[0]+\frac{\Delta t^2}{2}\phi^{(2)}_{n+r-1}[0]+\frac{\Delta t^3}{6}\phi^{(3)}_{n+r-1}[0],\\
\phi^{(1)}_{n+r}[0]=\ddf{0}{\phi^{(0)}_{n+r}[0],t_{n+r}},\\
\phi^{(2)}_{n+r}[0]=\ddf{1}{\phi^{(0)}_{n+r}[0],\phi^{(1)}_{n+r}[0],t_{n+r}},\\
\phi^{(3)}_{n+r}[0]=\ddf{2}{\phi^{(0)}_{n+r}[0],\phi^{(1)}_{n+r}[0],\phi^{(2)}_{n+r}[0],t_{n+r}},
\end{gather*}
where $\phi^{(k)}_{n}[0]=\phi^{(k)}_{n}$, $k\bint{0,3}$.
\end{itemize}

\subsection{Complexity}\rb{
The $\SK{K}{R}$ method computes successive $R$-point blocks of approximations, with a total number of points $N=LR$, where $L$ denotes the number of blocks. For each block, an iterative procedure is required to obtain satisfactory approximations. Each step of the iterative procedure is decomposed into two parts: the derivatives' update
\[
\Phi_n^0[\ell] \to \Phi_n^1[\ell+1]=\ddf{0}{\Phi_{n}^{(0)}[\ell]} \to \cdots \to \Phi_n^K[\ell+1]=\ddf{K-1}{\Phi_{n}^{(0)}[\ell],\cdots,\Phi_{n}^{(K-1)}[\ell]},
\]
and the solution update
\[
\Phi^{(0)}_{n}[\ell+1]=-\sum_{k=1}^K(\Delta t)^k\boldsymbol{B}^{(k)}\Phi^{(k)}_{n}[\ell+1]-\widehat{\boldsymbol{B}}_n\Psi_{n},
\]
with $\boldsymbol{B}^{(k)}=\Big[\boldsymbol{A}^{(0)}\Big]^{-1} \boldsymbol{A}^{(k)}$ and $\widehat{\boldsymbol{B}}_n\Psi_{n}=\Big[\boldsymbol{A}^{(0)}\Big]^{-1} \widehat{\boldsymbol{A}}_n\Psi_{n}$.
Of course, the $R\times R$ matrices $\boldsymbol{B}^{(k)}$ are precomputed once and thus the system only involves matrix vector product.
\\
To assess the complexity of the method, two major difficulties arise. First, the cost of evaluating $\Phi_n^1[\ell+1]$, corresponding to the first derivative, may differ substantially from the cost of evaluating $\Phi_n^K[\ell+1]$, corresponding to the highest derivative. Second, the number of iterations depends on the contraction properties of the fixed-point iteration and on the user-prescribed tolerance used to stop the loop.
\begin{itemize}
\item The computation of higher derivatives involves processing more information. Nevertheless, the corresponding routines are highly vectorised, since most arithmetic operations are not sequential. Consequently, we assume that the running time of the calls used to compute each derivative is approximately $R\tau_d$, where $\tau_d$ denotes the computational time associated with each point $t_{n+r}$.
\item The linear combination that provides $\Phi^{(0)}_{n}[\ell+1]$ is also highly parallelisable, or vectorisable, since it only involves independent matrix-vector products and summations. Consequently, we assume that the running-time complexity required to update the vector is $RK\tau_m$, where $\tau_m$ denotes the time required to compute a row-column product.
\item The total time required for one cycle is therefore $RK\tau_m+R(K-1)\tau_d$. For simple problems, $\tau_m$ and $\tau_d$ are of the same order, leading to a complexity of $O(2RK\tau_d)$. When the nonlinear contribution is more complex, for instance when it involves trigonometric or exponential functions, we have $\tau_d\gg\tau_m$, and the complexity becomes $O(RK\tau_d)$.
\end{itemize}
}

\section{Post-processing}

Previously, we focused on fully implicit structural equations, where all derivatives up to order K are coupled with the physical equations. However, using explicit structural equations may help compute higher derivatives beyond those obtained directly from the structural scheme. For example, the scheme $\SK{3}{3}$ yields approximations up to the third derivative, whereas post-processing could approximate fourth- or fifth-order derivatives. In general, the structural scheme \SK{K}{R} produces $\phi^{(k)}_N\approx\phi^{(k)}(t_N)$ for $k \bint{0,K}$ at the final time, with convergence order $d=K(R+1)$. Building on these approximations, we present a method to approximate higher derivatives $\phi_N^{(p)}$ for $p>K$.

One basic strategy is to derive additional physical equations, $\PE{K+1}$ to $\PE{p}$. However, this approach often becomes impractical, particularly if the partial derivatives are complex or the physical equation uses tabulated functions. As a practical alternative, our goal is to provide an approximation $\phi_N^{(p)}\approx\phi^{(p)}(t_N)$ with the best precision using an explicit structural equation (referred to as the post-processing structural equation) that uses the approximations $\phi_{n}^{(k)}$ computed from the structural scheme. The explicit relation takes the form
\[
\phi_N^{(p)}=\sum_{i=0}^{I_{p}}\sum_{k=0}^Ka^{p}_{ik}\,\Delta t^k\,\phi_{N-i}^{(k)},
\]
where $I_p$ is the number of backward time steps required to exactly approximate any polynomial up to degree $d$. To determine the $(I_p+1)\times(K+1)$ matrix $\ba^p=\big(a^{p}_{ik}\big)_{i=0,\ldots,I_p\atop k=0,\ldots,K}$, we introduce the functional
\[
E_p(\ba^p;\pi)=\sum_{i=0}^{I_p}\sum_{k=0}^Ka^{p}_{ik}\,\Delta t^k\,\pi^{(k)}(t_{N-i})-\pi^{(p)}(t_N),
\]
for any function $\pi=\pi(t)$, regular enough.

We reshape the matrix $\ba^p$ into a single-index entry vector $\ba^p$ (we use the same notation for simplicity) with one-to-one mapping $(i,k)\to\ell=\ell(i,k)=k(I_p+1)+i+1$ from $\bbint{0,K}\times\bbint{0,I_p}$ onto $\bbint{1,M_p}$, with $M_p=(K+1)(I_p+1)$ the number of entries from $\ba^p$, and the functional can be rewritten as
\[
E_p(\ba^p;\pi)=\sum_{i=0}^{I_p}\sum_{k=0}^Ka^{p}_\ell\,\Delta t^k\,\pi^{(k)}(t_{N-i})-\pi^{(p)}(t_N),
\quad\ell=\ell(i,k).
\]

To ensure accuracy, we require that the functional exactly satisfies $E(\ba;\pi_m)=0$ for polynomial functions up to a chosen degree. For this, we consider the following set of polynomial functions
\[
\pi_m(t;t_N)=\left(\frac{t-t_N}{\Delta t}\right)^{m-1},\quad m\bint{1,\overline{M}_p},\,\overline{M}_p=p+d.
\]
Note that the number of coefficients is $M_p$ and we need to impose the condition $\overline{M}_p\leqslant M_p$ such that the linear system has at least one non-trivial solution. We then deduce a sufficient condition for the number of backward time steps $I_p$ as
\[
p+K(R+1)-K-1\leqslant I_p(K+1)\implies I_p\geqslant\frac{KR+p-1}{K+1}.
\]

\begin{rem}
For $K=1$ and $R=2$, we have $I_{p}\geqslant(1+p)/2$, while $K=3$ and $R=4$ requires $I_{p}\geqslant(11+p)/4$.
\hfill\smallBB
\end{rem}

To compute the coefficients in $\ba^p$, we solve an under-determined linear system $\mathcal{M}^p\ba^p=\bb^p$. Here, $\mathcal{M}^p$ is a $\overline{M}_p\times M_p$ matrix. The general solution consists of a particular solution together with any vector in the kernel of $\mathcal{M}^p$. For practical computation, the least-squares method is effective, since $\mathcal{M}^p$ has maximal rank.

The entries of matrix $\mathcal{M}^p$ given by the $k$-th derivative of the polynomial functions $\pi_m(t;t_N)$, $m\bint{0,\overline{M}_p}$, read
\[
\pi_m^{(k)}(t;t_N)=\frac{(m-1)\cdots(m-k)}{\Delta t^k}\left(\frac{t-t_N}{\Delta t}\right)^{m-1-k},\,k<m,
\]
and
\[
\pi_m^{(k)}(t;t_N)=0,\,k\geqslant m.
\]
After calculations, for $k<m$ and $i\bint{0,I_p}$, yields
\[
\pi_m^{(k)}(t_{N-i};t_N)=\frac{(m-1)\cdots(m-k)}{\Delta t^k}(-i)^{m-1-k}.
\]
Then, the entries of matrix $\mathcal{M}^p$ are given as
\[
\mathcal{M}^p[m,\ell]=C^m_\ell(-i)^{m-1-\ell},\quad C^m_\ell=
\begin{cases}
\frac{m!}{(m-\ell)!} & \text{if $m>\ell$},\\
0 & \text{otherwise},
\end{cases}
\quad\ell=\ell(i,k),
\]
for $m\bint{1,\overline{M}_p}$, $\ell\bint{1,M_p}$. Several cases of post-processing analytical formulas are given in~\ref{appendix::postproc}.

\begin{rem}
For larger values of $K$ and $p$, or for non-uniform time grids, the matrix $\mathcal{M}^p$ and right-hand side vector $\bb^p$ are still computed numerically. This involves enforcing the condition $E_p(\ba^p;\pi)=0$ for $m$ in $\bint{0,\overline{M}_p}$. In such cases, applying the least-squares method provides a solution and yields the post-processing structural equation.
\end{rem}

We propose a simple benchmark to assess approximation errors, using the function $\phi(t)=\exp(t)$ as an example. We compute $\phi^{(p)}_N$ with the explicit structural equations, where $\phi^{(k)}_{N-i}$ are substituted with the exact values. The errors at the final time $\phi^{(p)}(T)-\phi^{(p)}_N$, $T=1$, together with the convergence orders are provided in Tables~\ref{ODE1_K1_consistency},~\ref{ODE1_K2_consistency}, and~\ref{ODE1_K3_consistency}. Several examples of $p$-th derivative approximation using post-processing structural equations are assessed with different parameter values $K$ and $R$.

\begin{table}[ht]
\centering
\caption{Consistency errors with $K=1$.}
\label{ODE1_K1_consistency}
\begin{tabular}{@{}r@{}r@{}r@{}r@{}rr@{}r@{}rr@{}r@{}rr@{}}
\toprule
Ip & \phantom{aaa} & N & \phantom{aaa} &
\multicolumn{2}{@{}l@{}}{$\phi^{(2)}$} & \phantom{aaa} &
\multicolumn{2}{@{}l@{}}{$\phi^{(3)}$} & \phantom{aaa} &
\multicolumn{2}{@{}l@{}}{$\phi^{(4)}$}\\
\cmidrule{5-6}\cmidrule{8-9}\cmidrule{11-12}
&& && err & ord && err & ord && err & ord\\
\midrule
\multirow{3}{*}{1} && 60 && 6.25E-05 & \emdash && 2.25E-02 & \emdash && 2.72E+00 & \emdash\\
&& 120 && 1.57E-05 & 2.0 && 1.13E-02 & 1.0 && 2.72E+00 & 0.0\\
&& 240 && 3.93E-06 & 2.0 && 5.66E-03 & 1.0 && 2.72E+00 & 0.0\\
\midrule
\multirow{3}{*}{2} && 60 && 2.30E-09 & \emdash & & 1.24E-06 & \emdash & & 3.23E-04 & \emdash\\
&& 120 && 1.45E-10 & 4.0 & & 1.56E-07 & 3.0 & & 8.13E-05 & 2.0\\
&& 240 && 9.07E-12 & 4.0 & & 1.96E-08 & 3.0 & & 2.04E-05 & 2.0\\
\midrule
\multirow{3}{*}{3} && 60 && 1.02E-13 & \emdash & & 6.72E-11 & \emdash & & 2.36E-08 & \emdash\\
&& 120 && 1.61E-15 & 6.0 & & 2.12E-12 & 5.0 & & 1.49E-09 & 4.0\\
&& 240 && 2.53E-17 & 6.0 & & 6.67E-14 & 5.0 & & 9.36E-11 & 4.0\\
\midrule
\multirow{3}{*}{4} && 60 && 4.98E-18 & \emdash & & 3.74E-15 & \emdash & & 1.56E-12 & \emdash\\
&& 120 && 1.98E-20 & 8.0 & & 2.97E-17 & 7.0 & & 2.48E-14 & 6.0\\
&& 240 && 7.78E-23 & 8.0 & & 2.33E-19 & 7.0 & & 3.90E-16 & 6.0\\
\midrule
\multirow{3}{*}{5} && 60 && 2.60E-22 & \emdash & & 2.14E-19 & \emdash & & 1.01E-16 & \emdash\\
&& 120 && 2.59E-25 & 10.0 & & 4.26E-22 & 9.0 & & 4.01E-19 & 8.0\\
&& 240 && 2.55E-28 & 10.0 & & 8.40E-25 & 9.0 & & 1.58E-21 & 8.0\\
\bottomrule
\end{tabular}
\end{table}

\begin{table}[H]
\centering
\begin{minipage}{0.45\textwidth}
\centering
\caption{Consistency errors with $K=2$.}
\label{ODE1_K2_consistency}
\begin{tabular}{@{}r@{}r@{}r@{}r@{}rr@{}r@{}rr@{}}
\toprule
Ip & \phantom{aaa} & N & \phantom{aaa} &
\multicolumn{2}{@{}l@{}}{$\phi^{(3)}$} & \phantom{aaa} &
\multicolumn{2}{@{}l@{}}{$\phi^{(4)}$}\\
\cmidrule{5-6}\cmidrule{8-9}
&& && err & ord && err & ord\\
\midrule
\multirow{3}{*}{1} && 60 && 1.04E-07 & \emdash & & 7.50E-05 & \emdash\\
&& 120 && 1.31E-08 & 3.0 & & 1.88E-05 & 2.0\\
&& 240 && 1.64E-09 & 3.0 & & 4.71E-06 & 2.0\\
\midrule
\multirow{3}{*}{2} && 60 && 7.59E-15 & \emdash & & 8.20E-12 & \emdash\\
&& 120 && 1.20E-16 & 6.0 & & 2.58E-13 & 5.0\\
&& 240 && 1.87E-18 & 6.0 & & 8.10E-15 & 5.0\\
\midrule
\multirow{3}{*}{3} && 60 && 7.13E-22 & \emdash & & 9.42E-19 & \emdash\\
&& 120 && 1.41E-24 & 9.0 & & 3.72E-21 & 8.0\\
&& 240 && 2.77E-27 & 9.0 & & 1.46E-23 & 8.0\\
\midrule
\multirow{3}{*}{4} && 60 && 7.68E-29 & \emdash & & 1.15E-25 & \emdash\\
&& 120 && 1.90E-32 & 12.0 & & 5.71E-29 & 11.0\\
&& 240 && 4.68E-36 & 12.0 & & 2.81E-32 & 11.0\\
\midrule
\multirow{3}{*}{5} && 60 && 9.00E-36 & \emdash & & 1.48E-32 & \emdash\\
&& 120 && 2.80E-40 & 15.0 & & 9.21E-37 & 14.0\\
&& 240 && 8.63E-45 & 15.0 & & 5.68E-41 & 14.0\\
\bottomrule
\end{tabular}
\end{minipage}
\hfill
\begin{minipage}{0.45\textwidth}
\centering
\caption{Consistency errors with $K=3$.}
\label{ODE1_K3_consistency}
\begin{tabular}{@{}r@{}r@{}r@{}r@{}rr@{}}
\toprule
Ip & \phantom{aaa} & N & \phantom{aaa} &
\multicolumn{2}{@{}l@{}}{$\phi^{(4)}$}\\
\cmidrule{5-6}
&& && err & ord\\
\midrule
\multirow{3}{*}{1} && 60 && 1.24E-10 & \emdash\\
&& 120 && 7.77E-12 & 4.0\\
&& 240 && 4.87E-13 & 4.0\\
\midrule
\multirow{3}{*}{2} && 60 && 1.28E-20 & \emdash\\
&& 120 && 5.03E-23 & 8.0\\
&& 240 && 1.97E-25 & 8.0\\
\midrule
\multirow{3}{*}{3} && 60 && 1.81E-30 & \emdash\\
&& 120 && 4.48E-34 & 12.0\\
&& 240 && 1.10E-37 & 12.0\\
\midrule
\multirow{3}{*}{4} && 60 && 3.06E-40 & \emdash\\
&& 120 && 4.74E-45 & 16.0\\
&& 240 && 7.28E-50 & 16.0\\
\midrule
\multirow{3}{*}{5} && 60 && 5.73E-50 & \emdash\\
&& 120 && 5.57E-56 & 20.0\\
&& 240 && 5.37E-62 & 20.0\\
\bottomrule
\end{tabular}
\end{minipage}
\end{table}

The structural scheme $\SK{K}{R}$ has order $d=K(R+1)$, so $\phi^{(K)}_N$ approximates $\phi^{(K)}(t_N)$ to convergence order $O(\Delta t^d)$. However, approximating higher derivatives results in a loss of convergence order. Specifically, $\phi^{(K+1)}_N$ has order $d-1$, while $\phi^{(p)}_N$ for $p < d+K$ has order $d-(p-K)$.

\begin{rem}
Convergence cannot be achieved for $\phi^{(p)}_N$ if $d-(p-K)\leqslant 0$. Therefore, the explicit structural equations require $p<K(R+1)+K=K(R+2)$, such that convergence is controlled by the two parameters $K$ and $R$ along with the mandatory number of backward time steps $I_p\geqslant I_{p,min}=|(KR+p-1)/(K+1)|$.
\hfill\smallBB
\end{rem}

\section{Benchmarks}

A comprehensive benchmark is presented to assess the relevant properties of the structural schemes, namely, numerical accuracy, linear stability, spectral resolution, computational efficiency, and the ability to handle stiff problems. To assess the numerical accuracy of the structural schemes, the approximation error at time $t_{N}=N\Delta t=T$ is defined as
\[
e_N=\left|\phi^{(k)}_N-\phi^{(k)}(T)\right|,
\]
where $\phi^{(k)}_N$ is the numerical approximation of $\phi^{(k)}(t_N)$. Then, the convergence order for the approximation error of two numerical solutions $\phi^{(k)}_{N_1}$ and $\phi^{(k)}_{N_2}$ computed on two time grids of sizes $N_1$ and $N_2$, respectively, with $N_1\neq N_2$, is given as
\[
O(N_1,N_2)=\frac{\left|\log(e_{N_1}/e_{N_2}\right)|}{|\log(N_1/N_2)|}.
\]
\rb{
\begin{rem}
We would like to mention that the structural method has also been tested on Hamiltonian problems covering a broad range of situations, including the pendulum, the Euler system, the three-body problem, and the motion of a particle in a variable electromagnetic field~\cite{CFM25}.
\end{rem}
}
\rb{
\begin{rem}
A library dedicated to the numerical solution of ODEs using the structural method is available through an open-access GitHub repository~\cite{C26}. The codes were developed in {\tt Julia} and {\tt C++}, using the {\tt qd} library to handle quad- and octuple-precision arithmetic.
 \end{rem}
 }
\subsection{Linear scalar ODE}

The first sanity check benchmark concerns the linear ODE $\phi'(t)=-\lambda\phi(t)$ in $(0,1]$, with the initial condition $\phi(0)=1$, whose exact solution is $\phi(t)=\exp(-\lambda t)$, to assess several relevant issues that arise in the calculation of the approximations:
\begin{enumerate}
\item Verify whether the optimal accuracy, as predicted from the previous analysis, is obtained for this simple example. In particular, when dealing with an asymptotic convergence $C\Delta t^m$, the magnitude of the multiplicative constant $C$ for different structural schemes of the same convergence order is analysed. In that regard, the convergence orders of the structural schemes $\SK{K}{R}$ for the approximation of the solution and derivatives $\phi^{(k)}(t)$, $k\leqslant K$, at the final time, are reported in black, whereas an approximation of the derivatives $\phi^{(k)}(t)$ with $k>K$, obtained from post-processing structural equations, are reported in blue. Several possible combinations of the number of backward time steps $I_p$ are analysed.
\item Assess the impact of the tolerance parameter $\varepsilon$ for the fixed-point iterative process on the accuracy of the structural schemes. Indeed, if $\varepsilon$ is not sufficiently small, some accuracy loss is expected since the iterative process stops before reaching the optimal approximation. On the contrary, if $\varepsilon$ is too small, there is an unnecessary computational effort from the additional fixed-point iterations performed after the optimal accuracy has already been reached.
\item Assess the influence of the number of variables and the convergence order through the parameters $K$ and $R$ on the memory load. In particular, if the structural schemes are used for non-stationary PDE problems or large differential systems, the computations might become mostly memory-bound.
\item Confirm whether intermediate time grid point approximations in the last block achieve the same optimal convergence order as the final time point, for which the approximation errors $|\phi^{(k)}_{N-r}-\phi^{(k)}(T-r\Delta t)|$ for $r\bint{1,R-1}$, are computed.
\end{enumerate}

\subsubsection{Convergence order}

Considering $\lambda=1$, we check the convergence order of the structural schemes $\SK{1}{R}$ with $R\bint{1,5}$. The errors and convergence orders are reported in Table~\ref{ODE1_K1_0_v1}, where the second, third, and fourth derivatives were obtained from the post-processing structural equations using $I_p=I_{p,\min}$. Super-convergence is observed for $R=2,4$ (where the third and fifth orders were expected, respectively), although the post-calculated derivatives present the nominal convergence order, which decreases as the derivative order increases, as expected.

\begin{rem}
Post-processing has been performed with higher values for $I_p$, namely $I_p=I_{p,\min}+1$, $I_p=I_{p,\min}+2$, and $I_p=I_{p,\min}+3$, which are not reported since comparable convergence orders to the case with $I_p=I_{p,\min}$ were obtained. Nevertheless, it is observed that the approximation errors increase with $I_p$ due to larger stencils. Thus, the optimal choice remains $I_p=I_{p,\min}$.
\hfill\smallBB
\end{rem}

\begin{table}[H]
\centering
\caption{Errors and convergence orders for the linear scalar ODE benchmark with $K=1$, $I_p=I_{p,\min}$, and $\varepsilon=10^{-20}$.}
\label{ODE1_K1_0_v1}
\begin{tabular}{@{}r@{}r@{}r@{}r@{}rr@{}r@{}rr@{}r@{}rr@{}r@{}rr@{}r@{}rr@{}}
\toprule
R & \phantom{aa} & N & \phantom{aa} &
\multicolumn{2}{@{}l@{}}{$\phi$} & \phantom{aa} &
\multicolumn{2}{@{}l@{}}{$\phi^{(1)}$} & \phantom{aa} &
\multicolumn{2}{@{}l@{}}{$\phi^{(2)}$} & \phantom{aa} &
\multicolumn{2}{@{}l@{}}{$\phi^{(3)}$} & \phantom{aa} &
\multicolumn{2}{@{}l@{}}{$\phi^{(4)}$}\\
\cmidrule{5-6}\cmidrule{8-9}\cmidrule{11-12}\cmidrule{14-15}\cmidrule{17-18}
&& && err & ord && err & ord && err & ord && err & ord && err & ord\\
\midrule
\multirow{4}{*}{1} && && && && && \multicolumn{2}{r@{}}{\textcolor{blue}{Ip=1}} && \multicolumn{2}{r@{}}{\textcolor{blue}{Ip=2}} && \multicolumn{2}{r@{}}{\textcolor{blue}{Ip=2}}\\
&& 60 && \textcolor{black}{8.52E-06} & \textcolor{black}{\emdash} && \textcolor{black}{8.52E-06} & \textcolor{black}{\emdash} && \textcolor{blue}{3.08E-03} & \textcolor{blue}{\emdash} && \textcolor{blue}{2.33E+00} & \textcolor{blue}{\emdash} && \textcolor{blue}{3.36E+02} & \textcolor{blue}{\emdash}\\
&& 120 && \textcolor{black}{2.13E-06} & \textcolor{black}{2.0} && \textcolor{black}{2.13E-06} & \textcolor{black}{2.0} && \textcolor{blue}{1.54E-03} & \textcolor{blue}{1.0} && \textcolor{blue}{2.32E+00} & \textcolor{blue}{0.0} && \textcolor{blue}{6.67E+02} & \textcolor{blue}{\emdash}\\
&& 240 && \textcolor{black}{5.32E-07} & \textcolor{black}{2.0} && \textcolor{black}{5.32E-07} & \textcolor{black}{2.0} && \textcolor{blue}{7.67E-04} & \textcolor{blue}{1.0} && \textcolor{blue}{2.31E+00} & \textcolor{blue}{0.0} && \textcolor{blue}{1.33E+03} & \textcolor{blue}{\emdash}\\
\midrule
\multirow{4}{*}{2} && && && && && \multicolumn{2}{r@{}}{\textcolor{blue}{Ip=2}} && \multicolumn{2}{r@{}}{\textcolor{blue}{Ip=3}} && \multicolumn{2}{r@{}}{\textcolor{blue}{Ip=3}}\\
&& 60 && \textcolor{black}{6.31E-10} & \textcolor{black}{\emdash} && \textcolor{black}{6.31E-10} & \textcolor{black}{\emdash} && \textcolor{blue}{3.46E-05} & \textcolor{blue}{\emdash} && \textcolor{blue}{6.96E-03} & \textcolor{blue}{\emdash} && \textcolor{blue}{3.47E+00} & \textcolor{blue}{\emdash}\\
&& 120 && \textcolor{black}{3.94E-11} & \textcolor{black}{4.0} && \textcolor{black}{3.94E-11} & \textcolor{black}{4.0} && \textcolor{blue}{8.59E-06} & \textcolor{blue}{2.0} && \textcolor{blue}{3.50E-03} & \textcolor{blue}{1.0} && \textcolor{blue}{3.47E+00} & \textcolor{blue}{\emdash}\\
&& 240 && \textcolor{black}{2.46E-12} & \textcolor{black}{4.0} && \textcolor{black}{2.46E-12} & \textcolor{black}{4.0} && \textcolor{blue}{2.14E-06} & \textcolor{blue}{2.0} && \textcolor{blue}{1.76E-03} & \textcolor{blue}{1.0} && \textcolor{blue}{3.47E+00} & \textcolor{blue}{\emdash}\\
\midrule
\multirow{4}{*}{3} && && && && && \multicolumn{2}{r@{}}{\textcolor{blue}{Ip=2}} && \multicolumn{2}{r@{}}{\textcolor{blue}{Ip=3}} && \multicolumn{2}{r@{}}{\textcolor{blue}{Ip=3}}\\
&& 60 && \textcolor{black}{3.55E-10} & \textcolor{black}{\emdash} && \textcolor{black}{3.55E-10} & \textcolor{black}{\emdash} && \textcolor{blue}{4.36E-07} & \textcolor{blue}{\emdash} && \textcolor{blue}{9.56E-05} & \textcolor{blue}{\emdash} && \textcolor{blue}{9.33E-03} & \textcolor{blue}{\emdash}\\
&& 120 && \textcolor{black}{2.22E-11} & \textcolor{black}{4.0} && \textcolor{black}{2.22E-11} & \textcolor{black}{4.0} && \textcolor{blue}{5.39E-08} & \textcolor{blue}{3.0} && \textcolor{blue}{2.37E-05} & \textcolor{blue}{2.0} && \textcolor{blue}{4.63E-03} & \textcolor{blue}{1.0}\\
&& 240 && \textcolor{black}{1.39E-12} & \textcolor{black}{4.0} && \textcolor{black}{1.39E-12} & \textcolor{black}{4.0} && \textcolor{blue}{6.69E-09} & \textcolor{blue}{3.0} && \textcolor{blue}{5.88E-06} & \textcolor{blue}{2.0} && \textcolor{blue}{2.31E-03} & \textcolor{blue}{1.0}\\
\midrule
\multirow{4}{*}{4} && && && && && \multicolumn{2}{r@{}}{\textcolor{blue}{Ip=3}} && \multicolumn{2}{r@{}}{\textcolor{blue}{Ip=4}} && \multicolumn{2}{r@{}}{\textcolor{blue}{Ip=4}}\\
&& 60 && \textcolor{black}{1.00E-13} & \textcolor{black}{\emdash} && \textcolor{black}{1.00E-13} & \textcolor{black}{\emdash} && \textcolor{blue}{5.87E-09} & \textcolor{blue}{\emdash} && \textcolor{blue}{1.46E-06} & \textcolor{blue}{\emdash} && \textcolor{blue}{1.83E-04} & \textcolor{blue}{\emdash}\\
&& 120 && \textcolor{black}{1.56E-15} & \textcolor{black}{6.0} && \textcolor{black}{1.56E-15} & \textcolor{black}{6.0} && \textcolor{blue}{3.61E-10} & \textcolor{blue}{4.0} && \textcolor{blue}{1.80E-07} & \textcolor{blue}{3.0} && \textcolor{blue}{4.53E-05} & \textcolor{blue}{2.0}\\
&& 240 && \textcolor{black}{2.44E-17} & \textcolor{black}{6.0} && \textcolor{black}{2.44E-17} & \textcolor{black}{6.0} && \textcolor{blue}{2.24E-11} & \textcolor{blue}{4.0} && \textcolor{blue}{2.23E-08} & \textcolor{blue}{3.0} && \textcolor{blue}{1.12E-05} & \textcolor{blue}{2.0}\\
\midrule
\multirow{4}{*}{5} && && && && && \multicolumn{2}{r@{}}{\textcolor{blue}{Ip=3}} && \multicolumn{2}{r@{}}{\textcolor{blue}{Ip=4}} && \multicolumn{2}{r@{}}{\textcolor{blue}{Ip=4}}\\
&& 60 && \textcolor{black}{3.59E-14} & \textcolor{black}{\emdash} && \textcolor{black}{3.59E-14} & \textcolor{black}{\emdash} && \textcolor{blue}{8.22E-11} & \textcolor{blue}{\emdash} && \textcolor{blue}{2.24E-08} & \textcolor{blue}{\emdash} && \textcolor{blue}{3.31E-06} & \textcolor{blue}{\emdash}\\
&& 120 && \textcolor{black}{5.60E-16} & \textcolor{black}{6.0} && \textcolor{black}{5.60E-16} & \textcolor{black}{6.0} && \textcolor{blue}{2.52E-12} & \textcolor{blue}{5.0} && \textcolor{blue}{1.38E-09} & \textcolor{blue}{4.0} && \textcolor{blue}{4.06E-07} & \textcolor{blue}{3.0}\\
&& 240 && \textcolor{black}{8.75E-18} & \textcolor{black}{6.0} && \textcolor{black}{8.75E-18} & \textcolor{black}{6.0} && \textcolor{blue}{7.78E-14} & \textcolor{blue}{5.0} && \textcolor{blue}{8.52E-11} & \textcolor{blue}{4.0} && \textcolor{blue}{5.03E-08} & \textcolor{blue}{3.0}\\
\bottomrule
\end{tabular}
\end{table}

\subsubsection{Sensitivity to the tolerance $\varepsilon$}

The solution of a block of size $R$ is obtained from an iterative procedure described in Section~\ref{sec::numerical_solver} and controlled by a tolerance $\varepsilon$. We denote by $\bar\kappa$ the average number of iterations to solve an $R$-block. Therefore, for a structural scheme $\SK{K}{R}$, the physical equations are evaluated $KR\bar\kappa$ times to solve a block, thus $KN\bar\kappa$ is the total number of evaluations for $N$ steps. We then denote by $\bar\tau=K\bar\kappa$ the average number of evaluations per time step.

To assess the computational effort, the values of $\bar\kappa$ and $\bar \tau$ for two values of $\varepsilon$ are reported in Tables~\ref{ODE1_N60_1},~\ref{ODE1_N60_2}, and~\ref{ODE1_N60_3}. The first table concerns structural schemes with convergence orders of $2$, $4$, and $6$, where $\varepsilon$ is chosen based on the approximation errors. The second table concerns the convergence orders of $8$, $10$, and $12$, for which smaller values of $\varepsilon$ are used. The last table concerns the highest convergence orders with a substantial reduction in the values of $\varepsilon$. As expected, the number of iterations increases when $\varepsilon$ decreases. Quantifying the additional iterations is not easy, as there is significant variation between cases, ranging from half to twice the number of iterations, even when $\varepsilon$ varies by 10 or 15 orders of magnitude. In conclusion, the computational effort is not substantially affected by the strong reduction in the tolerance criterion, and for the remainder of the section, we use, by default, $\varepsilon=10^{-20}$ for $K=1$, $\varepsilon=10^{-40}$ for $K=2$, and $\varepsilon=10^{-60}$ for $K=3$.

\begin{table}[H]
\centering
\caption{Average number of fixed-point iterations for the linear scalar ODE benchmark with $N=60$ for structural schemes with convergence orders of 2, 4, and 6.}
\label{ODE1_N60_1}
\begin{tabular}{@{}lr@{}rrrrr@{}r@{}rr@{}r@{}r@{}}
\toprule
K & \phantom{aa} & \multicolumn{5}{c}{1} & \phantom{aa} & \multicolumn{2}{c}{2} & \phantom{aa} & \multicolumn{1}{c}{3}\\\cmidrule{1-1}\cmidrule{3-7}\cmidrule{9-10}\cmidrule{12-12}
R (order) && 1 (2) & 2 (4) & 3 (4) & 4 (6) & 5 (6) && 1 (4) & 2 (6) && 1 (6)\\
\midrule
$\bar\kappa$ with $\varepsilon$=1E-20 && 9.00 & 5.00 & 3.33 & 2.50 & 2.20 && 8.00 & 4.50 && 7.00\\
$\bar\tau$ with $\varepsilon$=1E-20 && 9.00 & 5.00 & 3.33 & 2.50 & 2.20 && 16.00 & 4.50 && 21.00\\
\midrule
$\bar\kappa$ with $\varepsilon$=1E-30 && 14.00 & 7.50 & 5.00 & 4.00 & 3.20 && 13.00 & 7.00 && 12.00\\
$\bar\tau$ with $\varepsilon$=1E-30 && 14.00 & 7.50 & 5.00 & 4.00 & 3.20 && 26.00 & 14.00 && 36.00\\
\bottomrule
\end{tabular}
\end{table}

\begin{table}[H]
\centering
\caption{Average number of fixed-point iterations for the linear scalar ODE benchmark with $N=60$ for structural schemes with convergence orders of 8, 10, and 12.}
\label{ODE1_N60_2}
\begin{tabular}{@{}lr@{}rrr@{}r@{}rr@{}r@{}rr@{}}
\toprule
K &\phantom{aa} &\multicolumn{3}{c}{2} &\phantom{aa} &\multicolumn{2}{c}{3} &\phantom{aa} &\multicolumn{2}{c}{4}\\\cmidrule{1-1}\cmidrule{3-5}\cmidrule{7-8}\cmidrule{10-11}
R (order) && 3 (8) & 4 (10) & 5 (12) && 2 (10) & 3 (12) && 1 (10) & 2 (12)\\
\midrule
$\bar\kappa$ with $\varepsilon$=1E-25 && 4.17 & 3.47 & 3.00 && 4.97 & 3.33 && 8.00 & 4.50\\
$\bar\tau$ with $\varepsilon$=1E-25 && 8.34 & 6.94 & 6.00 && 14.91 & 9.99 && 32.00 & 18.00\\
\midrule
$\bar\kappa$ with $\varepsilon$=1E-40 && 7.00 & 5.75 & 5.00 && 8.50 & 6.00 && 15.00 & 8.45\\
$\bar\tau$ with $\varepsilon$=1E-40 && 14.00 & 11.50 & 10.00 && 25.50 & 18.00 && 60.00 & 33.80\\
\bottomrule
\end{tabular}
\end{table}

\begin{table}[H]
\centering
\caption{Average number of fixed-point iterations for the linear scalar ODE benchmark with $N=60$ for structural schemes with convergence orders of 16, 18, 20, and 24.}
\label{ODE1_N60_3}
\begin{tabular}{@{}lr@{}rr@{}r@{}rrr@{}}
\toprule
K &\phantom{aa} &\multicolumn{2}{c}{3} &\phantom{aa} &\multicolumn{3}{c}{4}\\\cmidrule{1-1}\cmidrule{3-4}\cmidrule{6-8}
R (order) && 4 (16) & 5 (18) && 3 (16) & 4 (20) & 5 (24)\\
\midrule
$\bar\kappa$ with $\varepsilon$=1E-40 && 4.93 & 4.40 && 6.12 & 5.35 & 5.40\\
$\bar\tau$ with $\varepsilon$=1E-40 && 14.79 & 13.20 && 24.48 & 21.40 & 21.60\\
\midrule
$\bar\kappa$ with $\varepsilon$=1E-60 && 7.75 & 6.83 && 10.00 & 8.75 & 8.98\\
$\bar\tau$ with $\varepsilon$=1E-60 && 23.25 & 20.49 && 40.00 & 35.00 & 35.92\\
\bottomrule
\end{tabular}
\end{table}

\subsubsection{Memory load}

The number of variables to solve in a block of the structural scheme $\SK{K}{R}$ is $M=(R+1)(K+1)$, while the convergence order is $O=(R+1)K$, from which we deduce the simple relation $M=(K+1)/K\,O$ between the memory storage and the convergence order. Hence, for a vector system of dimension $L$, for instance, a finite difference semi-discretisation in space with $L$ nodes, the number of variables to compute is $M=L(R+1)(K+1)$. As an example, a sixth-order accurate method in time with $K=2$ and $R=1$ requires $M=6L$ variables.

\subsubsection{Convergence order of the intermediate steps}

An important property of the structural scheme is to guarantee the same convergence order at all time steps, including the intermediate time grid points within a block, that is, at $t_{n+1}$ up to $t_{n+R-1}$. We check this property for the simple ODE $\phi'(t)=-\phi(t)$ for two representative situations. First, we consider the structural scheme $\SK{2}{2}$ and check that the variables at the middle time step $t_{n+1}$ have the same convergence order as at $t_{n+2}$. To this end, the approximation errors $\phi^{(k)}_{N-1}-\phi^{(k)}(T-\Delta t)$ and $\phi^{(k)}_{N}-\phi^{(k)}(T)$ for $k\bint{0,2}$, and the corresponding convergence orders, are reported in Table~\ref{ODE1_SemPre_intermediate_steps_K=2_R=2_intermediate_steps}. The results confirm that the approximations at the intermediate time grid points have the same convergence order as those at the final time.

\begin{table}[H]
\centering
\caption{Errors and convergence orders for the linear scalar ODE benchmark at the last $R$ time steps (last block) with $K=2$ and $R=2$.}
\label{ODE1_SemPre_intermediate_steps_K=2_R=2_intermediate_steps}
\begin{tabular}{@{}r@{}r@{}r@{}r@{}rr@{}r@{}rr@{}r@{}rr@{}}
\toprule
t & \phantom{aaa} & N & \phantom{aaa} &
\multicolumn{2}{@{}l@{}}{$\phi$} & \phantom{aaa} &
\multicolumn{2}{@{}l@{}}{$\phi^{(1)}$} & \phantom{aaa} &
\multicolumn{2}{@{}l@{}}{$\phi^{(2)}$}\\
\cmidrule{5-6}\cmidrule{8-9}\cmidrule{11-12}
&& && err & ord && err & ord && err & ord\\
\midrule
\multirow{3}{*}{$T-1\Delta t$} && 60 && 8.34E-16 & --- && 8.34E-16 & --- && 8.34E-16 & ---\\
&& 120 && 1.30E-17 & 6.0 && 1.30E-17 & 6.0 && 1.30E-17 & 6.0\\
&& 240 && 2.04E-19 & 6.0 && 2.04E-19 & 6.0 && 2.04E-19 & 6.0\\
\midrule
\multirow{3}{*}{$T-0\Delta t$} && 60 && 8.34E-16 & --- && 8.34E-16 & --- && 8.34E-16 & ---\\
&& 120 && 1.30E-17 & 6.0 && 1.30E-17 & 6.0 && 1.30E-17 & 6.0\\
&& 240 && 2.04E-19 & 6.0 && 2.04E-19 & 6.0 && 2.04E-19 & 6.0\\
\bottomrule
\end{tabular}
\end{table}

We repeat the experience with a more challenging situation, considering a 24th-order of convergence with the structural scheme $\SK{4}{5}$ and $\varepsilon=10^{-80}$, and check the approximation errors and convergence orders for all intermediate time steps, that is, $\phi^{(k)}_{N-r}-\phi^{(k)}(T-r\Delta t)$, $r\bint{0,R-1}$, which are reported in Table~\ref{ODE1_SemPre_intermediate_steps_K=4_R=5_intermediate_steps}. As observed, all intermediate time steps present the same optimal convergence order. Thus, we conclude that $\phi^{(k)}_{N-r}$ are not auxiliary variables but genuine approximations that can be used in complex models involving the coupling of differential equations, non-linearity, or boundary conditions if the system derives from a semi-discretisation in space.

\begin{table}[H]
\centering
\caption{Errors and convergence orders for the linear scalar ODE benchmark at the last $R$ time steps (last block) with $K=4$ and $R=5$.}
\label{ODE1_SemPre_intermediate_steps_K=4_R=5_intermediate_steps}
\begin{tabular}{@{}r@{}r@{}r@{}r@{}rr@{}r@{}rr@{}r@{}rr@{}r@{}rr@{}r@{}rr@{}}
\toprule
t & \phantom{aaa} & N & \phantom{aaa} &
\multicolumn{2}{@{}l@{}}{$\phi$} & \phantom{aaa} &
\multicolumn{2}{@{}l@{}}{$\phi^{(1)}$} & \phantom{aaa} &
\multicolumn{2}{@{}l@{}}{$\phi^{(2)}$} & \phantom{aaa} &
\multicolumn{2}{@{}l@{}}{$\phi^{(3)}$} & \phantom{aaa} &
\multicolumn{2}{@{}l@{}}{$\phi^{(4)}$}\\
\cmidrule{5-6}\cmidrule{8-9}\cmidrule{11-12}\cmidrule{14-15}\cmidrule{17-18}
&& && err & ord && err & ord && err & ord && err & ord && err & ord\\
\midrule
\multirow{3}{*}{$T-4\Delta t$} && 60 && 1.47E-63 & --- && 1.47E-63 & --- && 1.47E-63 & --- && 1.47E-63 & --- && 1.47E-63 & ---\\
&& 120 && 8.67E-71 & 24.0 && 8.67E-71 & 24.0 && 8.67E-71 & 24.0 && 8.67E-71 & 24.0 && 8.67E-71 & 24.0\\
&& 240 && 5.14E-78 & 24.0 && 5.14E-78 & 24.0 && 5.14E-78 & 24.0 && 5.14E-78 & 24.0 && 5.14E-78 & 24.0\\
\midrule
\multirow{3}{*}{$T-3\Delta t$} && 60 && 1.45E-63 & --- && 1.45E-63 & --- && 1.45E-63 & --- && 1.45E-63 & --- && 1.45E-63 & ---\\
&& 120 && 8.60E-71 & 24.0 && 8.60E-71 & 24.0 && 8.60E-71 & 24.0 && 8.60E-71 & 24.0 && 8.60E-71 & 24.0\\
&& 240 && 5.12E-78 & 24.0 && 5.12E-78 & 24.0 && 5.12E-78 & 24.0 && 5.12E-78 & 24.0 && 5.12E-78 & 24.0\\
\midrule
\multirow{3}{*}{$T-2\Delta t$} && 60 && 1.42E-63 & --- && 1.42E-63 & --- && 1.42E-63 & --- && 1.42E-63 & --- && 1.42E-63 & ---\\
&& 120 && 8.53E-71 & 24.0 && 8.53E-71 & 24.0 && 8.53E-71 & 24.0 && 8.53E-71 & 24.0 && 8.53E-71 & 24.0\\
&& 240 && 5.10E-78 & 24.0 && 5.10E-78 & 24.0 && 5.10E-78 & 24.0 && 5.10E-78 & 24.0 && 5.10E-78 & 24.0\\
\midrule
\multirow{3}{*}{$T-1\Delta t$} && 60 && 1.40E-63 & --- && 1.40E-63 & --- && 1.40E-63 & --- && 1.40E-63 & --- && 1.40E-63 & ---\\
&& 120 && 8.46E-71 & 24.0 && 8.46E-71 & 24.0 && 8.46E-71 & 24.0 && 8.46E-71 & 24.0 && 8.46E-71 & 24.0\\
&& 240 && 5.08E-78 & 24.0 && 5.08E-78 & 24.0 && 5.08E-78 & 24.0 && 5.08E-78 & 24.0 && 5.08E-78 & 24.0\\
\midrule
\multirow{3}{*}{$T-0\Delta t$} && 60 && 1.44E-63 & --- && 1.44E-63 & --- && 1.44E-63 & --- && 1.44E-63 & --- && 1.44E-63 & ---\\
&& 120 && 8.57E-71 & 24.0 && 8.57E-71 & 24.0 && 8.57E-71 & 24.0 && 8.57E-71 & 24.0 && 8.57E-71 & 24.0\\
&& 240 && 5.11E-78 & 24.0 && 5.11E-78 & 24.0 && 5.11E-78 & 24.0 && 5.11E-78 & 24.0 && 5.11E-78 & 24.0\\
\bottomrule
\end{tabular}
\end{table}

\subsection{Non-linear scalar ODE}

Non-linear scalar ODEs with stiff and sensitive solutions are addressed in this section to assess the ability of the structural schemes to provide accurate numerical approximations. Non-linearity may also lead to blow-up in finite time, and the performance of the scheme depends on its ability to capture the solution very close to the final time.

\subsubsection{High sensitivity to the initial condition}

Consider the non-linear scalar ODE $\phi'(t)=\phi(t)(1-\phi(t))$ in $(-10,0]$ with the initial condition $\phi(-10)=1/(1+\exp(10))$. The solution is the sigmoid function $\displaystyle \phi(t)=\sigma(t)=1/(1+\exp(-t))$, which presents two horizontal asymptotes at $-\infty$ and $+\infty$. The main issue is that small perturbations around the initial condition, for instance, a low-accuracy scheme, can lead to large deviations in the numerical approximation at $t=0$. To assess the ability of the structural schemes to handle such a stiff initial condition, we run the benchmark for different values of $K$ and $R$ and report in Table~\ref{ODE3b_merge_K1K2K3_RN} the errors and convergence orders.

\begin{table}[H]
\centering
\caption{Errors and convergence orders for the non-linear scalar ODE benchmark with high sensitivity to the initial condition.}
\label{ODE3b_merge_K1K2K3_RN}
\begin{tabular}{@{}r@{}r@{}r@{}r@{}rr@{}r@{}rr@{}r@{}rr@{}r@{}rr@{}r@{}rr@{}r@{}rr@{}}
\toprule
R & \phantom{aa} & N & \phantom{aa} &
\multicolumn{8}{c}{K=1} & \phantom{aa} &
\multicolumn{5}{c}{K=2} & \phantom{aa} &
\multicolumn{2}{c}{K=3}\\
\cmidrule{5-12}\cmidrule{14-18}\cmidrule{20-21}
&&&& \multicolumn{2}{@{}l@{}}{$\phi$} & \phantom{aa} &
\multicolumn{2}{@{}l@{}}{$\phi^{(2)}$} & \phantom{aa} &
\multicolumn{2}{@{}l@{}}{$\phi^{(3)}$} &
& \multicolumn{2}{@{}l@{}}{$\phi$} & \phantom{aa} &
\multicolumn{2}{@{}l@{}}{$\phi^{(3)}$} &
& \multicolumn{2}{@{}l@{}}{$\phi$}\\
\cmidrule{5-6}\cmidrule{8-9}\cmidrule{11-12}\cmidrule{14-15}\cmidrule{17-18}\cmidrule{20-21}
&&&& err & ord & & err & ord & & err & ord & & err & ord & & err & ord & & err & ord\\
\midrule
\multirow{4}{*}{1} & & & &
& & & \multicolumn{2}{r@{}}{\textcolor{blue}{Ip=1}} & & \multicolumn{2}{r@{}}{\textcolor{blue}{Ip=2}}
& & & & & \multicolumn{2}{r@{}}{\textcolor{blue}{Ip=2}}
& & &\\
& & 120 & &
\textcolor{black}{1.01E-03} & \textcolor{black}{\emdash} & &
\textcolor{blue}{4.69E-03} & \textcolor{blue}{\emdash} & &
\textcolor{blue}{7.76E-01} & \textcolor{blue}{\emdash} & &
\textcolor{black}{8.38E-08} & \textcolor{black}{\emdash} & &
\textcolor{blue}{2.55E-04} & \textcolor{blue}{\emdash} & &
\textcolor{black}{2.50E-12} & \textcolor{black}{\emdash}\\
& & 240 & &
\textcolor{black}{2.53E-04} & \textcolor{black}{2.0} & &
\textcolor{blue}{2.48E-03} & \textcolor{blue}{0.9} & &
\textcolor{blue}{7.80E-01} & \textcolor{blue}{\emdash} & &
\textcolor{black}{5.23E-09} & \textcolor{black}{4.0} & &
\textcolor{blue}{6.34E-05} & \textcolor{blue}{2.0} & &
\textcolor{black}{3.90E-14} & \textcolor{black}{6.0}\\
& & 480 & &
\textcolor{black}{6.33E-05} & \textcolor{black}{2.0} & &
\textcolor{blue}{1.27E-03} & \textcolor{blue}{1.0} & &
\textcolor{blue}{7.81E-01} & \textcolor{blue}{\emdash} & &
\textcolor{black}{3.27E-10} & \textcolor{black}{4.0} & &
\textcolor{blue}{1.58E-05} & \textcolor{blue}{2.0} & &
\textcolor{black}{6.10E-16} & \textcolor{black}{6.0}\\
& & 960 & &
\textcolor{black}{1.58E-05} & \textcolor{black}{2.0} & &
\textcolor{blue}{6.43E-04} & \textcolor{blue}{1.0} & &
\textcolor{blue}{7.81E-01} & \textcolor{blue}{\emdash} & &
\textcolor{black}{2.04E-11} & \textcolor{black}{4.0} & &
\textcolor{blue}{3.96E-06} & \textcolor{blue}{2.0} & &
\textcolor{black}{9.52E-18} & \textcolor{black}{6.0}\\
\midrule
\multirow{4}{*}{2} & & & &
& & & \multicolumn{2}{r@{}}{\textcolor{blue}{Ip=2}} & & \multicolumn{2}{r@{}}{\textcolor{blue}{Ip=3}}
& & & & & \multicolumn{2}{r@{}}{\textcolor{blue}{Ip=2}}
& & &\\
& & 120 & &
\textcolor{black}{1.68E-06} & \textcolor{black}{\emdash} & &
\textcolor{blue}{3.68E-05} & \textcolor{blue}{\emdash} & &
\textcolor{blue}{9.45E-04} & \textcolor{blue}{\emdash} & &
\textcolor{black}{2.67E-11} & \textcolor{black}{\emdash} & &
\textcolor{blue}{5.57E-07} & \textcolor{blue}{\emdash} & &
\textcolor{black}{2.18E-20} & \textcolor{black}{\emdash}\\
& & 240 & &
\textcolor{black}{1.05E-07} & \textcolor{black}{4.0} & &
\textcolor{blue}{4.57E-06} & \textcolor{blue}{3.0} & &
\textcolor{blue}{2.01E-04} & \textcolor{blue}{2.2} & &
\textcolor{black}{4.16E-13} & \textcolor{black}{6.0} & &
\textcolor{blue}{3.54E-08} & \textcolor{blue}{4.0} & &
\textcolor{black}{2.13E-23} & \textcolor{black}{10.0}\\
& & 480 & &
\textcolor{black}{6.54E-09} & \textcolor{black}{4.0} & &
\textcolor{blue}{5.68E-07} & \textcolor{blue}{3.0} & &
\textcolor{blue}{4.80E-05} & \textcolor{blue}{2.1} & &
\textcolor{black}{6.50E-15} & \textcolor{black}{6.0} & &
\textcolor{blue}{2.22E-09} & \textcolor{blue}{4.0} & &
\textcolor{black}{2.08E-26} & \textcolor{black}{10.0}\\
& & 960 & &
\textcolor{black}{4.09E-10} & \textcolor{black}{4.0} & &
\textcolor{blue}{7.08E-08} & \textcolor{blue}{3.0} & &
\textcolor{blue}{1.19E-05} & \textcolor{blue}{2.0} & &
\textcolor{black}{1.02E-16} & \textcolor{black}{6.0} & &
\textcolor{blue}{1.39E-10} & \textcolor{blue}{4.0} & &
\textcolor{black}{2.04E-29} & \textcolor{black}{10.0}\\
\midrule
\multirow{4}{*}{3} & & & &
& & & \multicolumn{2}{r@{}}{\textcolor{blue}{Ip=2}} & & \multicolumn{2}{r@{}}{\textcolor{blue}{Ip=3}}
& & & & & \multicolumn{2}{r@{}}{\textcolor{blue}{Ip=3}}
& & &\\
& & 120 & &
\textcolor{black}{7.59E-07} & \textcolor{black}{\emdash} & &
\textcolor{blue}{3.56E-05} & \textcolor{blue}{\emdash} & &
\textcolor{blue}{1.56E-03} & \textcolor{blue}{\emdash} & &
\textcolor{black}{5.71E-15} & \textcolor{black}{\emdash} & &
\textcolor{blue}{4.30E-09} & \textcolor{blue}{\emdash} & &
\textcolor{black}{3.94E-23} & \textcolor{black}{\emdash}\\
& & 240 & &
\textcolor{black}{4.72E-08} & \textcolor{black}{4.0} & &
\textcolor{blue}{4.52E-06} & \textcolor{blue}{3.0} & &
\textcolor{blue}{3.96E-04} & \textcolor{blue}{2.0} & &
\textcolor{black}{2.23E-17} & \textcolor{black}{8.0} & &
\textcolor{blue}{7.11E-11} & \textcolor{blue}{5.9} & &
\textcolor{black}{9.59E-27} & \textcolor{black}{12.0}\\
& & 480 & &
\textcolor{black}{2.95E-09} & \textcolor{black}{4.0} & &
\textcolor{blue}{5.66E-07} & \textcolor{blue}{3.0} & &
\textcolor{blue}{9.93E-05} & \textcolor{blue}{2.0} & &
\textcolor{black}{8.69E-20} & \textcolor{black}{8.0} & &
\textcolor{blue}{1.13E-12} & \textcolor{blue}{6.0} & &
\textcolor{black}{2.34E-30} & \textcolor{black}{12.0}\\
& & 960 & &
\textcolor{black}{1.84E-10} & \textcolor{black}{4.0} & &
\textcolor{blue}{7.07E-08} & \textcolor{blue}{3.0} & &
\textcolor{blue}{2.49E-05} & \textcolor{blue}{2.0} & &
\textcolor{black}{3.39E-22} & \textcolor{black}{8.0} & &
\textcolor{blue}{1.77E-14} & \textcolor{blue}{6.0} & &
\textcolor{black}{5.71E-34} & \textcolor{black}{12.0}\\
\midrule
\multirow{4}{*}{4} & & & &
& & & \multicolumn{2}{r@{}}{\textcolor{blue}{Ip=3}} & & \multicolumn{2}{r@{}}{\textcolor{blue}{Ip=4}}
& & & & & \multicolumn{2}{r@{}}{\textcolor{blue}{Ip=4}}
& & &\\
& & 120 & &
\textcolor{black}{4.45E-09} & \textcolor{black}{\emdash} & &
\textcolor{blue}{1.37E-06} & \textcolor{blue}{\emdash} & &
\textcolor{blue}{6.57E-05} & \textcolor{blue}{\emdash} & &
\textcolor{black}{3.76E-18} & \textcolor{black}{\emdash} & &
\textcolor{blue}{5.22E-11} & \textcolor{blue}{\emdash} & &
\textcolor{black}{5.08E-29} & \textcolor{black}{\emdash}\\
& & 240 & &
\textcolor{black}{6.93E-11} & \textcolor{black}{6.0} & &
\textcolor{blue}{4.41E-08} & \textcolor{blue}{5.0} & &
\textcolor{blue}{4.20E-06} & \textcolor{blue}{4.0} & &
\textcolor{black}{3.65E-21} & \textcolor{black}{10.0} & &
\textcolor{blue}{2.37E-13} & \textcolor{blue}{7.8} & &
\textcolor{black}{7.72E-34} & \textcolor{black}{16.0}\\
& & 480 & &
\textcolor{black}{1.08E-12} & \textcolor{black}{6.0} & &
\textcolor{blue}{1.39E-09} & \textcolor{blue}{5.0} & &
\textcolor{blue}{2.64E-07} & \textcolor{blue}{4.0} & &
\textcolor{black}{3.56E-24} & \textcolor{black}{10.0} & &
\textcolor{blue}{9.61E-16} & \textcolor{blue}{7.9} & &
\textcolor{black}{1.18E-38} & \textcolor{black}{16.0}\\
& & 960 & &
\textcolor{black}{1.69E-14} & \textcolor{black}{6.0} & &
\textcolor{blue}{4.34E-11} & \textcolor{blue}{5.0} & &
\textcolor{blue}{1.65E-08} & \textcolor{blue}{4.0} & &
\textcolor{black}{3.47E-27} & \textcolor{black}{10.0} & &
\textcolor{blue}{3.79E-18} & \textcolor{blue}{8.0} & &
\textcolor{black}{1.80E-43} & \textcolor{black}{16.0}\\
\midrule
\multirow{4}{*}{5} & & & &
& & & \multicolumn{2}{r@{}}{\textcolor{blue}{Ip=3}} & & \multicolumn{2}{r@{}}{\textcolor{blue}{Ip=4}}
& & & & & \multicolumn{2}{r@{}}{\textcolor{blue}{Ip=4}}
& & &\\
& & 120 & &
\textcolor{black}{1.17E-09} & \textcolor{black}{\emdash} & &
\textcolor{blue}{6.24E-07} & \textcolor{blue}{\emdash} & &
\textcolor{blue}{3.44E-05} & \textcolor{blue}{\emdash} & &
\textcolor{black}{9.40E-21} & \textcolor{black}{\emdash} & &
\textcolor{blue}{8.34E-13} & \textcolor{blue}{\emdash} & &
\textcolor{black}{5.54E-31} & \textcolor{black}{\emdash}\\
& & 240 & &
\textcolor{black}{1.80E-11} & \textcolor{black}{6.0} & &
\textcolor{blue}{2.15E-08} & \textcolor{blue}{4.9} & &
\textcolor{blue}{2.36E-06} & \textcolor{blue}{3.9} & &
\textcolor{black}{2.27E-24} & \textcolor{black}{12.0} & &
\textcolor{blue}{1.17E-15} & \textcolor{blue}{9.5} & &
\textcolor{black}{5.70E-38} & \textcolor{black}{23.2}\\
& & 480 & &
\textcolor{black}{2.80E-13} & \textcolor{black}{6.0} & &
\textcolor{blue}{6.89E-10} & \textcolor{blue}{5.0} & &
\textcolor{blue}{1.51E-07} & \textcolor{blue}{4.0} & &
\textcolor{black}{5.52E-28} & \textcolor{black}{12.0} & &
\textcolor{blue}{1.23E-18} & \textcolor{blue}{9.9} & &
\textcolor{black}{2.17E-43} & \textcolor{black}{18.0}\\
& & 960 & &
\textcolor{black}{4.37E-15} & \textcolor{black}{6.0} & &
\textcolor{blue}{2.17E-11} & \textcolor{blue}{5.0} & &
\textcolor{blue}{9.50E-09} & \textcolor{blue}{4.0} & &
\textcolor{black}{1.35E-31} & \textcolor{black}{12.0} & &
\textcolor{blue}{1.23E-21} & \textcolor{blue}{10.0} & &
\textcolor{black}{8.28E-49} & \textcolor{black}{18.0}\\
\bottomrule
\end{tabular}
\end{table}

Following the same colour scheme as before, the errors and convergence orders in blue correspond to derivative approximations obtained from post-processing structural equations. In black, only the errors and convergence orders for the function are provided, since the derivatives associated with the physical equations have the same convergence order by construction of the structural schemes. In that regard, the errors exhibit the expected convergence orders, and the structural schemes with higher convergence orders prevent the numerical approximation from degrading at the final time. Experiments with a steeper solution $\sigma(\lambda t)$ with $\lambda=10,20$, which are not provided for compactness, further support the ability of the structural schemes to handle low-conditioning problems.

\subsubsection{Finite time solution with blow-up}

Another class of challenging problems concerns solutions that end at a finite time $T$ with a blow-up, that is, $\lim_{t\to T}\phi(t)=+\infty$. As an example, let us consider the ODE $\phi'=\phi^2$ in $(0,1)$ with the initial condition $\phi(0)=1$, whose solution is $\phi(t)=1/(1-t)$. We compute the numerical solution up to $t=0.95$, whose exact value is $\phi(0.95)=20$. Table~\ref{ODE3c_merge_K1K2K3_RN} reports the errors and convergence orders for different values of $K$ and $R$, choosing $I_p=I_{p,\min}$ to complete the table with the higher derivatives, determined from post-processing structural equations.

\begin{table}[H]
\centering
\caption{Errors and convergence orders for the non-linear scalar ODE benchmark with a blow-up solution at $t=0.95$.}
\label{ODE3c_merge_K1K2K3_RN}
\begin{tabular}{@{}r@{}r@{}r@{}r@{}rr@{}r@{}rr@{}r@{}rr@{}r@{}rr@{}r@{}rr@{}r@{}rr@{}}
\toprule
R & \phantom{aa} & N & \phantom{aa} &
\multicolumn{8}{c}{K=1} & \phantom{aa} &
\multicolumn{5}{c}{K=2} & \phantom{aa} &
\multicolumn{2}{c}{K=3}\\
\cmidrule{5-12}\cmidrule{14-18}\cmidrule{20-21}
&&&& \multicolumn{2}{@{}l@{}}{$\phi$} & \phantom{aa} &
\multicolumn{2}{@{}l@{}}{$\phi^{(2)}$} & \phantom{aa} &
\multicolumn{2}{@{}l@{}}{$\phi^{(3)}$} &
& \multicolumn{2}{@{}l@{}}{$\phi$} & \phantom{aa} &
\multicolumn{2}{@{}l@{}}{$\phi^{(3)}$} &
& \multicolumn{2}{@{}l@{}}{$\phi$}\\
\cmidrule{5-6}\cmidrule{8-9}\cmidrule{11-12}\cmidrule{14-15}\cmidrule{17-18}\cmidrule{20-21}
&&&& err & ord & & err & ord & & err & ord & & err & ord & & err & ord & & err & ord\\
\midrule
\multirow{5}{*}{1} & & & &
& & & \multicolumn{2}{r@{}}{\textcolor{blue}{Ip=1}} & & \multicolumn{2}{r@{}}{\textcolor{blue}{Ip=2}}
& & & & & \multicolumn{2}{r@{}}{\textcolor{blue}{Ip=2}}
& & &\\
& & 120 & &
\textcolor{black}{2.45E-01} & \textcolor{black}{\emdash} & &
\textcolor{blue}{2.56E+03} & \textcolor{blue}{\emdash} & &
\textcolor{blue}{4.04E+06} & \textcolor{blue}{\emdash} & &
\textcolor{black}{7.00E-04} & \textcolor{black}{\emdash} & &
\textcolor{blue}{5.65E+04} & \textcolor{blue}{\emdash} & &
\textcolor{black}{3.12E-06} & \textcolor{black}{\emdash}\\
& & 240 & &
\textcolor{black}{6.00E-02} & \textcolor{black}{2.0} & &
\textcolor{blue}{1.56E+03} & \textcolor{blue}{0.7} & &
\textcolor{blue}{4.78E+06} & \textcolor{blue}{\emdash} & &
\textcolor{black}{4.37E-05} & \textcolor{black}{4.0} & &
\textcolor{blue}{1.62E+04} & \textcolor{blue}{1.8} & &
\textcolor{black}{4.91E-08} & \textcolor{black}{6.0}\\
& & 480 & &
\textcolor{black}{1.49E-02} & \textcolor{black}{2.0} & &
\textcolor{blue}{8.58E+02} & \textcolor{blue}{0.9} & &
\textcolor{blue}{5.31E+06} & \textcolor{blue}{\emdash} & &
\textcolor{black}{2.73E-06} & \textcolor{black}{4.0} & &
\textcolor{blue}{4.26E+03} & \textcolor{blue}{1.9} & &
\textcolor{black}{7.69E-10} & \textcolor{black}{6.0}\\
& & 960 & &
\textcolor{black}{3.72E-03} & \textcolor{black}{2.0} & &
\textcolor{blue}{4.51E+02} & \textcolor{blue}{0.9} & &
\textcolor{blue}{5.63E+06} & \textcolor{blue}{\emdash} & &
\textcolor{black}{1.70E-07} & \textcolor{black}{4.0} & &
\textcolor{blue}{1.08E+03} & \textcolor{blue}{2.0} & &
\textcolor{black}{1.20E-11} & \textcolor{black}{6.0}\\
\midrule
\multirow{5}{*}{2} & & & &
& & & \multicolumn{2}{r@{}}{\textcolor{blue}{Ip=2}} & & \multicolumn{2}{r@{}}{\textcolor{blue}{Ip=3}}
& & & & & \multicolumn{2}{r@{}}{\textcolor{blue}{Ip=2}}
& & &\\
& & 120 & &
\textcolor{black}{2.75E-03} & \textcolor{black}{\emdash} & &
\textcolor{blue}{9.07E+02} & \textcolor{blue}{\emdash} & &
\textcolor{blue}{3.64E+05} & \textcolor{blue}{\emdash} & &
\textcolor{black}{3.15E-05} & \textcolor{black}{\emdash} & &
\textcolor{blue}{2.09E+03} & \textcolor{blue}{\emdash} & &
\textcolor{black}{2.60E-09} & \textcolor{black}{\emdash}\\
& & 240 & &
\textcolor{black}{1.74E-04} & \textcolor{black}{4.0} & &
\textcolor{blue}{2.95E+02} & \textcolor{blue}{1.6} & &
\textcolor{blue}{2.76E+05} & \textcolor{blue}{0.4} & &
\textcolor{black}{5.16E-07} & \textcolor{black}{5.9} & &
\textcolor{blue}{2.08E+02} & \textcolor{blue}{3.3} & &
\textcolor{black}{3.05E-12} & \textcolor{black}{9.7}\\
& & 480 & &
\textcolor{black}{1.09E-05} & \textcolor{black}{4.0} & &
\textcolor{blue}{8.54E+01} & \textcolor{blue}{1.8} & &
\textcolor{blue}{1.62E+05} & \textcolor{blue}{0.8} & &
\textcolor{black}{8.17E-09} & \textcolor{black}{6.0} & &
\textcolor{blue}{1.67E+01} & \textcolor{blue}{3.6} & &
\textcolor{black}{3.14E-15} & \textcolor{black}{9.9}\\
& & 960 & &
\textcolor{black}{6.82E-07} & \textcolor{black}{4.0} & &
\textcolor{blue}{2.31E+01} & \textcolor{blue}{1.9} & &
\textcolor{blue}{8.55E+04} & \textcolor{blue}{0.9} & &
\textcolor{black}{1.28E-10} & \textcolor{black}{6.0} & &
\textcolor{blue}{1.20E+00} & \textcolor{blue}{3.8} & &
\textcolor{black}{3.11E-18} & \textcolor{black}{10.0}\\
\midrule
\multirow{5}{*}{3} & & & &
& & & \multicolumn{2}{r@{}}{\textcolor{blue}{Ip=2}} & & \multicolumn{2}{r@{}}{\textcolor{blue}{Ip=3}}
& & & & & \multicolumn{2}{r@{}}{\textcolor{blue}{Ip=3}}
& & &\\
& & 120 & &
\textcolor{black}{6.09E-03} & \textcolor{black}{\emdash} & &
\textcolor{blue}{3.27E+02} & \textcolor{blue}{\emdash} & &
\textcolor{blue}{1.81E+05} & \textcolor{blue}{\emdash} & &
\textcolor{black}{2.87E-06} & \textcolor{black}{\emdash} & &
\textcolor{blue}{2.62E+02} & \textcolor{blue}{\emdash} & &
\textcolor{black}{7.19E-10} & \textcolor{black}{\emdash}\\
& & 240 & &
\textcolor{black}{3.89E-04} & \textcolor{black}{4.0} & &
\textcolor{blue}{6.71E+01} & \textcolor{blue}{2.3} & &
\textcolor{blue}{6.78E+04} & \textcolor{blue}{1.4} & &
\textcolor{black}{1.40E-08} & \textcolor{black}{7.7} & &
\textcolor{blue}{9.50E+00} & \textcolor{blue}{4.8} & &
\textcolor{black}{2.90E-13} & \textcolor{black}{11.3}\\
& & 480 & &
\textcolor{black}{2.45E-05} & \textcolor{black}{4.0} & &
\textcolor{blue}{1.10E+01} & \textcolor{blue}{2.6} & &
\textcolor{blue}{2.14E+04} & \textcolor{blue}{1.7} & &
\textcolor{black}{5.84E-11} & \textcolor{black}{7.9} & &
\textcolor{blue}{2.38E-01} & \textcolor{blue}{5.3} & &
\textcolor{black}{8.37E-17} & \textcolor{black}{11.8}\\
& & 960 & &
\textcolor{black}{1.53E-06} & \textcolor{black}{4.0} & &
\textcolor{blue}{1.60E+00} & \textcolor{blue}{2.8} & &
\textcolor{blue}{6.05E+03} & \textcolor{blue}{1.8} & &
\textcolor{black}{2.32E-13} & \textcolor{black}{8.0} & &
\textcolor{blue}{4.79E-03} & \textcolor{blue}{5.6} & &
\textcolor{black}{2.14E-20} & \textcolor{black}{11.9}\\
\midrule
\multirow{5}{*}{4} & & & &
& & & \multicolumn{2}{r@{}}{\textcolor{blue}{Ip=3}} & & \multicolumn{2}{r@{}}{\textcolor{blue}{Ip=4}}
& & & & & \multicolumn{2}{r@{}}{\textcolor{blue}{Ip=4}}
& & &\\
& & 120 & &
\textcolor{black}{5.12E-04} & \textcolor{black}{\emdash} & &
\textcolor{blue}{1.51E+02} & \textcolor{blue}{\emdash} & &
\textcolor{blue}{8.97E+04} & \textcolor{blue}{\emdash} & &
\textcolor{black}{4.00E-07} & \textcolor{black}{\emdash} & &
\textcolor{blue}{4.57E+01} & \textcolor{blue}{\emdash} & &
\textcolor{black}{8.18E-12} & \textcolor{black}{\emdash}\\
& & 240 & &
\textcolor{black}{9.68E-06} & \textcolor{black}{5.7} & &
\textcolor{blue}{1.90E+01} & \textcolor{blue}{3.0} & &
\textcolor{blue}{2.14E+04} & \textcolor{blue}{2.1} & &
\textcolor{black}{6.70E-10} & \textcolor{black}{9.2} & &
\textcolor{blue}{6.50E-01} & \textcolor{blue}{6.1} & &
\textcolor{black}{4.15E-16} & \textcolor{black}{14.3}\\
& & 480 & &
\textcolor{black}{1.61E-07} & \textcolor{black}{5.9} & &
\textcolor{blue}{1.79E+00} & \textcolor{blue}{3.4} & &
\textcolor{blue}{3.89E+03} & \textcolor{blue}{2.5} & &
\textcolor{black}{7.92E-13} & \textcolor{black}{9.7} & &
\textcolor{blue}{5.34E-03} & \textcolor{blue}{6.9} & &
\textcolor{black}{1.03E-20} & \textcolor{black}{15.3}\\
& & 960 & &
\textcolor{black}{2.55E-09} & \textcolor{black}{6.0} & &
\textcolor{blue}{1.40E-01} & \textcolor{blue}{3.7} & &
\textcolor{blue}{5.98E+02} & \textcolor{blue}{2.7} & &
\textcolor{black}{8.17E-16} & \textcolor{black}{9.9} & &
\textcolor{blue}{3.12E-05} & \textcolor{blue}{7.4} & &
\textcolor{black}{1.82E-25} & \textcolor{black}{15.8}\\
\midrule
\multirow{5}{*}{5} & & & &
& & & \multicolumn{2}{r@{}}{\textcolor{blue}{Ip=3}} & & \multicolumn{2}{r@{}}{\textcolor{blue}{Ip=4}}
& & & & & \multicolumn{2}{r@{}}{\textcolor{blue}{Ip=4}}
& & &\\
& & 120 & &
\textcolor{black}{9.94E-04} & \textcolor{black}{\emdash} & &
\textcolor{blue}{7.11E+01} & \textcolor{blue}{\emdash} & &
\textcolor{blue}{4.75E+04} & \textcolor{blue}{\emdash} & &
\textcolor{black}{7.48E-08} & \textcolor{black}{\emdash} & &
\textcolor{blue}{1.00E+01} & \textcolor{blue}{\emdash} & &
\textcolor{black}{3.23E-12} & \textcolor{black}{\emdash}\\
& & 240 & &
\textcolor{black}{2.00E-05} & \textcolor{black}{5.6} & &
\textcolor{blue}{6.04E+00} & \textcolor{blue}{3.6} & &
\textcolor{blue}{7.47E+03} & \textcolor{blue}{2.7} & &
\textcolor{black}{4.87E-11} & \textcolor{black}{10.6} & &
\textcolor{blue}{5.99E-02} & \textcolor{blue}{7.4} & &
\textcolor{black}{7.45E-17} & \textcolor{black}{15.4}\\
& & 480 & &
\textcolor{black}{3.41E-07} & \textcolor{black}{5.9} & &
\textcolor{blue}{3.38E-01} & \textcolor{blue}{4.2} & &
\textcolor{blue}{8.06E+02} & \textcolor{blue}{3.2} & &
\textcolor{black}{1.77E-14} & \textcolor{black}{11.4} & &
\textcolor{blue}{1.69E-04} & \textcolor{blue}{8.5} & &
\textcolor{black}{6.41E-22} & \textcolor{black}{16.8}\\
& & 960 & &
\textcolor{black}{5.46E-09} & \textcolor{black}{6.0} & &
\textcolor{blue}{1.45E-02} & \textcolor{blue}{4.5} & &
\textcolor{blue}{6.81E+01} & \textcolor{blue}{3.6} & &
\textcolor{black}{4.90E-18} & \textcolor{black}{11.8} & &
\textcolor{blue}{2.96E-07} & \textcolor{blue}{9.2} & &
\textcolor{black}{3.25E-27} & \textcolor{black}{17.6}\\
\bottomrule
\end{tabular}
\end{table}

The convergence orders are perfectly in line with the expected results for both the solution and derivative approximations. Experiments considering the numerical approximations at $t=0.995$, whose exact solution is $\phi(0.995)=200$, and which are not provided for compactness, provide the same convergence orders.

\subsection{Linear vectorial ODE}

Differential system involving several variables coupled via ODEs is a fundamental topic for several reasons: the splitting of a second- or higher-order differential scalar equation into a first-order differential system, Hamiltonian problems, which lead to a set of first-order differential equations combining the primal and dual variables, and semi-discretisation in space of non-stationary PDEs, which provides large (non-)linear systems to numerically solve. We explore the basic case of the linear wave propagation as a representative example, where we split the wave equation into the first-order differential system
\[
\left\{
\begin{array}{l}
\phi'(t)=\alpha\psi(t),\\
\psi'(t)=-\alpha\phi(t),
\end{array}
\right.
\]
in $(0,1]$ with the initial conditions $\phi(0)=1$ and $\psi(0)=0$. The solutions are simply $\phi(t)=\cos(\alpha t)$ and $\psi(t)=-\sin(\alpha t)$, where $\alpha$ is the angular frequency.

\begin{table}[H]
\centering
\caption{Errors and convergence orders for $\phi$ the linear vectorial ODE benchmark.}
\label{ODE4_SemPre_merge_K1K2K3_phi}
\begin{tabular}{@{}r@{}r@{}r@{}r@{}rr@{}r@{}rr@{}r@{}rr@{}}
\toprule
R & \phantom{aa} & N & \phantom{aa} &
\multicolumn{2}{c}{K=1} & \phantom{aa} &
\multicolumn{2}{c}{K=2} & \phantom{aa} &
\multicolumn{2}{c}{K=3}\\
\cmidrule{5-6}\cmidrule{8-9}\cmidrule{11-12}
&&&& err & ord && err & ord && err & ord\\
\midrule
\multirow{3}{*}{1} && 60 && 2.03E-03 & \emdash && 4.14E-07 & \emdash && 3.57E-11 & \emdash\\
&& 120 && 5.12E-04 & 2.0 && 2.59E-08 & 4.0 && 5.58E-13 & 6.0\\
&& 240 && 1.28E-04 & 2.0 && 1.62E-09 & 4.0 && 8.73E-15 & 6.0\\
\midrule
\multirow{3}{*}{2} && 60 && 6.60E-06 & \emdash && 3.80E-10 & \emdash && 2.01E-17 & \emdash\\
&& 120 && 4.14E-07 & 4.0 && 5.95E-12 & 6.0 && 1.96E-20 & 10.0\\
&& 240 && 2.59E-08 & 4.0 && 9.31E-14 & 6.0 && 1.92E-23 & 10.0\\
\midrule
\multirow{3}{*}{3} && 60 && 3.68E-06 & \emdash && 4.14E-13 & \emdash && 3.36E-21 & \emdash\\
&& 120 && 2.32E-07 & 4.0 && 1.62E-15 & 8.0 && 8.25E-25 & 12.0\\
&& 240 && 1.45E-08 & 4.0 && 6.36E-18 & 8.0 && 2.02E-28 & 12.0\\
\midrule
\multirow{3}{*}{4} && 60 && 4.53E-08 & \emdash && 5.28E-16 & \emdash && 1.09E-26 & \emdash\\
&& 120 && 7.13E-10 & 6.0 && 5.22E-19 & 10.0 && 1.66E-31 & 16.0\\
&& 240 && 1.12E-11 & 6.0 && 5.11E-22 & 10.0 && 2.54E-36 & 16.0\\
\midrule
\multirow{3}{*}{5} && 60 && 1.58E-08 & \emdash && 7.78E-19 & \emdash && 1.28E-30 & \emdash\\
&& 120 && 2.54E-10 & 6.0 && 1.94E-22 & 12.0 && 4.97E-36 & 18.0\\
&& 240 && 3.99E-12 & 6.0 && 4.76E-26 & 12.0 && 1.91E-41 & 18.0\\
\bottomrule
\end{tabular}
\end{table}

We consider $\alpha=2.1\pi$ to avoid the situation where the solution at the final time does not correspond to an extremum, which might produce super-convergence effects. The errors and convergence orders for different values of $K$ and $R$ are reported in Table~\ref{ODE4_SemPre_merge_K1K2K3_phi}. As observed, the expected convergence orders are obtained, following the same behaviour as in the scalar case, and the coupling between $\phi(t)$ and $\psi(t)$ does not cause any additional difficulty. Experiments considering larger values for $\alpha$, which are not provided for compactness, reproduced the same convergence orders and stability.

Another point concerns the invariant conservation. Indeed, the quantity $m(t)=\phi^2(t)+\psi^2(t)$ has to be preserved throughout the simulation. Let us denote as $m_{n+r}$ the approximation at time $t_{n+r}$ and, for any structural scheme $\SK{K}{R}$, for which the invariant error is given as
\[
e_m=\max_{i=0,\ldots,N/R}|m_{iR+r}-1|,\quad r\bint{1,R}.
\]
In other words, $r=1,\ldots,R-1$ are the intermediate time steps, whereas $r=R$ is the final time step within each block.

\begin{table}[ht]
\centering
\caption{Invariant errors for the linear vectorial ODE benchmark with $K=1,2,3$ and $R=4$.}
\label{ODE4_invariant}
\begin{tabular}{@{}rrcccc@{}}
\toprule
K & N & r=1 & r=2 & r=3 & r=4\\
\midrule
\multirow{3}{*}{1} & 60 & 6.51E-08 & 3.82E-08 & 6.51E-08 & \textcolor{blue}{1.47E-61}\\
& 120 & 1.03E-09 & 6.09E-10 & 1.03E-09 & \textcolor{blue}{9.45E-62}\\
& 240 & 1.62E-11 & 9.57E-12 & 1.62E-11 & \textcolor{blue}{2.96E-62}\\
\midrule
\multirow{3}{*}{2} & 60 & 1.73E-17 & 1.82E-17 & 1.73E-17 & \textcolor{blue}{1.35E-60}\\
& 120 & 4.25E-21 & 4.48E-21 & 4.25E-21 & \textcolor{blue}{3.30E-61}\\
& 240 & 1.04E-24 & 1.09E-24 & 1.04E-24 & \textcolor{blue}{7.67E-61}\\
\midrule
\multirow{3}{*}{3} & 60 & 1.34E-26 & 1.25E-26 & 1.34E-26 & \textcolor{blue}{9.47E-61}\\
& 120 & 2.07E-31 & 1.94E-31 & 2.07E-31 & \textcolor{blue}{3.80E-61}\\
& 240 & 3.17E-36 & 2.97E-36 & 3.17E-36 & \textcolor{blue}{1.77E-61}\\
\bottomrule
\end{tabular}
\end{table}

The invariant errors for $K=1,2,3$ and $R=4$ are reported in Table~\ref{ODE4_invariant} using octuple-precision in the {\tt Julia} framework. As observed, the invariant preservation is only achieved with full precision at the last step of each block ($R=4$, in blue), whereas the invariant errors of the intermediate steps are of the same convergence order as that of the solution and derivatives. Indeed, for $R=4$, the invariant errors observed correspond to the machine truncation for octuple-precision numbers. Similar experiments have been carried out with different structural schemes, not provided for compactness, which show the same behaviour.

\subsection{The Van der Pol equation}

The van der Pol equation enables the modelling of oscillatory non-linear systems in many scientific fields and is a standard benchmark for assessing numerical schemes~\cite{G72,XA13}. Indeed, the solution exhibits several regimes (almost constant in some intervals, with rapid changes) that are challenging from a computational perspective. We consider the original Van der Pol equation
$x''(t)-\mu(1-x^2(t))x'(t)+x(t)=0$ in $(0,T]$, where $x$ is the position, with the initial conditions $x(0)=0$ and $x'(0)=0$, which is rewritten as a first-order differential system, given as
\[
\left\{
\begin{array}{l}
x'(t)=y(t),\\
y'(t)=\mu(1-x^2(t))y(t)-x(t),
\end{array}
\right.
\]
where the parameter $\mu$ controls the stiffness of the solution.
\subsubsection{Convergence assessemnt}
We consider two situations: a smooth transition with $\mu=1$ and a sharp transition with $\mu=50$. Following the reference~\cite{AMORE2022133279}, we choose the initial conditions corresponding to the maximum of the solution and evaluate the numerical approximations after a complete cycle, for which the position, velocity, and period with 76 digits are retrieved from the same reference, and read
\begin{itemize}
\item For $\mu=1$:

{\smalltablepolice
\hspace{0.5cm}$\scriptstyle x(0)=$2.008619860 8748431365 0964018836 2640366192 0261377207 1446112770 3247615558 726585,\\
\hspace*{0.5cm}$\scriptstyle y(0)=$0,\\
\hspace*{0.5cm}$\scriptstyle T=$6.663286859 3231301896 9968203048 2328706812 6463168838 7665511486 3182080925 864845.
}
\item For $\mu=50$:

{\smalltablepolice
\hspace{0.5cm}$\scriptstyle x(0)=$2.002955933 7576484028 2610971894 5871764979 9147201460 6949163926 8919228280 941255,\\
\hspace*{0.5cm}$\scriptstyle y(0)=$0,\\
\hspace*{0.5cm}$\scriptstyle T=$82.50833389 3230782186 8333686974 9987459374 0590858618 6277867538 5441021240 035171.
}
\end{itemize}

The errors and convergence orders for the position are reported in Table~\ref{vdPolmu1} with $\mu=1$, while Figure~\ref{fig:vdPolmu1} illustrates the position and the velocity for a complete cycle for $K=1$, $R=1$, and $N=480$. The velocity errors are not reported because they exhibit the same behaviour as the position errors. We first observe that the expected convergence orders are achieved, further supporting the excellent efficiency of the structural schemes. Table~\ref{vdPolmu1} also provides the average number of times the physical equations are evaluated per time step, $\bar\tau$, required to achieve the convergence of the fixed-point.

\begin{table}[H]
\centering
\caption{Position errors and convergence order for the van der Pol equations benchmark with $\mu=1$.}
\label{vdPolmu1}
\begin{tabular}{@{}r@{}r@{}r@{} r
@{}rrr@{} r
@{}rrr@{} r
@{}rrr@{}}
\toprule
R & \phantom{aa} & N & \phantom{aa} &
\multicolumn{3}{c}{K=1} & \phantom{aa} &
\multicolumn{3}{c}{K=2} & \phantom{aa} &
\multicolumn{3}{c}{K=3}\\
\cmidrule{5-7}\cmidrule{9-11}\cmidrule{13-15}
&&&& $\bar\tau$ & err & ord && $\bar\tau$ & err & ord && $\bar\tau$ & err & ord\\
\midrule
\multirow{3}{*}{1} && 480 && 20.85 & 2.88E-05 & \emdash && 39.94 & 3.10E-09 & \emdash && 88.74 & 8.37E-14 & \emdash\\
&& 600 && 19.79 & 1.85E-05 & 2.0 && 37.80 & 1.27E-09 & 4.0 && 84.06 & 2.20E-14 & 6.0\\
&& 720 && 19.04 & 1.28E-05 & 2.0 && 36.20 & 6.12E-10 & 4.0 && 80.25 & 7.35E-15 & 6.0\\
\midrule
\multirow{3}{*}{2} && 480 && 10.90 & 3.26E-08 & \emdash && 22.00 & 8.93E-13 & \emdash && 45.36 & 2.35E-20 & \emdash\\
&& 600 && 10.33 & 1.34E-08 & 4.0 && 20.76 & 2.34E-13 & 6.0 && 42.72 & 2.52E-21 & 10.0\\
&& 720 && 9.89 & 6.44E-09 & 4.0 && 19.84 & 7.84E-14 & 6.0 && 41.07 & 4.07E-22 & 10.0\\
\midrule
\multirow{3}{*}{3} && 480 && 7.67 & 2.79E-08 & \emdash && 16.10 & 6.39E-16 & \emdash && 31.80 & 1.60E-23 & \emdash\\
&& 600 && 7.25 & 1.14E-08 & 4.0 && 15.08 & 1.07E-16 & 8.0 && 30.03 & 1.10E-24 & 12.0\\
&& 720 && 6.95 & 5.51E-09 & 4.0 && 14.34 & 2.49E-17 & 8.0 && 28.62 & 1.23E-25 & 12.0\\
\midrule
\multirow{3}{*}{4} && 480 && 5.90 & 5.16E-11 & \emdash && 13.30 & 1.66E-18 & \emdash && 35.32 & 1.84E-28 & \emdash\\
&& 600 && 5.59 & 1.36E-11 & 6.0 && 12.44 & 1.78E-19 & 10.0 && 33.20 & 4.91E-30 & 16.2\\
&& 720 && 5.33 & 4.54E-12 & 6.0 && 11.76 & 2.87E-20 & 10.0 && 31.52 & 2.58E-31 & 16.2\\
\midrule
\multirow{3}{*}{5} && 480 && 4.87 & 3.84E-11 & \emdash && 11.98 & 3.76E-21 & \emdash && 24.21 & 9.41E-31 & \emdash\\
&& 600 && 4.59 & 1.01E-11 & 6.0 && 11.04 & 2.59E-22 & 12.0 && 22.44 & 1.68E-32 & 18.0\\
&& 720 && 4.38 & 3.37E-12 & 6.0 && 10.38 & 2.91E-23 & 12.0 && 21.24 & 6.30E-34 & 18.0\\
\bottomrule
\end{tabular}
\end{table}

\begin{figure}[H]
\centering
\begin{tabular}{cc}
\includegraphics[width=0.4\textwidth]{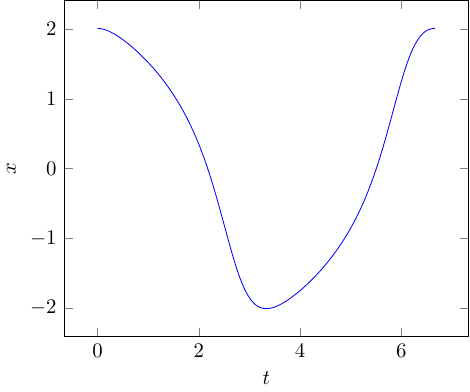} &
\includegraphics[width=0.4\textwidth]{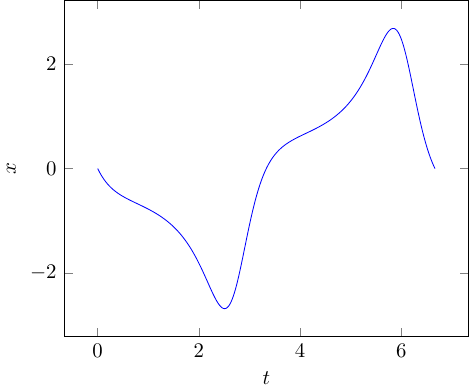}\\
(a) Position. & (b) Velocity.
\end{tabular}
\caption{Approximate solution for the van der Pol equations benchmark with $\mu=1$ for $K=1$, $R=1$, and $N=480$}.
\label{fig:vdPolmu1}
\end{figure}

Table~\ref{vdPolmu50} reports the errors and convergence orders for the position with $\mu=50$ (again, velocity is not reported). Due to the sharp transition, we used a larger number of time grid points to catch the velocity singularity, as shown in Figure~\ref{fig:vndPol50}. We observe a drastic reduction in the parameter $\bar\tau$ compared to the previous case with $\mu=1$, since the smaller $\Delta t$ provides better predictors and allows fixed-point convergence with fewer iterations per time step.

\begin{table}[H]
\centering
\caption{Position errors and convergence orders for the van der Pol equations benchmark with $\mu=50$.}
\label{vdPolmu50}
\begin{tabular}{@{}r@{}r@{}r@{}r@{}rrr@{}r@{}rrr@{}r@{}rrr@{}}
\toprule
R & \phantom{aa} & N & \phantom{aa} &
\multicolumn{3}{c}{K=1} & \phantom{aa} &
\multicolumn{3}{c}{K=2} & \phantom{aa} &
\multicolumn{3}{c}{K=3}\\
\cmidrule{5-7}\cmidrule{9-11}\cmidrule{13-15}
&&&& $\bar\tau$ & err & ord && $\bar\tau$ & err & ord && $\bar\tau$ & err & ord\\
\midrule
\multirow{3}{*}{1} && 840,000 && 4.47 & 1.57E-06 & \emdash && 21.48 & 2.16E-12 & \emdash && 50.73 & 3.10E-18 & \emdash\\
&& 960,000 && 4.17 & 1.21E-06 & 2.0 && 20.94 & 1.26E-12 & 4.0 && 49.20 & 1.39E-18 & 6.0\\
&& 1,080,000 && 4.09 & 9.55E-07 & 2.0 && 20.20 & 7.89E-13 & 4.0 && 48.21 & 6.87E-19 & 6.0\\
\midrule
\multirow{3}{*}{2} && 840,000 && 2.51 & 3.44E-11 & \emdash && 11.62 & 3.31E-17 & \emdash && 25.95 & 2.36E-27 & \emdash\\
&& 960,000 && 2.51 & 2.02E-11 & 4.0 && 11.32 & 1.49E-17 & 6.0 && 25.32 & 6.20E-28 & 10.0\\
&& 1,080,000 && 2.51 & 1.26E-11 & 4.0 && 10.90 & 7.33E-18 & 6.0 && 24.66 & 1.91E-28 & 10.0\\
\midrule
\multirow{3}{*}{3} && 840,000 && 1.68 & 1.94E-11 & \emdash && 8.32 & 1.06E-21 & \emdash && 18.03 & 2.64E-32 & \emdash\\
&& 960,000 && 1.68 & 1.14E-11 & 4.0 && 8.06 & 3.64E-22 & 8.0 && 17.49 & 5.31E-33 & 12.0\\
&& 1,080,000 && 1.67 & 7.10E-12 & 4.0 && 7.82 & 1.42E-22 & 8.0 && 17.04 & 1.29E-33 & 12.0\\
\midrule
\multirow{3}{*}{4} && 840,000 && 1.26 & 3.99E-15 & \emdash && 6.68 & 6.27E-26 & \emdash && 14.64 & 9.02E-40 & \emdash\\
&& 960,000 && 1.26 & 1.79E-15 & 6.0 && 6.48 & 1.65E-26 & 10.0 && 14.16 & 1.07E-40 & 16.0\\
&& 1,080,000 && 1.26 & 8.83E-16 & 6.0 && 6.28 & 5.08E-27 & 10.0 && 13.80 & 1.62E-41 & 16.0\\
\midrule
\multirow{3}{*}{5} && 840,000 && 1.20 & 1.42E-15 & \emdash && 5.78 & 6.21E-30 & \emdash && 12.84 & 1.61E-44 & \emdash\\
&& 960,000 && 1.03 & 6.38E-16 & 6.0 && 5.58 & 1.25E-30 & 12.0 && 12.39 & 1.46E-45 & 18.0\\
&& 1,080,000 && 1.01 & 3.15E-16 & 6.0 && 5.38 & 3.04E-31 & 12.0 && 12.06 & 1.75E-46 & 18.0\\
\bottomrule
\end{tabular}
\end{table}

\begin{figure}[H]
\centering
\begin{tabular}{cc}
\includegraphics[width=0.4\textwidth]{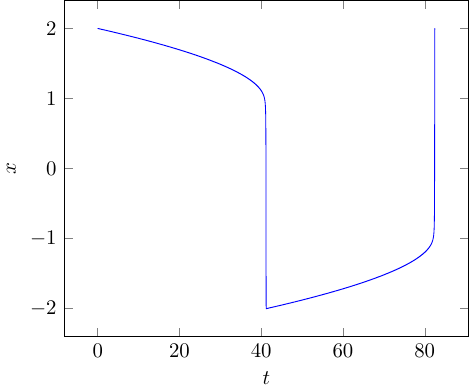} &
\includegraphics[width=0.4\textwidth]{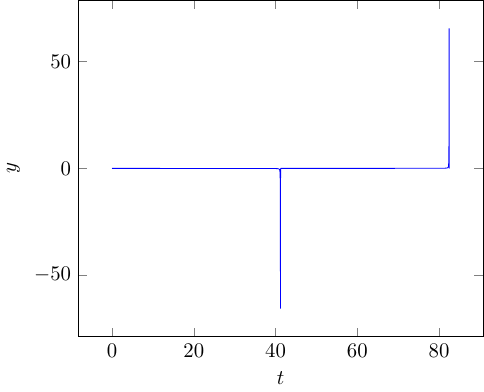}\\
(a) Position. & (b) Velocity.
\end{tabular}
\caption{Approximate solution for the van der Pol equations benchmark with $\mu=50$ for $K=1$, $R=1$, and $N=840000$}.
\label{fig:vndPol50}
\end{figure}

\subsubsection{Variable time step}
\rb{The present benchmark aims to demonstrate that a substantial reduction in computational cost can be achieved through better control of the time step according to the smoothness of the solution. The van der Pol equation with large $\mu$ is an interesting example, since the solution is almost linear over large portions of the time interval, while strong variations occur over small intervals. We use a grid with two different time steps, referred to as a two-step grid: the smooth part is given by the set $\Omega_s=[0,41]\cup[42,79]$, while the peak regions correspond to the domain $[41,42]\cup[79,80]$ {\color{red} Provide the exact values of the intervals}. We use $\Delta t_s$, corresponding to a subdivision with $N_s$ points in the smooth region, whereas $\Delta t_p$ is controlled by $N_p$ in the peak regions. On the other hand, we also consider a uniform time discretisation with time step $\Delta t$ and $N$ points.
}

\begin{table}[H]
\centering
\caption{Comparison of time (in ms) versus error for $K=1$, $K=2$, and $K=3$.}
\begin{tabular}{@{}c r rr rr rr@{}}
\toprule
R & N &
\multicolumn{2}{c}{K=1} &
\multicolumn{2}{c}{K=2} &
\multicolumn{2}{c}{K=3} \\
\cmidrule(lr){3-4}
\cmidrule(lr){5-6}
\cmidrule(lr){7-8}
& {N\textsubscript{s}}/{N\textsubscript{p}} &
time & err &
time & err &
time & err \\
\midrule

\multirow{6}{*}{1}
& 840k   & 878  & 4.6e-04 & 843  & 2.4e-06 & 1075 & 4.9e-09 \\
& 960k   & 956  & 3.5e-04 & 939  & 1.6e-06 & 1274 & 2.9e-09 \\
& 1,080k  & 1216 & 2.8e-04 & 1079 & 1.1e-06 & 1372 & 1.8e-09 \\
& 35k/7,5k   & 73   & 1.8e-03 & 85   & 2.0e-05 & 112  & 8.2e-08 \\
& 40k/10k    & 83   & 1.0e-03 & 98   & 8.4e-06 & 128  & 2.6e-08 \\
& 45k/12,5k  & 91   & 6.8e-04 & 113  & 4.3e-06 & 153  & 1.1e-08 \\

\midrule

\multirow{6}{*}{2}
& 840k   & 3947 & 5.3e-09 & 2571 & 4.9e-15 & 2089 & 1.9e-23 \\
& 960k   & 4300 & 3.1e-09 & 2825 & 2.2e-15 & 2271 & 2.2e-23 \\
& 1,080k  & 3974 & 1.9e-09 & 3065 & 1.1e-15 & 2478 & 2.4e-23 \\
& 35k/7,5k  & 537  & 9.0e-08 & 555  & 4.6e-13 & 538  & 6.7e-20 \\
& 40k/10k    & 571  & 2.9e-08 & 607  & 1.1e-13 & 606  & 1.8e-20 \\
& 45k/12,5k  & 623  & 1.2e-08 & 678  & 4.0e-14 & 599  & 5.5e-21 \\

\midrule

\multirow{6}{*}{4}
& 84k   & 3659 & 5.9e-13 & 3648 & 1.1e-23 & 2993 & 8.9e-23 \\
& 960k   & 3878 & 2.6e-13 & 3986 & 4.0e-24 & 3372 & 1.0e-22 \\
& 1,080k  & 4200 & 1.3e-13 & 4301 & 2.5e-24 & 3412 & 1.1e-22 \\
& 35k/7,5k & 690  & 4.0e-11 & 856  & 1.7e-17 & 826  & 7.9e-24 \\
& 40k/10k & 712  & 6.5e-12 & 922  & 4.5e-18 & 884  & 9.1e-24 \\
& 45k/12,5k  & 767  & 1.5e-12 & 1066 & 1.4e-18 & 969  & 1.0e-23 \\

\bottomrule
\end{tabular}
\end{table}

\rb{
We have selected several choices of $R$ and $N_s/N_p$ that provide similar accuracy errors in order to draw some conclusions. For a given error level, we are able to reduce the running time by a factor between 6 and 8. In the case $R=4$ for the two-step grid, the method reaches the limits of quad-precision arithmetic, and the errors are therefore bounded by the floating-point representation.
\\
These preliminary results demonstrate the potential for substantial gains by dynamically adjusting the time step according to a smoothness parameter, such as the first and second derivatives, and by appropriately tuning $\Delta t$ so as to preserve a prescribed accuracy while reducing the computational cost.
}

\subsection{The Chen system}

Chaotic differential equations are unstable in the sense that small perturbations of the initial conditions lead to entirely alternative solutions. Therefore, achieving consistency and stability in numerical computation is challenging. Indeed, chaotic differential equations are quite sensitive to the discretisation errors, which are amplified, leading to wrong approximations.
\subsubsection{General view}
An important dynamical system is the Chen chaotic problem, first introduced by Chen and Ueta~\cite{CU99}. It is considered as a traditional benchmark for assessing a numerical method's ability to reproduce the correct solution up to a given time since small variations lead to important deviations at the final time. We adopt the original formulation, given as
\[
\left\{
\begin{array}{l}
x'(t)= a(y(t)-x(t)),\\
y'(t)=(c-a)x(t)-x(t)z(t)+cy(t),\\
z'(t)=x(t)y(t)-bz(t),
\end{array}
\right.
\]
with $a$, $b$, and $c$ the parameters of the system. In the present simulation, we consider the usual values of $a = 35$, $b = 3$, and $c = 28$ and $T=15$, which provide a chaotic behaviour~\cite{UC00,LZC02}.

\subsubsection{Numerical simulations}

In Figure~\ref{tab:chen3D}, a three-dimensional view of the solution is illustrated using a second-order accurate approximation with $N=8400000$ and a $24^{\text{th}}$-order accurate approximation with $N=5000$. The solution has two attractors, clearly identified by two planes and two centres, which are captured in both approximations.

\begin{figure}
\centering
\begin{tabular}{cc}
\includegraphics[width=0.32\textwidth]{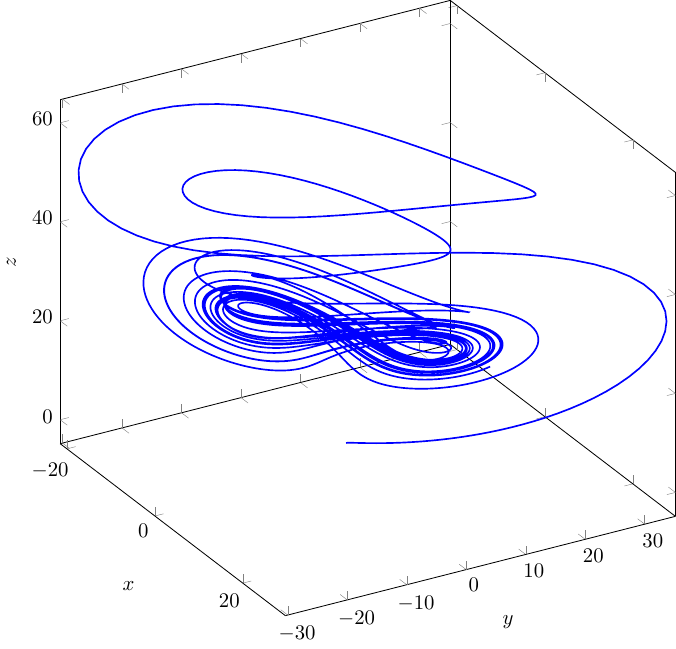}
\hspace{0.25cm} & \hspace{0.25cm}
\includegraphics[width=0.32\textwidth]{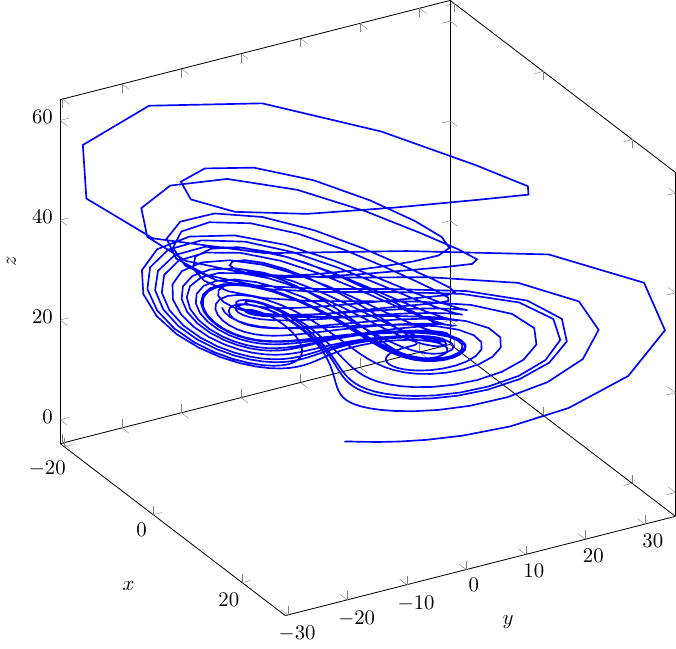}\\
(a) K=1, R=1, and N=8400000. & (b) K=4, R=5, and N=5000.
\end{tabular}
\caption{Approximate solution to the Chen system.}
\label{tab:chen3D}
\end{figure}

The same experience is reproduced with a structural scheme of an intermediate convergence order of $8$ with $K=2$ and $R=3$, and $N=60000$, whose approximate solution for the three coordinates as a function of time is illustrated in Figure~\ref{fig:chenCoordinates}.

\begin{figure}
\centering
\begin{tabular}{ccc}
\includegraphics[width=0.32\textwidth]{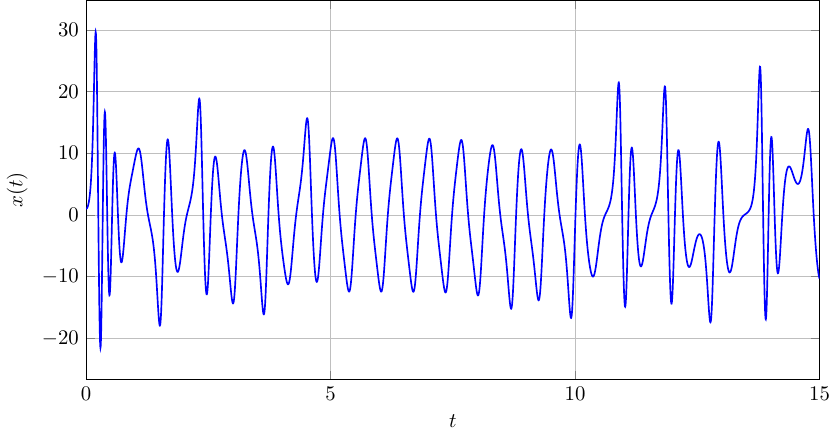}&
\includegraphics[width=0.32\textwidth]{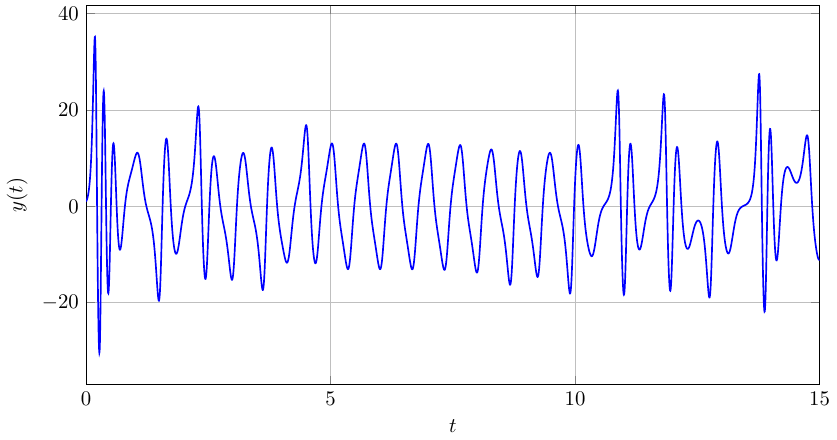}&
\includegraphics[width=0.32\textwidth]{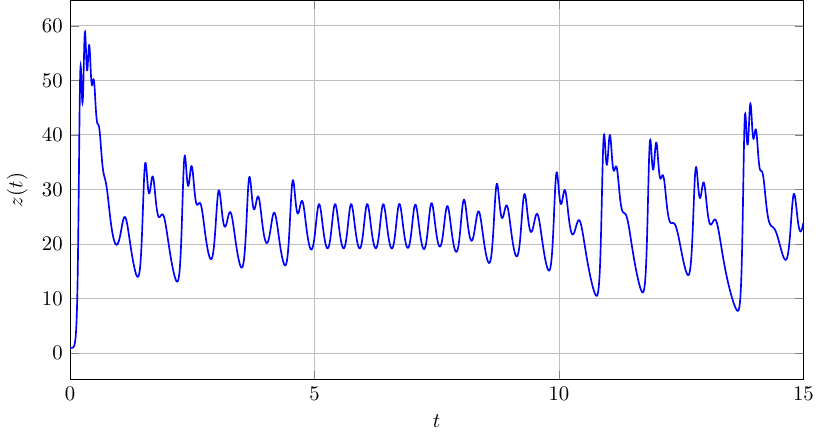}\\
(a) $x$-coordinate. & (c) $y$-coordinate. & (c) $z$-coordinate.
\end{tabular}
\caption{The three coordinate curves as a function of time for the Chen system benchmark with $K=2$, $R=3$, and $N=60000$.}
\label{fig:chenCoordinates}
\end{figure}

\subsubsection{Sensitivity}

The main issue in computing an approximate solution to a chaotic system is the high sensitivity of the structural schemes to the parameters $K$ and $R$, which determine the accuracy of the numerical approximations. Moreover, the iterative fixed-point procedure depends on the tolerance $\varepsilon$, which determines the residual of the final converged numerical solution. We study the impact of these parameters on the chaotic Chen system.

Since there is no exact solution to evaluate the approximate solution accuracy and stability, a reference solution is taken from performing a simulation with the structural scheme $\SK{4}{8}$ with convergence of order of $36$ using an octuple-precision, a very restrict tolerance of $\varepsilon=10^{-60}$ for the fixed-point residual, and a very fine time grid with $N=800000$ points to guarantee an extreme accuracy of the approximate solution. On the other hand, we perform the simulation with three different configurations: the structural scheme $\SK{1}{1}$ with $N=8400000$, $\SK{1}{2}$ with $N=4200000$, and $\SK{4}{5}$ with $N=5000$.

The three coordinates of the reference and approximate solutions are reported in Table~\ref{tab::sensitivity_parameters_KR}. The correct digits corresponding to the reference solution are reported in blue, whereas the digits that do not match the reference solution are reported in black. The structural scheme $\SK{1}{1}$, with second-order convergence, does not provide an acceptable approximate solution to the chaotic system, even with a large number of points, demonstrating the importance of using higher-order accurate schemes. The structural scheme $\SK{1}{2}$, with fourth-order of convergence, provides a rough approximation and is ineffective from a practical viewpoint. On the other side, the structural scheme $\SK{4}{5}$, with $24^{\text{th}}$-order of accuracy, manages to compute a reasonable approximation, even with a rather low number of time grid points. Indeed, very high-order accurate structural schemes deliver excellent approximations with compact stencils.

\begin{table}[H]
\centering
\caption{Reference and approximate solutions for the Chen system benchmark.}
\label{tab::sensitivity_parameters_KR}
\begin{tabular}{@{}lcc@{}}
\toprule
scheme & \phantom{a} & position coordinates\\
\midrule
\multirow{3}{*}{\parbox{3.7cm}{Reference\\$\SK{4}{8}$, N=800000}} &&
x=\textcolor{blue}{-0.1033913519761038851277830669958320E+02}\\
&& y=\textcolor{blue}{-0.1110031188034702371431870391381926E+02}\\
&& z=\textcolor{blue}{+0.2384877914088903906611669394788490E+02}\\
\midrule
\midrule
\multirow{3}{*}{$\SK{1}{1}$, N=8400000} &&
x=+0.1303137147172269558345093394817314E+02\\
&& y=+0.1633314173284638979195080632679370E+02\\
&& z=+0.1990546738498253871112868554565699E+02\\
\midrule
\multirow{3}{*}{$\SK{1}{2}$, N=4200000} &&
x=\textcolor{blue}{-0.1034}310938749620698298379213201446E+02\\
&& y=\textcolor{blue}{-0.1109}718401089064385254542611775984E+02\\
&& z=\textcolor{blue}{+0.238}6294859230536471454733356240557E+02\\
\midrule
\multirow{3}{*}{$\SK{4}{5}$, N=5000} &&
x=\textcolor{blue}{-0.10339135197610}47665823283445056424E+02\\
&& y=\textcolor{blue}{-0.11100311880346}95540840356160225629E+02\\
&& z=\textcolor{blue}{+0.23848779140889}35167622294039641605E+02\\
\bottomrule
\end{tabular}
\end{table}

The impact of the tolerance $\varepsilon$ on the accuracy of the approximate solution of the Chen system is also assessed. The error of the $x$-coordinate, in relation to the reference solution, is reported in Table~\ref{tab_error_x_vs_tolerance} as a function of the tolerance $\varepsilon$ with the structural schemes $\SK{3}{R}$ with $R=1,\ldots,5$. Due to the intricate nature of the chaotic system, $\varepsilon=10^{-8}$ is far from providing an acceptable approximate solution. For $\varepsilon=10^{-15}$, a rough approximation to the reference solution is obtained, but the expected convergence order of the structural scheme is not achieved. On the other side, $\varepsilon=10^{-25}$ enables recovery only of the sixth-order of convergence. In fact, we need to set $\varepsilon=10^{-30}$ to achieve the tenth-order of convergence, and, finally, $\varepsilon=10^{-50}$ to get all the expected convergence orders. In conclusion, for this challenging benchmark, the tolerance value is a critical parameter, together with quad- or octuple-precision, to achieve extreme convergence orders.

\begin{table}[H]
\centering
\caption{$x$-coordinate errors and convergence orders for the Chen system benchmark with $K=3$ and $R=1,\ldots,5$.}
\label{tab_error_x_vs_tolerance}
\resizebox{\textwidth}{!}{
\begin{tabular}{@{}r@{}r@{}r@{}r@{}rrrrrrrrrr@{}}
\toprule
R &\phantom{aaa} & N &\phantom{aaa} &
\multicolumn{2}{c}{$\varepsilon$=1E-8} &
\multicolumn{2}{c}{$\varepsilon$=1E-15} &
\multicolumn{2}{c}{$\varepsilon$=1E-25} &
\multicolumn{2}{c}{$\varepsilon$=1E-30} &
\multicolumn{2}{c}{$\varepsilon$=1E-50}\\
\cmidrule(lr){5-6}\cmidrule(lr){7-8}\cmidrule(lr){9-10}\cmidrule(lr){11-12}\cmidrule(lr){13-14}
&&&& err & ord & err & ord & err & ord & err & ord & err & ord\\
\midrule
\multirow{3}{*}{1} && 60,000 && 1.33E+01 & \emdash & 2.12E-02 & \emdash & 1.33E-02 & \emdash & 1.33E-02 & \emdash & 1.33E-02 & \emdash\\
&& 120,000 && 1.27E+01 & 0.1 & 7.46E-03 & 1.5 & 2.11E-04 & 6.0 & 2.11E-04 & 6.0 & 2.11E-04 & 6.0\\
&& 180,000 && 1.41E+01 & \emdash & 6.97E-03 & 0.2 & 1.85E-05 & 6.0 & 1.85E-05 & 6.0 & 1.85E-05 & 6.0\\
\midrule
\multirow{3}{*}{2} && 60,000 && 3.23E-01 & \emdash & 4.02E-03 & \emdash & 1.16E-11 & \emdash & 1.18E-11 & \emdash & 1.18E-11 & \emdash\\
&& 120,000 && 2.09E+01 & \emdash & 4.25E-03 & \emdash & 1.40E-13 & 6.4 & 1.15E-14 & 10.0 & 1.15E-14 & 10.0\\
&& 180,000 && 5.63E-02 & 14.6 & 1.91E-03 & 2.0 & 2.49E-13 & \emdash & 1.99E-16 & 10.0 & 2.00E-16 & 10.0\\
\midrule
\multirow{3}{*}{3} && 60,000 && 1.78E+01 & \emdash & 3.30E-03 & \emdash & 1.28E-13 & \emdash & 1.16E-16 & \emdash & 1.15E-16 & \emdash\\
&& 120,000 && 2.66E+01 & \emdash & 3.45E-03 & \emdash & 2.88E-13 & \emdash & 5.95E-19 & 7.6 & 2.81E-20 & 12.0\\
&& 180,000 && 2.62E+01 & 0.0 & 2.33E-03 & 1.0 & 2.75E-14 & 5.8 & 2.44E-19 & 2.2 & 2.16E-22 & 12.0\\
\midrule
\multirow{3}{*}{4} && 60,000 && 1.40E+01 & \emdash & 1.23E-04 & \emdash & 2.10E-13 & \emdash & 1.92E-18 & \emdash & 2.50E-24 & \emdash\\
&& 120,000 && 1.57E+00 & 3.2 & 2.88E-03 & \emdash & 2.15E-13 & \emdash & 1.12E-18 & 0.8 & 3.82E-29 & 16.0\\
&& 180,000 && 3.09E+00 & \emdash & 2.48E-03 & 0.4 & 1.03E-13 & 1.8 & 4.26E-19 & 2.4 & 5.82E-32 & 16.0\\
\midrule
\multirow{3}{*}{5} && 60,000 && 3.76E+00 & \emdash & 3.48E-03 & \emdash & 1.39E-13 & \emdash & 1.45E-18 & \emdash & 3.43E-29 & \emdash\\
&& 120,000 && 1.63E+01 & \emdash & 2.53E-03 & 0.5 & 8.36E-15 & 4.1 & 5.57E-19 & 1.4 & 1.31E-34 & 18.0\\
&& 180,000 && 4.32E+00 & 3.3 & 4.03E-03 & \emdash & 1.80E-13 & \emdash & 7.46E-19 & \emdash & 8.34E-38 & 18.2\\
\bottomrule
\end{tabular}
}
\end{table}

\subsection{Numerical efficiency}

Running time \textit{versus} accuracy characterises the computational effort to achieve a given solution quality. In that regard, higher-order accurate methods are highly effective at reducing the running time for a given solution quality, since higher-order convergence allows the use of coarser grids. Indeed, for smooth solutions with low variation at the final time, the computational savings are significant. On the other hand, when dealing with solutions involving abrupt time changes, a physical constraint on the time step may arise. For instance, if the solution exhibits a high-frequency phenomenon, a satisfactory spectral resolution can only be achieved with a sufficiently small time step. To assess the numerical efficiency of the structural schemes, the Chen system is analysed at the final time $T=15$ as a representative case, where we found that $N=5000$ is the minimum number of time grid points to guarantee an acceptable spectral resolution.

\subsubsection{Running time versus accuracy for the $\SK{K}{R}$ schemes}
First-derivative structural schemes $\SK{1}{R}$ with convergence order of $R+1$ involve only the function and the first derivative, and no additional physical equation is required, which is convenient when dealing with problems that do not support additional derivations. Note that one iteration of the fixed-point procedure evaluates $R$ times the physical equation and involves $R$ unknowns. On the other hand, high-order accurate structural schemes, such as $\SK{2}{R}$ and $\SK{3}{R}$, require additional physical equations and unknowns. Moreover, in these cases, one iteration of the fixed-point procedure evaluates $KR$ times the physical equations and involves $RK$ unknowns, hence, a larger system needs to be solved.

\begin{figure}[H]
\centering
\begin{tikzpicture}
\begin{loglogaxis}[
width=13cm, 
height=4.5cm, 
scale only axis, 
xlabel={\tablepolice RT (s)},
ylabel={\tablepolice E},
grid=both,
legend style={
at={(0.5,-0.30)}, anchor=north, draw=none,
font=\scriptsize,
/tikz/column sep=4pt,
/tikz/row sep=1pt,
nodes={inner xsep=2pt, inner ysep=1pt},
},
legend image post style={xscale=0.75, yscale=0.75, mark options={scale=0.75}},
legend columns=5, 
legend cell align=left,
tick align=outside,
minor grid style={dashed,opacity=0.35},
major grid style={solid,opacity=0.5},
every axis plot/.append style={line width=1.3pt, mark options={solid, scale=1.0}},
xtick={1e1,1e2},
xticklabels={\tablepolice 1.0E+1,\tablepolice 1.0E+2},
ytick={1e-2,1e-3,1e-4},
yticklabels={\tablepolice 1.0E-2,\tablepolice 1.0E-3,\tablepolice 1.0E-4},
]

\addlegendimage{color=oiBlue, line width=1.6pt, mark=none}
\addlegendentry{ord 4:}

\addplot+[color=oiBlue, mark=o]
coordinates {
(180.64, 0.00819)
(190.36, 0.00733)
(194.12, 0.00658)
(200.79, 0.00592)
};
\addlegendentry{$\SKfoot{1}{2}$}

\addplot+[color=oiBlue, mark=square*]
coordinates {
(187.74, 0.00874)
(193.97, 0.00718)
(201.82, 0.00595)
(209.76, 0.00497)
};
\addlegendentry{$\SKfoot{1}{3}$}

\addplot+[color=oiBlue, mark=triangle*]
coordinates {
(134.27, 0.0092)
(139.21, 0.00734)
(145.49, 0.00592)
(149.77, 0.00483)
};
\addlegendentry{$\SKfoot{2}{1}$}

\addlegendimage{empty legend}\addlegendentry{}

\addlegendimage{color=oiOrange, line width=1.6pt, mark=none}
\addlegendentry{ord 6:}

\addplot+[color=oiOrange, mark=diamond*]
coordinates {
(21.78, 0.00425)
(24.71, 0.00168)
(27.39, 0.000753)
(30.24, 0.000371)
};
\addlegendentry{$\SKfoot{1}{4}$}

\addplot+[color=oiOrange, mark=star]
coordinates {
(20.94, 0.00449)
(24.57, 0.00151)
(28.27, 0.000599)
(31.89, 0.000269)
};
\addlegendentry{$\SKfoot{1}{5}$}

\addplot+[color=oiOrange, mark=x]
coordinates {
(12.32, 0.00679)
(14.44, 0.00226)
(16.36, 0.000894)
(18.27, 0.000401)
};
\addlegendentry{$\SKfoot{2}{2}$}

\addplot+[color=oiOrange, mark=+]
coordinates {
(12.49, 0.0024)
(14.69, 0.00063)
(16.96, 0.000211)
(18.96, 8.37E-05)
};
\addlegendentry{$\SKfoot{3}{1}$}

\addlegendimage{color=oiGreen, line width=1.6pt, mark=none}
\addlegendentry{ord 8:}

\addplot+[color=oiGreen, mark=oplus*]
coordinates {
(4.66, 0.00322)
(5.16, 0.00139)
(5.56, 0.00065)
(6.24, 0.000324)
};
\addlegendentry{$\SKfoot{2}{3}$}

\addplot+[color=oiGreen, mark=otimes*]
coordinates {
(3.80, 0.0034)
(4.18, 0.00147)
(4.54, 0.000686)
(4.92, 0.000342)
};
\addlegendentry{$\SKfoot{4}{1}$}

\addlegendimage{empty legend}\addlegendentry{}
\addlegendimage{empty legend}\addlegendentry{}

\addlegendimage{color=oiRed, line width=1.6pt, mark=none}
\addlegendentry{ord 10:}

\addplot+[color=oiRed, mark=square*]
coordinates {
(3.25, 0.00303)
(3.64, 0.000649)
(3.96, 0.000171)
(4.24, 5.28E-05)
};
\addlegendentry{$\SKfoot{2}{4}$}

\addplot+[color=oiRed, mark=triangle*]
coordinates {
(2.07, 0.00669)
(2.49, 0.000714)
(2.87, 0.000115)
(3.19, 2.47E-05)
};
\addlegendentry{$\SKfoot{3}{2}$}

\addlegendimage{empty legend}\addlegendentry{}
\addlegendimage{empty legend}\addlegendentry{}
\end{loglogaxis}
\end{tikzpicture}\\
(a) Structural schemes with convergence orders of 4, 6, 8, and 10.
\vskip 2em
\begin{tikzpicture}
\begin{loglogaxis}[
xlabel={\tablepolice RT (s)},
ylabel={\tablepolice E},
width=13cm, 
height=4.5cm, 
scale only axis, 
grid=both,
legend style={
at={(0.5,-0.30)}, anchor=north, draw=none,
font=\scriptsize,
/tikz/column sep=4pt,
/tikz/row sep=1pt,
nodes={inner xsep=2pt, inner ysep=1pt},
},
legend image post style={xscale=0.75, yscale=0.75, mark options={scale=0.75}},
legend columns=5, 
legend cell align=left,
tick align=outside,
minor grid style={dashed,opacity=0.35},
major grid style={solid,opacity=0.5},
cycle list name=okabe-ito,
every axis plot/.append style={line width=1.2pt, mark options={solid, scale=1.0}},
xtick={1e1,1e2}, 
xticklabels={\tablepolice 1.0E+1,\tablepolice 1.0E+2}, 
ytick={1e-8,1e-12,1e-16,1e-20},
yticklabels={\tablepolice 1.0E-8,\tablepolice 1.0E-12,\tablepolice 1.0E-16,\tablepolice 1.0E-20},
minor x tick num=9, 
xmin=8, xmax=120, 
]

\addplot+ coordinates {(18.28, 4.39E-08) (21.72, 3.02E-09) (25.10, 3.39E-10) (28.76, 5.34E-11)};
\addlegendentry{$\SKfoot{1}{10}$}

\addplot+ coordinates {(20.54, 7.42E-09) (24.78, 5.14E-10) (29.17, 5.8E-11) (32.76, 9.14E-12)};
\addlegendentry{$\SKfoot{1}{11}$}

\addplot+ coordinates {(16.06, 3.52E-11) (21.73, 2.72E-13) (27.06, 8.63E-15) (32.37, 5.93E-16)};
\addlegendentry{$\SKfoot{2}{5}$}

\addplot+ coordinates {(15.23, 1.5E-13) (21.00, 1.16E-15) (26.30, 3.66E-17) (31.48, 2.52E-18)};
\addlegendentry{$\SKfoot{3}{3}$}

\addplot+ coordinates {(18.19, 5.02E-15) (25.06, 3.87E-17) (31.36, 1.23E-18) (37.07, 8.42E-20)};
\addlegendentry{$\SKfoot{4}{2}$}

\end{loglogaxis}
\end{tikzpicture}\\
(b) Structural schemes with convergence order of 12.
\caption{Running time ($RT$) \textit{versus} error ($E$) for various structural schemes.}
\label{fig::accuracy_vs_running-time}
\end{figure}

Figure~\ref{fig::accuracy_vs_running-time}(a) illustrates the accuracy ($E$, on the vertical axis) as a function of the running time ($RT$, on the horizontal axis) for various structural schemes with convergence orders of 4, 6, 8, and 10. We clearly distinguish three clusters for the fourth-order accurate schemes (right), the sixth-order accurate schemes (middle), and the eighth-/tenth-order accurate schemes (left). For a target accuracy of $E=10^{-2}$, it takes around $200$s for the fourth-order accurate structural schemes, $20$s for the sixth-order accurate structural schemes, and only $2$s for the tenth-order accurate structural schemes.

Figure~\ref{fig::accuracy_vs_running-time}(b) illustrates the differences between the possible combinations of parameters $K$ and $R$ that achieve a $12^{\text{th}}$-order of convergence. The lines are parallel, confirming that all the structural schemes have $12^{\text{th}}$-order convergence. On the other hand, increasing $K$ rather than $R$ to achieve the same convergence order is more computationally efficient, since it results in a drastic reduction in the multiplicative constant associated with the asymptotic regime. We stress that it would be difficult to take advantage of such a property if the analytical high-order derivatives of the physical equations turn out to be difficult or even impossible.

Assuming that the approximate solution has reached the asymptotic regime of the relation running time \textit{versus} accuracy, we provide some extrapolations of the running time in Table~\ref{tab:tempos-estimados} to reach a given accuracy for several structural schemes. For example, to achieve an error above $E=10^{-8}$, the structural scheme $\SK{1}{2}$ with fourth-order of convergence requires about $3000$ more time than the structural scheme $\SK{3}{2}$. Such a ratio corresponds to a substantial difference between, for instance, $250$s and one week of computation.

\begin{table}[H]
\centering
\caption{Extrapolation time for various structural schemes.}
\label{tab:tempos-estimados}
\begin{tabular}{@{}llrrr@{}}
\toprule
ord & scheme & E=10-3 & E=10-4 & E=10-8\\
\midrule
\multirow{3}{*}{4} &
$\SK{1}{2}$ & 357.11 & 749.16 & 1,4509.68\\
& $\SK{1}{3}$ & 287.83 & 454.52 & 2826.47\\
& $\SK{2}{1}$ & 197.40 & 294.63 & 1462.14\\
\midrule
\multirow{4}{*}{6} &
$\SK{1}{4}$ & 26.44 & 35.99 & 123.63\\
& $\SK{1}{5}$ & 26.18 & 36.95 & 146.56\\
& $\SK{2}{2}$ & 16.11 & 22.19 & 79.80\\
& $\SK{3}{1}$ & 13.92 & 18.56 & 58.73\\
\midrule
\multirow{2}{*}{8} &
$\SK{2}{3}$ & 5.36 & 7.16 & 22.67\\
& $\SK{4}{1}$ & 4.36 & 5.64 & 15.84\\
\midrule
\multirow{2}{*}{10} &
$\SK{2}{4}$ & 3.51 & 4.09 & 7.50\\
& $\SK{3}{2}$ & 2.41 & 2.88 & 5.89\\
\bottomrule
\end{tabular}
\end{table}

\rb{
\subsection{Comparison with the Collocation method}
Among the many existing numerical techniques, we restrict our comparison to the collocation method for several reasons. First, both the structural method and the collocation method can be interpreted as particular cases of RK methods and can be constructed to achieve arbitrary order of accuracy. Second, we do not consider multistep methods, which we regard as too different from the structural method, since our approach does not rely on backlog data. Finally, techniques such as deferred correction methods are based on a very different paradigm, increasing the order through successive improvements.
\subsubsection{Algorithm}
To provide a fair comparison between the structural and collocation methods, we use the same fixed-point technique as that proposed in Section~\ref{sec::numerical_solver}. We denote by $A=(a_{ij}){i,j=1}^s$, $b=(b_i){i=1}^s$, and $c=(c_i)_{i=1}^s$ the elements of the Butcher tableau, and by $Z_n$ an approximation of the solution at time $t_n$.
\\
We construct a sequence $(Z^{[k]}_{n+1})$, $k=0,1,\ldots$, according to the following three stages:
\begin{enumerate}
\item Initialization $Z^{[0]}_{n,i}=Z_n+c_i\Delta t D_{n}$, $D^{[0]}_{n,i}=f(t_n+c_i\Delta t,Z^{[0]}_{n,i})$.
\item Loop $Z^{[k+1]}_{n,i}=Z_n+\Delta t \sum_{j=1}^s a_{i,j} D^{[k]}_{n,j}$, $D^{[k+1]}_{n,i}=f(t_n+c_i\Delta t,Z^{[k+1]}_{n,i})$.
\item Test $Z^{[k+1]}_{n}=Z_n+\Delta t\sum_{i=1}^s b_i D^{[k+1]}_{n,i}$, stop if $|Z^{[k+1]}_{n}-Z^{[k]}_{n}|\leq \varepsilon$ else goto 2.
\item Termination $Z_{n+1}=Z^{[k+1]}_{n}$.
\end{enumerate}
where $Z^{[k]}_{n,i}$, $D^{[k]}_{n,j}$ are, respectively, the values and derivatives at the intermediate stage $t_{n,i}=t_n+c_i\Delta t$. We have implemented the fourth-order {\tt CL4} up to the 12th-order method {\tt CL12} and report the Butcher tableau in \ref{app::colocation}.
\subsubsection{Running-time versus accuracy}
We use the Chen dynamical system as a benchmark since small errors will quickly results into large deviations at the final time.
All the code were implemented in {\tt C++} using the {\tt qd} library to enable the quad- and octa-precision. Indeed, the approximations of a chaotic system require very high-precision. The {\tt qd } library is quite efficient since a quad-precision run lasts 3 times the double precision run and the octa-precision run lasts 14 times the quad-precision run.  For orders $4$, $6$ and $8$, quad-precision is adequate but higher orders $10$, $12$ or more require the full octa-precision.
\begin{figure}[ht]
    \centering
    \includegraphics[width=0.48\linewidth, clip, trim=2cm 1cm 3cm 2cm]{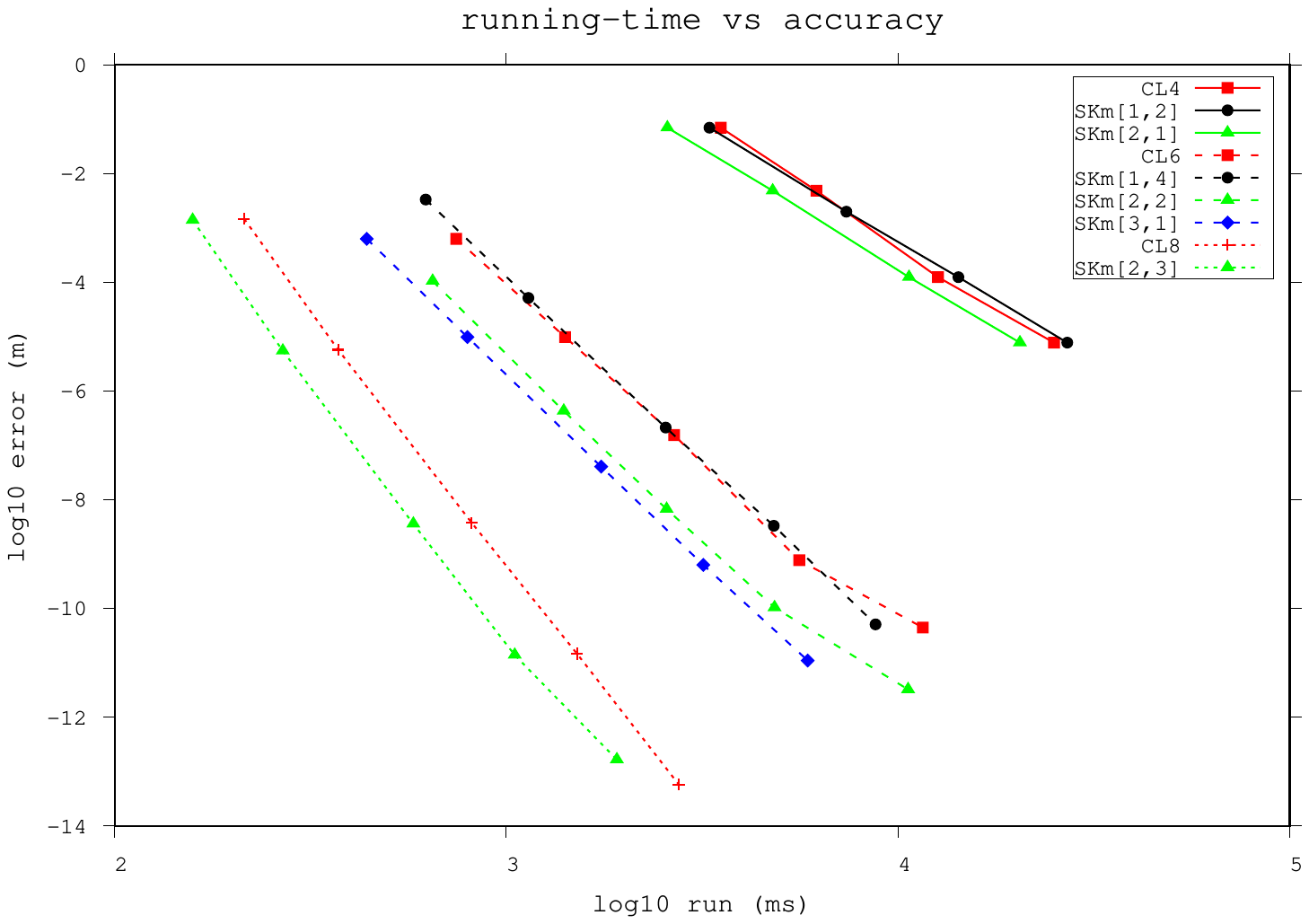}\hskip 1em
    \includegraphics[width=0.48\linewidth, clip, trim=2cm 1cm 3cm 2cm]{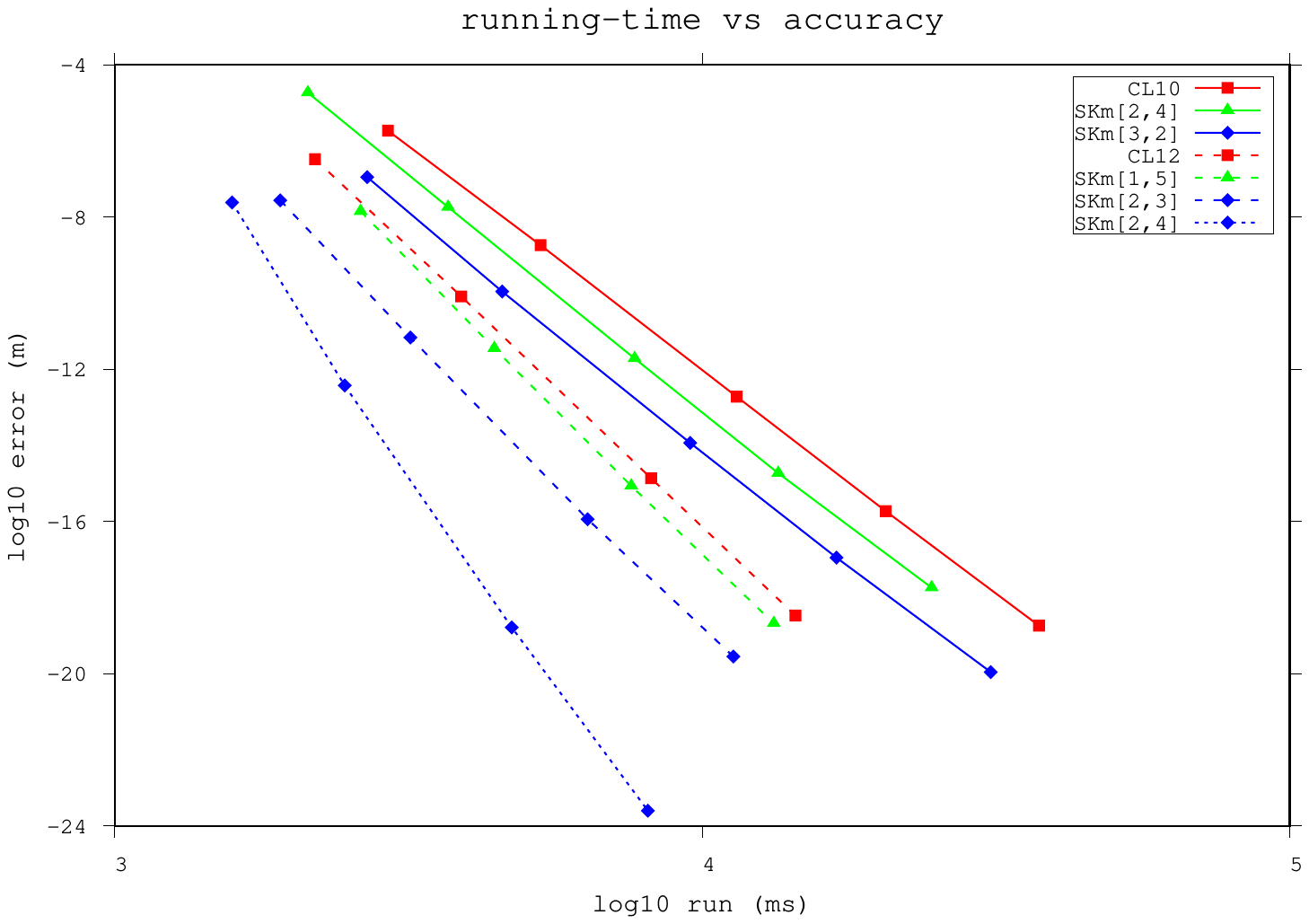}
    \caption{running time versus accuracy. Left panel: Order 4,6,8 using the quad-precision. Right panel: Order 10,12,14 using the octa-precision.}
    \label{fig:running-time_vs_accuracy}
\end{figure}
\\
We plot in Figure \ref{fig:running-time_vs_accuracy} left panel the accuracy in function of the running time for the $4$, $6$ and $8$ order using the quad-precision, while we report the higher orders $10$, $12$ and $16$ using the full octa-precision. Structural method of type $K=1$ (only one derivative) provides very similar results than the collocation method. By introducing the compact version using the second or the thid derivative as additional unknowns, we strongly improve the running-time for the same accuracy regarding to the equivalent collocation methods.
}

\begin{rem}
\rb{
Compact schemes usually provide better spectral resolution, {\it i.e.} the ability to preserve high frequency signal than the non-compact method. We expect that the $\SK{2}{R}$ and $\SK{3}{R}$ enjoy such a property in comparison to the collocation schemes. $\blacksquare$}
\end{rem}

\vskip 2em
\centerline{\includegraphics[width=0.30\textwidth]{vintage_text_separator.png}}
\centerline{\Large \sl Part II: Numerical analysis, accuracy, stability, spectral resolution}
\vskip 0.5em

{\sl \small The second part is dedicated to the numerical analysis of the structural method, namely the convergence, the linear stability. We also introduce a concept of spectral resolution adapted to the new method to quantify the ability to obtain sub-grid scale information. We first prove the convergence of the iterative method for the non-linear problem in section 7 while the consistency errors and convergence orders are established in Section~8, where we determine the convergence order $p$ as a function of $K$ and $R$. Section~9 is dedicated to stability, where our method is shown to be unconditionally stable across a wide range of configurations. Appendix C gives the technical detail about the stability of the schemes we propose in the section. We address the spectral resolution issue in Section~10 and focus on the ability of our schemes to provide very high resolution, enabling us to recover sub-scale information, that is, variations between successive time steps.}

\section{Convergence of the numerical solver}

The main result of this section establishes that, under a CFL-like condition linked to the Lipschitz constants of the physical equations and the coefficients of the structural equations, the proposed iterative procedure for numerical approximation converges.

\begin{lemma}
\label{lemma_bounded_Lipschitz}
Let $B(M_0)=[-M_0,M_0]$ be the closed ball of radius $M_0>0$ centred at the origin and assume that the function $(z,t)\to f(z,t)$ is a $C^{K+1}(\mathbb{R}\times[0,T])$ function. Then, for any $k\bint{1,K}$, the functions $(z_0,\ldots,z_{k-1},t)\to z_k=\ddf{k-1}{z_0,\ldots,z_{k-1},t}$ are bounded by a constant $M_k$.
Moreover, we have a $L_k$-Lipschitz property uniform in time $t\in[0,T]$ on the compact set $B(M_0)\times\cdots\times B(M_{k-1})$ given by
\begin{equation}\label{global_lipschitz_inequality}
|z_k-z_k'|\leqslant L_k\left(\sum_{\kappa=0}^{k-1}|z_\kappa-z_\kappa'|\right).
\end{equation}
\end{lemma}
\begin{proof}
We prove the property by induction. Since $(z_0,t)$ belongs to the compact set $B(M_0)\times[0,T]$, the continuity of $f$ means that $z_1=\ddf{0}{t,z_0}$ is bounded and belongs to a compact set, which we include inside a closed ball $B(M_1)$ centered at the origin for $M_1$ large enough. On the other hand, for two values $z_0,z_0'\in B(M_0)$, one has
\[
\ddf{0}{z_0,t}-\ddf{0}{z_0',t}=f(z_0,t)-f(z_0',t)=\partial_zf\big(\theta z_0+(1-\theta)z_0',t\big)(z_0-z_0')
\]
with $\theta\in[0,1]$. Note that the intermediate point belongs to the ball $B(M_0)$. Denoting
\[
L_1=\max^{\zeta\in B(M_0)}_{t\in[0,T]}\partial_z f(\zeta,t),
\]
we obtain the Lipschitz inequality, uniform in time
\[
|z_1-z_1'|=\big|\ddf{0}{z_0,t}-\ddf{0}{z_0',t}\big|\leqslant L_1|z_0-z_0'|.
\]

Assuming that the properties hold for $k<K$, we now prove them for $k+1$. First, we recall that $z_{k+1}=\ddf{k}{z_0,z_1,\ldots,z_k,t}$ is made up of the sum of polynomial functions with respect to variables $z_1,\ldots,z_k$ multiplied with partial derivatives of $f(z,t)$ in $z_0$ and $t$ of order lower than $k$. Since the partial derivatives depend only on $z_0$ and $t$, they are uniformly bounded in $ B(M_0)\times[0,T]$, while the polynomials in $z_1,\ldots,z_{k}$ are also bounded from the induction assumption. We then see that $z_{k+1}$ is bounded and belongs to a closed ball $B(M_{k+1})$ with $M_{k+1}$ large enough.
On the other hand, the Lipschitz property derives from the fact that $\ddf{k}{z_0,z_1,\ldots,z_k,t}$ is polynomial in $z_1,\ldots,z_k$ with bounded values, while the $z$-derivative of the partial derivatives of order $k$ remains bounded in $B(M_0)\times[0,T]$. Therefore, we have a Lipschitz operator uniform in time, and the property holds with a Lipschitz constant $L_{k+1}$.
\end{proof}

From Lemma~\ref{lemma_bounded_Lipschitz}, we deduce a Lipschitz property regarding the sole variable $z_0$.
\begin{lemma}\label{lemma_bounded_Lipschitz_z_0}
Under the assumptions of Lemma~\ref{lemma_bounded_Lipschitz}, let $z_0\in B(M_0)$ and $t\in[0,T]$, and define by induction
\[
z_k=\ddf{k-1}{z_0,\ldots,z_{k-1},t},\quad k\bint{1,K}.
\]
Then, for any $k\bint{1,K}$ and $z_0,z_0'\in B(M_0)$, we have the inequality
\begin{equation}\label{Lipschitz_z_0}
|z_k-z_k'|\leqslant\widehat{L}_k|z_0-z_0'|,
\end{equation}
with $\widehat{L}_k$ independent of $z_0$, $z_0'$, and time $t\in[0,T]$.
\end{lemma}
\begin{proof}
For $k=1$ we take $\widehat{L}_1=L_1$. For $k=2$, we note
\[
|z_2-z_2'|\leqslant L_2(|z_0-z_0'|+|z_1-z_1'|)\leqslant L_2(1+\widehat{L}_1)|z_0-z_0'|,
\]
and the inequality holds with $\widehat{L}_2=L_2(1+\widehat{L}_1)$.

The general case is proven by induction. From inequality~\eqref{global_lipschitz_inequality} we have
\[
|z_k-z_k'|\leqslant L_k\left(|z_0-z_0'|+\sum_{\kappa=1}^{k-1}|z_\kappa-z_\kappa'|\right)\leqslant L_k\left(|z_0-z_0'|+\sum_{\kappa=1}^{k-1}\widehat{L}_\kappa|z_0-z_0'|\right)\leqslant\widehat{L}_k|z_0-z_0'|
\]
with
\[
\widehat{L}_k=L_k\left(1+\sum_{\kappa=1}^{k-1}\widehat{L}_\kappa\right),
\]
and inequality~\eqref{Lipschitz_z_0} holds.
\end{proof}
\begin{cor}
Let $\Vert\cdot\Vert_{\infty}$ be the maximum norm and
\[
\Phi_n^{(k)}=\ddf{k-1}{\Phi_n^{(0)},\ldots,\Phi_n^{(k-1)},T_n},
\quad
{\Phi'}_n^{(k)}=\ddf{k-1}{{\Phi'}_n^{(0)},\ldots,{\Phi'}_n^{(k-1)},T_n},
\quad k\bint{1,K}.
\]
Then, under the assumptions of Lemma~\ref{lemma_bounded_Lipschitz}, we have for any $k\bint{1,K}$ the estimate
\begin{equation}\label{Vectorial_Lipschitz_0}
\big\Vert\Phi_n^{(k)}-{\Phi'}_n^{(k)}\big\Vert_{\infty}\leqslant\widehat{L}_k\big\Vert\Phi_n^{(0)}-{\Phi'}_n^{(0)}\big\Vert_{\infty}.
\end{equation}
\end{cor}
\begin{proof}
There exists $r_0$ such that $\big\Vert\Phi_n^{(k)}-{\Phi'}_n^{(k)}\big\Vert_{\infty}=\big\vert\phi_{n+r_0}^{(k)}-{\phi'}_{n+r_0}^{(k)}\big\vert$. Then, applying inequality~\eqref{Lipschitz_z_0}, we deduce
\[
\big\Vert\Phi_n^{(k)}-{\Phi'}_n^{(k)}\big\Vert_{\infty}=\big\vert\phi_{n+r_0}^{(k)}-{\phi'}_{n+r_0}^{(k)}\big\vert
\leqslant\widehat{L}_k\big\vert\phi_{n+r_0}^{(0)}-{\phi'}_{n+r_0}^{(0)}\big\vert
\leqslant\widehat{L}_k\big\Vert\Phi_{n}^{(0)}-{\Phi'}_{n}^{(0)}\big\Vert_{\infty}.
\]
\end{proof}
\begin{rem}
Since we are dealing with finite-dimensional spaces, all norms are equivalent, and the inequality~\eqref{Vectorial_Lipschitz_0} holds for any norm with an adapted coefficient $\widehat{L}_k$ depending on $R$.
\hfill\smallBB
\end{rem}

\subsection{Convergence of the iterative procedure for structural scheme $\SK{1}{R}$}

Following Lemma~\ref{lemma_bounded_Lipschitz}, we assume that the approximations in vectors $\Phi^{(0)}_{n}[\ell]$, $\ell\geqslant 0$ remain bounded in the ball $B(M_0)$. Hence, the vectors $\Phi^{(1)}_{n}[\ell]$ are bounded and the Lipschitz inequality~\eqref{Vectorial_Lipschitz_0} holds between two successive approximations, that is,
\[
\big\Vert\Phi^{(1)}_{n}[\ell+1]-\Phi^{(1)}_{n}[\ell]\big\Vert\leqslant\widehat{L}_1
\big\Vert\Phi^{(0)}_{n}[\ell+1]-\Phi^{(0)}_{n}[\ell]\big\Vert.
\]

\begin{prop}
Let $\Delta t$ be small enough such that
\begin{equation}
\label{CFL_K=1_fix-point}
\Delta t\widehat{L}_1\left\Vert\boldsymbol{B}^{(1)}\right\Vert<1,\,\boldsymbol{B}^{(1)}=\Big(\boldsymbol{A}^{(0)}\Big)^{-1}\boldsymbol{A}^{(1)}.
\end{equation}
Then, the fixed-point is a contraction, and the successive numerical approximations converge to the solution of the non-linear system derived from the structural scheme $\SK{1}{R}$.
\end{prop}

\begin{proof}
We note that
\[
\boldsymbol{A}^{(0)}\big(\Phi^{(0)}_{n}[\ell+1]-\Phi^{(0)}_{n}[\ell]\big)=
-\Delta t\boldsymbol{A}^{(1)}\big(\Phi^{(1)}_{n}[\ell+1]-\Phi^{(1)}_{n}[\ell]\big),
\]
and thus, applying the Lipschitz condition~\eqref{Vectorial_Lipschitz_0}, yields
\begin{align*}
\big\Vert\Phi^{(0)}_{n}[\ell+1]-\Phi^{(0)}_{n}[\ell]\big\Vert
&\leqslant\Delta t\left\Vert\boldsymbol{B}^{(1)}\right\Vert\,
\big\Vert\Phi^{(1)}_{n}[\ell+1]-\Phi^{(1)}_{n}[\ell]\big\Vert\\
&\leqslant
\Delta t\left\Vert\boldsymbol{B}^{(1)}\right\Vert
\big\Vert\ddf{0}{\Phi^{(0)}_{n}[\ell],T_n}-\ddf{0}{\Phi^{(0)}_{n}[\ell-1],T_n}\big\Vert\\
&\leqslant
\Delta t\left\Vert\boldsymbol{B}^{(1)}\right\Vert\widehat{L}_1
\big\Vert\Phi^{(0)}_{n}[\ell]-\Phi^{(0)}_{n}[\ell-1]\big\Vert,
\end{align*}
which provides the amplification coefficient $\displaystyle\chi_1=\Delta t\widehat{L}_1\left\Vert\boldsymbol{B}^{(1)}\right\Vert$.
Under the CFL-like condition~\eqref{CFL_K=1_fix-point}, we get a geometric series with coefficient $\chi_1<1$, hence the convergence of the fixed-point is guaranteed.
\end{proof}

\subsection{Convergence of the iterative procedure for structural scheme $\SK{2}{R}$}

Once again, we use Lemma~\ref{lemma_bounded_Lipschitz}. We assume that the approximations in vectors $\Phi^{(0)}_{n}[\ell]$, $\ell\geqslant 0$ remain bounded in the ball $B(M_0)$. Hence, vectors $\Phi^{(1)}_{n}[\ell]$ and $\Phi^{(2)}_{n}[\ell]$ are bounded and the Lipschitz inequalities~\eqref{Vectorial_Lipschitz_0} hold between two successive approximations, that is
\begin{gather*}
\Vert\Phi^{(1)}_{n}[\ell+1]-\Phi^{(1)}_{n}[\ell]\Vert\leqslant\widehat{L}_1\Vert\Phi^{(0)}_{n}[\ell+1]-\Phi^{(0)}_{n}[\ell]\Vert,\\
\Vert\Phi^{(2)}_{n}[\ell+1]-\Phi^{(2)}_{n}[\ell]\Vert\leqslant\widehat{L}_2\Vert\Phi^{(0)}_{n}[\ell+1]-\Phi^{(0)}_{n}[\ell]\Vert.
\end{gather*}

\begin{prop}
Let $\Delta t$ be small enough such that
\begin{equation}\label{CFL_K=2_fix-point}
\Delta t\widehat{L}_1\left\Vert\boldsymbol{B}^{(1)}\right\Vert+\Delta t^2\widehat{L}_2\left\Vert\boldsymbol{B}^{(2)}\right\Vert<1,\,
B^{(1)}=\Big(\boldsymbol{A}^{(0)}\Big)^{-1}\boldsymbol{A}^{(1)},\quad\boldsymbol{B}^{(2)}=\Big(\boldsymbol{A}^{(0)}\Big)^{-1}\boldsymbol{A}^{(2)}.
\end{equation}
Then, the fixed-point is a contraction, and the successive numerical approximations converge to the solution of the non-linear system derived from the structural scheme $\SK{2}{R}$.
\end{prop}

\begin{proof}
The proof is very similar to the case of the structural scheme $\SK{1}{R}$. We note that
\begin{align*}
\boldsymbol{A}^{(0)}\big(\Phi^{(0)}_{n}[\ell+1]-\Phi^{(0)}_{n}[\ell]\big)
=&-\Delta t \boldsymbol{A}^{(1)}\big(\Phi^{(1)}_{n}[\ell+1]-\Phi^{(1)}_{n}[\ell]\big)\\
&-\Delta t \boldsymbol{A}^{(2)}(\Phi^{(2)}_{n}[\ell+1]-\Phi^{(2)}_{n}[\ell]),
\end{align*}
and thus, applying the Lipschitz condition~\eqref{Vectorial_Lipschitz_0}, yields
\begin{eqnarray*}
\big\Vert\Phi^{(0)}_{n}[\ell+1]-\Phi^{(0)}_{n}[\ell]\big\Vert
&\leqslant&\Delta t\left\Vert\boldsymbol{B}^{(1)}\right\Vert\,\big\Vert\Phi^{(1)}_{n}[\ell+1]-\Phi^{(1)}_{n}[\ell]\big\Vert+\\
& &\Delta t^2\left\Vert\boldsymbol{B}^{(2)}\right\Vert\,\big\Vert\Phi^{(2)}_{n}[\ell+1]-\Phi^{(2)}_{n}[\ell]\big\Vert\\
&\leqslant&\Delta t\left\Vert\boldsymbol{B}^{(1)}\right\Vert\,\widehat{L}_1\big\Vert\Phi^{(0)}_{n}[\ell]-\Phi^{(0)}_{n}[\ell-1]\big\Vert+\\
& &\Delta t^2\left\Vert\boldsymbol{B}^{(2)}\right\Vert\widehat{L}_2\,\big\Vert\Phi^{(0)}_{n}[\ell]-\Phi^{(0)}_{n}[\ell-1]\big\Vert,
\end{eqnarray*}
which provides the amplification coefficient
\[
\chi_2=
\Delta t\widehat{L}_1\left\Vert\boldsymbol{B}^{(1)}\right\Vert+
\Delta t^2\widehat{L}_2\left\Vert\boldsymbol{B}^{(2)}\right\Vert,
\]
and the convergence of the fixed-point is guaranteed under the CFL-like condition~\eqref{CFL_K=2_fix-point}.
\end{proof}

\begin{rem}
Notice that the amplification coefficient is composed of a linear term with respect to $\Delta t$, corresponding to $k=1$, and a quadratic term with respect to $\Delta t^2$, corresponding to $k=2$. Therefore, the restriction is mainly attributed to the first derivative. In particular, if $\Delta t\widehat{L}_2\leqslant\widehat{L}_1$ and $\Delta t\widehat{L}_1\leqslant 1/2$, we obtain the contraction of the fixed-point.
\hfill\smallBB
\end{rem}

\begin{rem}
The analysis of the convergence of successive iterations for $K>2$ follows the same ideas as for the structural schemes $\SK{1}{R}$ and $\SK{2}{R}$, and derives from the uniform boundness of the approximations and the Lipschitz property for $f(z,t)$ and its derivatives under a CFL-like condition. A generalisation of the condition for the structural scheme $\SK{K}{R}$ then reads
\[
\chi_K=\sum_{k=1}^K\Delta t^k\,\widehat{L}_k\left\Vert\boldsymbol{B}^{(k)}\right\Vert\leqslant 1,\quad\boldsymbol{B}^k=\Big(\boldsymbol{A}^{(0)}\Big)^{-1}\boldsymbol{A}^{(k)}
\]
assuming that matrix $\boldsymbol{A}^{(0)}$ is non-singular.
\hfill\smallBB
\end{rem}

\begin{rem}
Note that the CFL-like conditions come only from the stability of the fixed-point algorithm. As we shall see in the sequel, the non-linear schemes are unconditionally stable, but the proposed iterative algorithms for computing the approximate solution may require conditions to guarantee convergence.
\hfill\smallBB
\end{rem}

\section{Consistency error and convergence order}

The approximations $\phi^{(k)}_{n+r}\approx\phi^{(k)}(t_{n+r})$ are the solution of a non-linear system coupling structural and physical equations. For the consistency error and convergence order analysis, let us denote
\[
\Phi_n^{(k)}=\Big(\phi^{(k)}_{n+1},\ldots,\phi^{(k)}_{n+R}\Big)^\trans\in\mathbb{R}^R,\quad
\overline\Phi_n^{(k)}=\Big(\overline\phi^{(k)}_{n+1},\ldots,\overline\phi^{(k)}_{n+R}\Big)^\trans\in\mathbb{R}^R,
\]
\[
\Psi^K_{n}=\big(\phi^{(0)}_{n},\ldots,\phi^{(K)}_{n}\Big)^\trans\in\mathbb{R}^{K+1},\quad
\overline\Psi^K_{n}=\big(\overline\phi^{(0)}_{n},\ldots,\overline\phi^{(K)}_{n}\Big)^\trans\in\mathbb{R}^{K+1},
\]
\[
\boldsymbol{A}^{(k)}=\Big(a_{k,r}^s\Big)_{r,s=1}^R\in\mathbb{R}^{R\times R},\quad
\widehat{\boldsymbol{A}}^K_n=\Big(a_{k,0}^s\,\Delta t^k\,\Big)^{s\bint{1,R}}_{k\bint{0,K}}\in\mathbb{R}^{R\times(K+1)},
\]
with $\overline\phi^{(k)}_{n+r}=\phi^{(k)}(t_{n+r})$, $k\bint{0,K}$, $r\bint{1,R}$, the exact solution.

\begin{lemma}\label{lemma_SE_with_constant}
Adopting a column vector expression of the matrix $\widehat{\boldsymbol{A}}^K_n=[\widehat{\ba}_0,\Delta t\widehat{\ba}_1,\ldots,\Delta t^K\widehat{\ba}_K]$ with $\widehat{\ba}_k\in\mathbb{R}^R$, $k\bint{0,K}$, we have the property
\begin{equation}\label{SE_with_constant}
\Big(\boldsymbol{A}^{(0)}\Big)^{-1}\widehat{\ba}_0=-\Vone\in\mathbb{R}^R,
\end{equation}
where $\Vone=[1,1,\ldots,1]^\trans$.
\end{lemma}
\begin{proof}
The structural equations are exact for the constant polynomial function $\pi_1(t;t_n)=1$. Therefore, taking $\Phi_n^{(0)}=\Vone^\trans\in\mathbb{R}^R$, then
$\Phi_n^{(k)}=[0,0,\ldots,0]^\trans\in\mathbb{R}^R$, $k\bint{1,K}$, and $\Psi^K_{n}=[1,0,\ldots,0]^\trans\in\mathbb{R}^{K+1}$. We deduce the compatibility relation
\[
\boldsymbol{A}^{(0)}\Phi_n^{(0)}+\widehat{\boldsymbol{A}}^K_n\Psi^K_{n}=0\implies\boldsymbol{A}^{(0)}\Vone+\widehat{\ba}_0=0.
\]
Assuming that the matrix $\boldsymbol{A}^{(0)}$ is nonsingular, we obtain the relation~\eqref{SE_with_constant}.
\end{proof}

\subsection{Consistency error of the structural equations}

For the structural scheme $\SK{K}{R}$, the $R$ structural equations in compact form read
\[
\sum_{k=0}^K\Delta t^k\,\boldsymbol{A}^{(k)}\Phi^{(k)}_{n}+\widehat{\boldsymbol{A}}^K_n\Psi^K_{n}=0
\]
that is
\begin{equation}
\label{exact equation}
0=\sum_{k=0}^K\sum_{r=0}^Ra_{k,r}^s\,\Delta t^k\,\phi^{(k)}_{n+r}.
\end{equation}
These equations are satisfied for any polynomial of degree $K(R+1)$. On the one hand, using the Taylor expansion, for $h>0$, we have
\[
\phi(t_n+h)=\pi_\phi(h;t_n)+\frac{h^{K(R+1)+1}}{(K(R+1)+1)!}\phi^{(K(R+1)+1)}(t_n+\theta_0 h),\,\theta_0\in [0,1],
\]
where $\pi_\phi(h;t_n)$ is the Taylor polynomial of degree $K(R+1)$. Moreover, the Taylor expansion for the $k$-th derivative with respect to the variable $h$ is
\[
\phi^{(k)}(t_n+h)=\pi_\phi^{(k)}(h;t_n)+\frac{h^{K(R+1)+1-k}}{(K(R+1)+1-k)!}\phi^{(K(R+1)+1)}(t_n+\theta_kh),\,\theta_k\in [0,1].
\]
Consistency errors correspond to the error when the approximate solution is substituted with the exact solution in the structural equations.

\begin{prop}
Let $\varepsilon^s_n$ be the consistency error for the structural equation $\SEE{K}{R}{R}(s)$, $s\bint{1,R}$, given by
\begin{equation}\label{consistency_error}
\varepsilon^s_n=\sum_{k=0}^K\sum_{r=0}^Ra_{k,r}^s\,\Delta t^k\,\phi^{(k)}(t_{n}+r\Delta t)
=\sum_{k=0}^K\sum_{r=0}^Ra_{k,r}^s\,\Delta t^k\,\overline\phi^{(k)}_{n+r}.
\end{equation}
Then
\begin{equation}
\label{consistency_error_ineq}
|\varepsilon^s_n|\leqslant C^s_{K,R}\,\Delta t^{K(R+1)+1}
\end{equation}
with $C^s_{K,R}$ a real number depending on $\displaystyle\max_{[0,T]}|\phi^{(K(R+1)+1}|$ but independent of $\Delta t$.
\end{prop}

\begin{proof}
Using that the relation holds exactly for the Taylor polynomial $\pi(h;t_n)$ as a function of $h$, we deduce, with $h=\Delta t, 2\Delta t,\ldots, R\Delta t$, that
\begin{align*}
\varepsilon^s_n&=\sum_{k=0}^K\sum_{r=0}^Ra_{k,r}^s\,\Delta t^k\,\frac{(r\Delta t)^{K(R+1)+1-k}}{(K(R+1)+1-k)!}\phi^{(K(R+1)+1)}(t_n+\theta_{r,k}\,r\Delta t)\\
&=\Delta t^{K(R+1)+1}\,\sum_{k=0}^K\sum_{r=0}^Ra_{k,r}^s\frac{r^{K(R+1)+1-k}}{\big(K(R+1)+1-k\big)!}\phi^{(K(R+1)+1)}(t_n+\theta_{r,k}\,r\Delta t).
\end{align*}
Then
\begin{align*}
|\varepsilon^s_n|&\leqslant\Delta t^{K(R+1)+1}\,L_{K,R}\sum_{k=0}^K\sum_{r=0}^R|a_{k,r}^s|\;\frac{r^{K(R+1)+1-k}}{\big(K(R+1)+1-k\big)!},
\end{align*}
with
\[
L_{K,R}=\max_{t\in[0,T]}|\phi^{(K(R+1)+1)}(t)|.
\]
Setting
\[
C^s_{K,R}=L_{K,R}\sum_{k=0}^K\sum_{r=0}^R|a_{k,r}^s|\;\frac{r^{K(R+1)+1-k}}{\big(K(R+1)+1-k\big)!},
\]
provides the relation~\eqref{consistency_error_ineq}.
\end{proof}

If $\mathcal{E}_n=(\varepsilon^s_n)_{s\bint{1,R}}$ stands for the vector that gathers the consistency errors, then we have
\[
\displaystyle\Vert\mathcal{E}_n\Vert_\infty=C_{K,R}\,\Delta t^{K(R+1)+1}
\]
with $\displaystyle C_{K,R}=\max_{s\bint{1,R}} C^s_{K,R}$.

\subsection{Convergence order of the structural scheme $\SK{K}{R}$}

We adapt the Gr\"onwall lemma~\cite{C06} for the structural scheme $\SK{K}{R}$.

\begin{lemma}
\label{my_gronwall_lemma}
Let $(e_n)_{n=0}^N$ be a sequence of real numbers with $N=J\times R$($J$ blocks of size $R$) and assume that, for $\displaystyle\Delta t\leqslant\overline\Delta t$ with $\overline\Delta t$ to be fixed further, there exist non-negative real numbers $b^0,\ldots,b^K$, independent of $\Delta t$, such that for $j\bint{0,J-1}$
\[
|e_{(j+1)R}|\leqslant|e_{jR}|\left( 1+\sum_{k=1}^K\Delta t^k\, b^k\right)+b^0\,\Delta t^{P+1},
\]
for a certain $P\in\mathbb{N}$. Then for any $j\bint{1,J}$, we have the estimate
\begin{equation}\label{inequality_gronwall}
|e_{jR}|\leqslant\widetilde{C}\,\Delta t^{P}+\widehat{C}|e_0|,
\end{equation}
where $\widetilde{C}$ and $\widehat{C}$ are constants independent of $\Delta t$.
\end{lemma}
\begin{proof}
Given $\overline{\Delta t}$, which shall be fixed later, we define
\[
C=\sum_{k=1}^K\overline\Delta t^{k-1}\, b^k.
\]
From the assumption and noting that $\Delta t/\overline{\Delta t}\leqslant 1$, we obtain the inequality
\[
|e_{(j+1)R}|\leqslant|e_{jR}|( 1+C\Delta t)+b^0\,\Delta t^{P+1}
\]
for any $j\bint{0,J-1}$. By induction, it yields
\begin{align*}
|e_{(j+1)R}|&\leqslant( 1+C\Delta t)^j|e_0|+b^0\,\Delta t^{P+1}\sum_{\sigma=0}^j( 1+C\Delta t)^\sigma\\
&\leqslant\exp(C j\Delta t)|e_0|+b^0\,\Delta t^{P}\sum_{\sigma=0}^j\Delta t\exp(C\sigma\Delta t)\\
&\leqslant\exp(CT)|e_0|+b^0\,\Delta t^{P}\int_{0}^{T}\exp(C\xi)\,\textrm{d}\xi\\
&\leqslant\exp(CT)|e_0|+b^0\,\Delta t^{P}\frac{\exp(CT)-1}{C}.
\end{align*}
Therefore, we deduce the inequality~\eqref{inequality_gronwall} with
\[
\widetilde{C}=b^0\,\frac{\exp(CT)-1}{C},\quad\widehat{C}=\exp(CT).
\]
\end{proof}
We define the {\it a priori} error (or convergence error) as the quantity $e^{(k)}_{n}=\overline\phi^{(k)}_{n}-\phi^{(k)}_{n}$ for $n\bint{0,N}$ and $k\bint{0,K}$, and $\boldsymbol{E}_n^{(k)}=[e_{n+1}^{(k)},\ldots,e_{n+R}^{(k)}]^\trans$.

\begin{prop}
\label{pro_error_end_of_block}
There exists $\overline{\Delta t}$ such that, for $\Delta t\leqslant\overline{\Delta t}$, we have the estimate
\begin{equation}
\label{estimate_blocl_sR}
\big\vert e^{(0)}_{jR}\big\vert\leqslant\overline{C}\,\Delta t^{K(R+1)},
\end{equation}
for any $j\bint{1,J}$, $jR\leqslant N$, where $\overline{C}$ is a real number independent of $\Delta t$.
\end{prop}

\begin{proof}
The difference between relations~\eqref{consistency_error} and~\eqref{exact equation} give
\[
\varepsilon^s_n=\sum_{k=0}^K\sum_{r=0}^Ra_{k,r}^s\,\Delta t^k\,(\overline\phi^{(k)}_{n+r}-\phi^{(k)}_{n+r})=\sum_{k=0}^K\sum_{r=0}^Ra_{k,r}^s\,\Delta t^k\,e_{n+r}^{(k)},
\]
which can be rewritten in the compact form as
\[
\mathcal{E}_n=\sum_{k=0}^K\Delta t^k\,\boldsymbol{A}^{(k)}\boldsymbol{E}_n^{(k)}+\sum_{k=0}^K\widehat{\ba}_k\,\Delta t^k\,e^{(k)}_n.
\]
Using Lemma~\ref{lemma_SE_with_constant}, we rewrite the consistency error in the form
\begin{align*}
\boldsymbol{E}_n^{(0)}= e^{(0)}_n\Vone-\sum_{k=1}^K\,\Delta t^k\,\boldsymbol{B}^{k}\boldsymbol{E}_n^{(k)} -\sum_{k=1}^K\,\Delta t^k\,\bb^k e^{(k)}_n+\Big(\boldsymbol{A}^{(0)}\Big)^{-1}\mathcal{E}_n
\end{align*}
with $\boldsymbol{B}^{k}=\Big(\boldsymbol{A}^{(0)}\Big)^{-1}\boldsymbol{A}^{(k)}$ and $\bb^k=\Big(\boldsymbol{A}^{(0)}\Big)^{-1}\widehat{\ba}_k$.

Using the $L^\infty$-norm $\Vert\cdot\Vert_\infty$ and the inductive matrix norm (also denoted $\Vert\cdot\Vert_\infty$), we have
\begin{align*}
\bVert\boldsymbol{E}_n^{(0)}\bVert_\infty
& \leqslant\bvert e^{(0)}_n\bvert+
\sum_{k=1}^K\,\Delta t^k\,\left\Vert\boldsymbol{B}^{k}\right\Vert_\infty\bVert\boldsymbol{E}_n^{(k)}\bVert_\infty\\
& \quad
+\sum_{k=1}^K\,\Delta t^k\,\left\Vert\bb^k\right\Vert_\infty\bvert e^{(k)}_n\bvert
+\left\Vert\Big(\boldsymbol{A}^{(0)}\Big)^{-1}\right\Vert_\infty\bVert\mathcal{E}_n\bVert_\infty.
\end{align*}
From inequalities~\eqref{Vectorial_Lipschitz_0}, we then deduce
\[
\bVert\boldsymbol{E}_n^{(k)}\bVert_\infty\leqslant\widehat{L}_k\bVert\boldsymbol{E}_n^{(0)}\bVert_\infty
\textrm{ and }
\big\vert e^{(k)}_n\big\vert\leqslant\widehat{L}_k\big\vert e^{(0)}_n\big\vert.
\]
Using the consistency error estimate, we obtain
\begin{equation}\label{estimate_all_block}
C_f\,\Delta t\,\big\Vert\boldsymbol{E}_n^{(0)}\big\Vert_\infty\leqslant C_b\,\Delta t\,\big\vert e^{(0)}_n\big\vert
+\left\Vert\Big(\boldsymbol{A}^{(0)}\Big)^{-1}\right\Vert_\infty C_{K,R}\,\Delta t^{K(R+1)+1},
\end{equation}
where
\[
C_f\,\Delta t=1-\sum_{k=1}^K\,\Delta t^k\,\widehat{L}_k\left\Vert\boldsymbol{B}^{k}\right\Vert_\infty,\quad
C_b\,\Delta t=1+\sum_{k=1}^K\,\Delta t^k\,\widehat{L}_k\left\Vert\bb^k\right\Vert_\infty.
\]
Due to the inequality $\bvert e^{(0)}_{n+R}\bvert\leqslant\bVert\boldsymbol{E}_n^{(0)}\bVert_\infty$, we deduce
\[
\big\vert e^{(0)}_{n+R}\big\vert\leqslant\frac{C_b\,\Delta t}{C_f\,\Delta t}\big\vert e^{(0)}_n\big\vert+
\left\Vert\Big(\boldsymbol{A}^{(0)}\Big)^{-1}\right\Vert_\infty\frac{C_{K,R}}{C_f\,\Delta t}\,\Delta t^{K(R+1)+1}.
\]
We now define $\overline\Delta t>0$ such that $C_f\,\overline\Delta t=1/2$. Consequently, for any $\Delta t<\overline\Delta t$, we have $C_f\,\Delta t\geqslant 1/2$, and, for $\Delta t\leqslant\overline\Delta t$, we obtain the inequality
\[
\frac{C_b\,\Delta t}{C_f\,\Delta t}= 1+\frac{C_b\,\Delta t-C_f\,\Delta t}{C_f\,\Delta t}\leqslant1+
2\sum_{k=1}^K\,\Delta t^k\,\widehat{L}_k\left(\left\Vert\boldsymbol{B}^{k}\right\Vert_\infty+\left\Vert\bb^k\right\Vert_\infty\right)
\]
We finally conclude that, for $\Delta t\leqslant\overline\Delta t$, we have
\[
\big\vert e^{(0)}_{n+R}\big\vert\leqslant\big\vert e^{(0)}_n\big\vert\left(1+\sum_{k=1}^K\,\Delta t^k\,b^k\right)
+ b^0\,\Delta t^{K(R+1)+1},
\]
with $\displaystyle b^k=2\widehat{L}_k\left(\left\Vert\boldsymbol{B}^{k}\right\Vert_\infty+\left\Vert\bb^k\right\Vert_\infty\right)$ and $b^0=2C_{K,R}$, where the non-negative real numbers $b^0,\ldots,b^K$ do not depend on $\Delta t$.

Since we treat the differential equation using blocks of size $R$, we have $n=jR$ and from Lemma~\ref{my_gronwall_lemma}, the relation~\eqref{inequality_gronwall} holds with $P=K(R+1)$, that is
\[
\big\vert e^{(0)}_{jR}\big\vert\leqslant\widetilde{C}\,\Delta t^{K(R+1)}+\widehat{C}|e^{(0)}_0|.
\]
If we use the exact solution at time $t=0$ for the initial condition, that is $e^{(0)}_0=0$, we deduce that $\overline{C}=\widetilde{C}$ and the error is of order $\Delta t^{K(R+1)}$.
\end{proof}
\begin{theorem}
Let $n=jR$, $j\bint{1,J}$. Under the assumption of Proposition~\ref{pro_error_end_of_block}, there exists a real value $C$ independent of $\Delta t$, $r$, $j$, and $k$ such that, if $\Delta t\leqslant\overline{\Delta t}$, then
\begin{equation}\label{error_convergence_estimate}
\big\vert e^{(k)}_{n+r}\big\vert\leqslant C\,\Delta t^{K(R+1)}.
\end{equation}
The scheme is of order $K(R+1)$ at any point and for any derivative.
\end{theorem}
\begin{proof}
From relation~\eqref{estimate_all_block} and $\displaystyle\big\vert e^{(0)}_{n+r}\big\vert\leqslant\big\Vert\boldsymbol{E}_n^{(0)}\big\Vert_\infty$, we deduce
\[
\big\vert e^{(0)}_{n+r}\big\vert\leqslant\big\vert e^{(0)}_n\big\vert\left(1+\sum_{k=1}^K\Delta t^k\, b^k\right)
+ b^0\,\Delta t^{K(R+1)+1},
\]
and, with the estimate~\eqref{estimate_blocl_sR}, we deduce the zero-derivative convergence order
\[
\big\vert e^{(0)}_{n+r}\big\vert\leqslant\widetilde{C}\,\Delta t^{K(R+1)}\,\left(1+\sum_{k=1}^K\,\Delta t^k\, b^k\right)+ b^0\,\Delta t^{K(R+1)+1}
\leqslant C^0\,\Delta t^{K(R+1)},
\]
where $C^0$ does not depend on $\Delta t$. On the other hand, using the Lipschitz property of Lemma~\ref{lemma_bounded_Lipschitz_z_0}, we get $\displaystyle\big\vert e^{(k)}_{n+r}\big\vert\leqslant\widehat{L}_k\big\vert e^{(0)}_{n+r}\big\vert\leqslant C^k\,\Delta t^{K(R+1)}$, with $C^k=\widehat{L}_k C^0$. Let $C=\max C^k$ over $k\bint{0,K}$, and we deduce the convergence errors~\eqref{error_convergence_estimate}.
\end{proof}

\section{Linear stability}
\label{subsec_Stab}

The linear stability of the structural schemes is provided by the von Neumann stability analysis. We compute the growth factor of the Fourier modes and check the magnitude of the block approximations. The analysis is based on the solution of the linear ODE $\phi^{(1)}=\lambda\phi^{(0)}$ with $\lambda\in\mathbb{C}$, that is, $f(z,t)=\lambda z$.

\subsection{Linear stability condition}

The first physical equation $\PE{1}(n+r)$ reads $\phi_{n+r}^{(1)}=\lambda\phi_{n+r}^{(0)}$, $r\bint{0,T}$, and the $k$-th physical equation $\PE{k}(n+r)$ is given by $\phi_{n+r}^{(k)}=\lambda^k\phi_{n+r}^{(0)}$. Substituting $\phi^{(k)}_{n+r}$ with $\lambda^k\phi^{(0)}_{n+r}$, $r\bint{0,R}$, into the $R$ structural equations, provides the system
\begin{equation}
\label{eq_stab_pol}
\sum_{k=0}^K\sum_{r=0}^R a_{k,r}^s(\lambda\Delta t)^k\phi^{(0)}_{n+r}=0,\quad s\bint{1,R}.
\end{equation}
The relation~\eqref{eq_stab_pol} is rewritten with the variable $z=\lambda\Delta t$, yielding
\[
\sum_{r=1}^R\left(\sum_{k=0}^K a_{k,r}^s z^k\right)\phi^{(0)}_{n+r}=-\left(\sum_{k=0}^K a_{k,0}^s z^k\right)\phi^{(0)}_{n}.
\]
Denoting
\[
g_{s,r}(z)=\sum_{k=0}^K a_{k,r}^s z^k,\quad
\boldsymbol{G}(z)=\Big(g_{s,r}(z)\Big)_{s,r=1}^{R},\quad
\bb(z)=\Big(g_{s0}(z)\Big)_{s=1}^{R},
\]
the relation~\eqref{eq_stab_pol} can be expressed in the following compact matrix form
\begin{equation}
\label{eqSec3G}
\boldsymbol{G}(z)\Phi_n^{(0)}=-\bb(z)\phi_n^{(0)}\implies\Phi_n^{(0)}=\bchi(z)\phi_n^{(0)},
\end{equation}
where $\bchi(z)=-\boldsymbol{G}^{-1}(z)\bb(z)\in\mathbb{R}^R$ is the vector whose entries are $\chi_r(z)$, $r\bint{1,R}$, representing the amplification function. For stability, each entry must satisfy $|\chi_r(z)|\leqslant 1$.

Because the computation of the next block only uses the vector $\Phi_{n+R}$, our analysis focuses on the amplification function $\chi_{R}(z)$ as a function of the parameter $z$, derived from system~\eqref{eqSec3G}.

We introduce an adaptation of the A-stability for the block structural method.
\begin{defn}
The scheme is A-stable if $|\chi_{R}(z)|\leqslant 1$.
\end{defn}

Since there are multiple ways to choose the set of $R$ structural equations, we establish that the stability condition does not depend on the particular basis chosen for the kernel.

\begin{prop}
Let $\boldsymbol{A}=[a^s]_{R\times M}$ and $\widetilde{\boldsymbol{A}}=[\widetilde{a}^s]_{R\times M}$ be two basis of the Kernel $\mathcal K^R$ given in Eq.~\eqref{eq_basis}. Then
\[
\boldsymbol{G}(z)\Phi_n^{(0)}=-{\bb}(z)\phi_n^{(0)}
\quad\text{and}\quad
\widetilde{\boldsymbol{G}}(z)\Phi_n^{(0)}=-\widetilde{\bb}(z)\phi_n^{(0)},
\]
where, $\widetilde{\boldsymbol{G}}(z)$ and $\widetilde{\bb}(z)$ refer to the matrix and vector defined as in Eq.~\eqref{eqSec3G}, but associated with the alternative basis $\widetilde{a}^s$. The amplification vector $\bchi(z)$ therefore remains the same regardless of the basis chosen, showing that stability is determined only by properties of the kernel subspace.
\end{prop}

\begin{proof}
Let $\widetilde{\boldsymbol{A}}$ and $\boldsymbol{A}$ be two different kernel bases, as in Eq.~\eqref{eq_basis}. We have $\widetilde{\boldsymbol{A}}=\boldsymbol{\Xi}\boldsymbol{A}$ where $\boldsymbol{\Xi}$ is a nonsingular $R\times R$ matrix that changes the basis, that is, $\widetilde{\boldsymbol{G}}(z)=\boldsymbol{\Xi}\boldsymbol{G}(z)$ and $\widetilde\bb(z)=\boldsymbol{\Xi}\bb(z)$. From Eq.~\eqref{eqSec3G}, we have $\Phi_n^{(0)}=\bchi(z)\phi_n^{(0)}$ and, using the alternative basis, $\widetilde{\boldsymbol{G}}(z)\Phi_n^{(0)}=-\widetilde{\bb}(z)\phi_n^{(0)}$, hence
\[
\widetilde\bchi(z)=-\widetilde{\boldsymbol{G}}^{-1}(z)\widetilde\bb(z)=-\Big(\boldsymbol{\Xi}\boldsymbol{G}(z)\Big)^{-1}\Xi\bb(z)=-\boldsymbol{G}^{-1}(z)\big(\boldsymbol{\Xi}^{-1}\boldsymbol{\Xi}\big)\bb(z)=\bchi(z).
\]
As a result, all amplification functions $\chi_r$, $r\bint{1,R}$, are independent of the kernel basis chosen. Consequently, the stability condition $|\chi_R(z)|\leqslant 1$ for the approximation at time $t_{n+R}$ is unaffected by the choice of basis.
\end{proof}

To assess stability, we propose a new expression for the amplification function.

\begin{prop}
The amplification coefficient at time $t_{n+R}$ is given by
\begin{equation}
\label{formal_expression_chi}
\chi_{R}(z)=\frac{\det\Big[\boldsymbol{G}(z)\big(\Id-\be_R\,\be_R^\trans\big)-\bb(z)\be_R^\trans\Big]}{\det\big[\boldsymbol{G}(z)\big]},
\end{equation}
where $\Id$ is the identity matrix of size $R\times R$ and $\be_R$ is the last canonic column vector. In particular, $\chi_{R}(z)$ is a rational function.
\end{prop}
\begin{proof}

For simplicity, we suppress the explicit $z$ dependence and apply Cramer's rule to the system~\eqref{eqSec3G} to determine the last unknown. Consider the upper triangular matrix $\Id-\be_R\,\be_R^\trans+\Phi_n^{(0)}\be_R^\trans$, which is the identity matrix with its $R$-th column replaced by the vector $\Phi_n^{(0)}$. Multiplying matrix $\boldsymbol{G}$ on the left side, yields
\[
\boldsymbol{G}\,\Big(\Id-\be_R\,\be_R^\trans+\Phi_n^{(0)}e_R^\trans\Big)=\boldsymbol{G}\,\Big(\Id-\be_R\,\be_R^\trans\Big)-\bb\phi_n^{(0)}\be_R^\trans.
\]
On the one hand, $\displaystyle\det\big[\Id-\be_R\,\be_R^\trans+\Phi_n^{(0)}\be_R^\trans\big]=\phi_{n+R}^{(0)}$, thus
\[
\det\Big[\boldsymbol{G}\big(\Id-\be_R\,\be_R^\trans+\Phi_n^{(0)}\be_R^\trans\big)\Big]=\det[\boldsymbol{G}]\phi_{n+R}^{(0)}.
\]
which, after some algebraic manipulations, provides the expression
\[
\det\Big[\boldsymbol{G}\big(\Id-\be_R\,\be_R^\trans+\Phi_n^{(0)}\be_R^\trans\big)\Big]=
\det\Big[\boldsymbol{G}\big(\Id-\be_R\,\be_R^\trans\big)-\bb\phi_n^{(0)}\be_R^\trans\Big]=
\det\Big[\boldsymbol{G}\big(\Id-\be_R\,\be_R^\trans\big)-\bb\be_R^\trans\Big]\phi_n^{(0)},
\]
and finally
\[
\phi_{n+R}^{(0)}(z)=\frac{\det\Big[\boldsymbol{G}(z)\big(\Id-\be_R\,\be_R^\trans\big)-\bb(z)\,\be_R^\trans\Big]}{\det\big[\boldsymbol{G}(z)\big]}\,\phi_n^{(0)},
\]
and the relation~\eqref{formal_expression_chi} holds. Since the entries of the matrix $\boldsymbol{G}$ and vector $\bb$ are polynomials of $z$ of degree $K$ with real coefficients, the amplification function $\chi_{R}$ is therefore a rational function with real coefficients.
\end{proof}

\subsection{Criteria for the unconditional stability}

Next, we introduce several tools for analysing the stability of the structural scheme \SK{K}{R}. This approach assumes that the transfer function has the form $\chi_R(z)=p(z)/p(-z)$, where $p(z)$ is a polynomial. With the Routh-Hurwitz criterion~\cite{Poz2021}, we thus demonstrate that if all roots of $p(z)$ have negative real parts, the scheme is A-stable.

\begin{lemma}
\label{lema1}
Assume that the transfer function associated with the structural scheme \SK{K}{R} is given by $\chi_R(z)=p(z)/p(-z)$, where $p(z)$ is a polynomial in $\mathbb{C}$ in which all roots have a negative real part. Then, the structural scheme \SK{K}{R} is A-stable.
\end{lemma}

\begin{proof}
If all roots of $p(z)$ are in the left-half of the complex plane, then all roots of $p(-z)$ lie in the right-half. Thus, $\chi_R(z)$ is analytic in the left-half plane.
According to the Maximum Modulus Theorem, the maximum value of $|\chi_R|$ occurs at the boundary of the semicircular region bounded by the imaginary axis from $-\rho$ to $\rho$ and the semicircle $z=\rho\exp(\complexnum{i}\theta)$, $\theta\in[\pi/2,3\pi/2]$.
Since $\chi_R(z)=p(z)/p(-z)$ then, for any $\theta\in[\pi/2,3\pi/2]$, $\lim_{\rho\to+\infty}|\chi_R(z)|=1$.
On the other hand, for imaginary numbers $z=\complexnum{i} y$, for all $y\in\mathbb{R}$, and taking advantage of the fact that the polynomial coefficients are real numbers, we have
\[
|\chi_R(\complexnum{i} y)|=\left|\frac{p(\complexnum{i} y)}{p(-\complexnum{i} y)}\right|
=\left|\frac{\,p(\complexnum{i} y)\,}{\overline{p(\complexnum{i} y)}}\right|=1.
\]
We then conclude that $\displaystyle|\chi_R(z)|\leqslant 1$ for $\real{z}\leqslant 0$, that is, the structural scheme \SK{K}{R} is A-stable.
\end{proof}

In summary, Lemma~\eqref{lema1} stipulates that unconditional stability requires: (1) $\chi_R(z)=p(z)/p(-z)$, and (2) all zeros of $p(z)$ are in the left-half of the complex plane. To check the second condition, we demonstrate that $p(z)$ is Hurwitz stable~\cite{Hurwitz1895}, which is verifiable via the Routh-Hurwitz criterion~\cite{Hurwitz1895,Routh1877} as detailed below.

\begin{defn}
Let
\begin{equation}
\label{HurPoly}
p(z)=a_Iz^I+a_{I-1}z^{I-1}+\cdots+a_1z+a_0
\end{equation}
be a polynomial with real coefficients. The $I\times I$ matrix
\[
\boldsymbol{H}=
\begin{bmatrix}
 a_{I-1} & a_{I-3} & a_{I-5} &\cdots&\cdots& \cdots&0&0&0\\
 a_I & a_{I-2} & a_{I-4} & \cdots&\cdots& \cdots &0&0&0\\
 0 & a_{I-1} & a_{I-3} &\cdots &\cdots& \cdots&0 &0&0\\
 0& a_I & a_{I-2} & \cdots & \cdots& \cdots&0&0&0\\
 0&0& a_{I-1} &\cdots& \cdots &\cdots&a_0&0&0\\
 0&0 & a_I & \cdots & \cdots & \cdots&a_1&0&0\\
 0&0&0& \cdots & \cdots & \cdots&a_2&a_0&0\\
 \vdots & \vdots & \vdots & \vdots & \vdots & \vdots& a_3&a_1&0\\
 0&0&0& \cdots&\cdots&\cdots& a_4&a_2&a_0
\end{bmatrix}
\]
is termed the Hurwitz matrix associated with $p(z)$.
\end{defn}

\begin{defn}
The polynomial $p(z)$ is called the Hurwitz polynomial if $\real{z_i}<0$ for all $i\bint{1,I}$, where $z_i$ are the complex roots of $p(z)$.
\end{defn}

The Routh-Hurwitz criterion is given in the following theorem (see~\cite{Poz2021}).

\begin{theorem} The polynomial $p(z)$ given in Eq.~\eqref{HurPoly} is a Hurwitz polynomial if and only if the principal minors of the Hurwitz matrix $\boldsymbol{H}$ are strictly positive.
\end{theorem}

From Lemma~\ref{lema1} and the definition of a Hurwitz polynomial, we deduce the following theorem.

\begin{theorem}
\label{TheorStab}
Let $\chi_R(z)=p(z)/p(-z)$ be the transfer function at time $t_{n+R}$ associated with the structural scheme \SK{K}{R}. If $p(z)$ is a Hurwitz polynomial in $\mathbb{C}$, then the structural scheme \SK{K}{R} is A-stable.
\end{theorem}

There is no generic result showing that all structural schemes \SK{K}{R} are unconditionally stable; we must check the A-stability of each scheme individually. To this end, the stability of several structural schemes \SK{K}{R} has been studied by computation of the $\chi_R(z)$ function given in Tables~\ref{table_stab_K1},~\ref{Table_stab_K2}, and~\ref{table_stab_K3} in~\ref{anexe::stability}.
The expressions $\chi_{R}(z)$ of all these schemes have the form $\chi_R(z)=p(z)/p(-z)$, where $p(z)$ is a polynomial of real positive coefficients. Verification of the Routh-Hurwitz criterion is also given in~\ref{anexe::stability}.

We summarise the results in the following corollary and provide, in the proof, the methodology for checking A-stability using Hurwitz polynomials.
\begin{cor}
\label{PropStabK1}
The structural schemes \SK{K}{R}, $K,R=1,2,3$ are A-stable.
\end{cor}

\begin{proof}
The methodology for providing evidence of A-stability is illustrated with the following cases.
\begin{itemize}
\item \textbf{Structural scheme \SK{1}{3}}. The Hurwitz polynomial $p(z)=3z^3+11z^2+18z+12$ is a third-degree polynomial with all positive coefficients (see~\ref{anexe::stability}), while the associated Hurwitz matrix reads
\[
\boldsymbol{H}_{13}=
\begin{bmatrix}
11 & 12 & 0\\
3 & 18 & 0\\
0 & 11 & 12
\end{bmatrix}.
\]
Calculations show that all principal minors $m_i$ are strictly positive and, from Theorem~\ref{TheorStab}, we conclude that the structural scheme \SK{1}{3} is A-stable.

\item \textbf{Structural scheme \SK{2}{2}}. The Hurwitz polynomial $p(z)=z^4+9z^3+39z^2+90z+90$ is a fourth-degree polynomial with all positive coefficients (see~\ref{anexe::stability}), while the associated Hurwitz matrix reads
\[
\boldsymbol{H}_{22}=
\begin{bmatrix}
9 & 90 & 0 & 0\\
1& 39 & 90 & 0\\
0 & 9 & 90 &0\\
0 & 1& 39 & 90
\end{bmatrix}.
\]
Calculations show that all principal minors $m_i$ are strictly positive and, from Theorem~\ref{TheorStab}, we conclude that the structural scheme \SK{2}{2} is A-stable.
\end{itemize}
\end{proof}

\begin{rem}
Since increasing the values of $K$ and $R$ yields Hurwitz polynomials with cumbersome expressions, the above methodology is illustrated only for the structural schemes \SK{1}{3} and \SK{2}{2}.
\hfill\smallBB
\end{rem}

\subsection{About the stability of other classes of structural method}
\rb{The \SK{K}{R} schemes is a particular case of structural schemes. Indeed, Given a block of size $R$ and considering  a scheme based on $K$ derivatives, we have $R(K+1)$ unknowns that requires the same number of equations. The schemes we propose use $K$ physical equations per node, thus $RK$ physical equations, we complete with $R$ structural equations. But other combinations are available. For instance we just use $K-1$ physical equations per nodes and then we need $2R$ structural equations to close the system.
\\
More generaly, any combination between a set of physical equations and a set  of structural equation provides a potential structural scheme but some combination are not consistant (not enough physical equations) and others are not very efficient.
\\
Our conjecture is that a scheme that does not employ all the possible physical equations suffers of a reduction of the order and unconditional stability is lost. Several experiments (not presented here) show that we obtain a conditional stability condition that restrict the size of the time step. Such "weaker scheme" may be interesting if one does not want (or cannot) to compute all the relations $\ddf{k}{.,.}$, $k=0,\ldots,K-1$.
}

\section{Spectral resolution}

Fourier analysis is the standard tool for assessing the spectral resolution of a compact method. The pioneer paper of Lele~\cite{Lele1992} provides an explicit relation for the first and second derivative approximations with respect to the local wave number, $\omega=2\pi\kappa\Delta t$. More precisely, the analysis is based on two linear relations (two structural equations, indeed) that connect zero- and first-derivative approximations on one side and zero- and second-derivative approximations on the other, as follows.
\begin{itemize}
\item The first relation for the zero and first derivative approximations is in the form
\[
\sum_{r=1}^Ra^0_r\phi^{(0)}_{n+r}+\sum_{r=1}^Ra^1_r\phi^{(1)}_{n+r}=0.
\]
By introducing the local modified wave number $\omega^{(1)}$ for the first derivative, we obtain an algebraic expression in terms of $\omega$. In this context, we assume the modified wavenumber remains constant at each discrete time $t_{n+r}$.

\item The second relation for the zero and second derivative approximations is in the form
\[
\sum_{r=1}^R b^0_r\phi^{(0)}_{n+r}+\sum_{r=1}^R b^2_r\phi^{(2)}_{n+r}=0,
\]
that provides an explicit algebraic relation of the modified wave number $\omega^{(2)}$ as a function of $\omega$, independent of the discrete time $t_{n+r}$.
\item The extension to the $k$-th derivative is based on an implicit coupling of $\phi^{(k)}_r$ with $\phi^{(0)}_r$~\cite{CaTy22} by introducing the modified wave number $\omega^{(k)}$. Once again, spectral analysis using von Neumann analysis is available, since we connect the function solely to the $k$-th derivative.
\end{itemize}
The Compact Combined Scheme (CCS)~\cite{Chu1998,Chu1999,Chu2000} uses the same approach but with a slight modification. More specifically, it combines the zero, first, and second derivative approximations at the same time but with two linearly independent equations, leading to a small $2\times 2$ system, which provides the modified wave numbers $\omega^{(1)}$ and $\omega^{(2)}$ for the first and second derivatives as a function of $\omega$, respectively.

\subsection{A new notion of spectral resolution}

Such an approach presents challenges for the structural equations. First, we want to assess the impact of the structural scheme over the entire block, particularly at the last time step $t_{n+R}$ relative to the initial time step $t_n$. Second, when $K\neq R$, there is no square, non-singular system to provide a direct solution for $\omega^{(1)}(\theta)$ up to $\omega^{(K)}(\theta)$. Finally, standard Fourier analysis assumes stationarity, so the phase shifts of the first and second derivatives are invariant with respect to $t_n$, which is not the case when dealing with an initial condition problem where we prescribe all the derivatives at $t=0$.

We propose a new approach for assessing spectral resolution tailored to the structural scheme \SK{K}{R}, specifically addressing the coupling between structural and physical equations. For this, we replace the fundamental solution $\exp(\complexnum{i}2\pi\kappa t)$ with the equation $\phi'=\complexnum{i}2\pi\kappa\phi$. Then, we connect all the derivatives to the solution thanks to the physical equations as
\[
\phi^{(k)}=\Big (\complexnum{i}2\pi\kappa\Big )^k\phi^{(0)},\ k\bint{1,K}.
\]
As a result, the $R$ structural equations yield an $R$-dimensional linear system as previously defined in Eq.~\eqref{eqSec3G}, for the specific case $z=\lambda\Delta t$ with $\lambda =\complexnum{i}2\pi\kappa$. To analyse diffusion and dispersion in the structural schemes \SK{K}{R}, we set the initial block condition at $t_n$ to the exact solution $\phi^{(0)}_n=\exp(\complexnum{i}2\pi\kappa t_n)$. The resulting transfer function, $\chi_r(\complexnum{i}\omega)$, describes the diffusion and dispersion at time $t_{n+r}$. The exact solution at $t_{n+r}$ is $\phi^{(0)}(t_{n+r})=\exp(\complexnum{i} r\omega)\exp(\complexnum{i}2\pi\kappa t_n)$, so the deviation from the exact solution is given by the ratio between the approximation and the exact value at $t_{n+r}$:
\[
\frac{\phi_{n+r}^{(0)}}{\phi^{(0)}(t_{n+r})}=\frac{\chi_r(\complexnum{i}\omega)\phi_n^{(0)}}{\exp(\complexnum{i} r\omega)\exp(\complexnum{i}2\pi\kappa t_n)}=\chi_r(\complexnum{i}\omega)\exp(-\complexnum{i}r\omega).
\]

\begin{defn}
The diffusion of the scheme is given by the modulus of $\zeta_r(\omega)=\chi_r(\complexnum{i}\omega)\exp(-\complexnum{i} r\omega)$ (amplification factor), while the dispersion corresponds to its argument (phase deviation).
\end{defn}

\begin{rem}
It is important to note that the function $\zeta_r$ also depends on the parameters $R$ and $K$, thus we employ the notation $\zeta_r(\omega;K,R)$ when necessary (as in Section~\ref{sec:zeta_R_K_dependency}).
\hfill\smallBB
\end{rem}

Noting that the solution at time $t_{n+R}$ will be used for the next block resolution, we focus our study on the case $r=R$. Of course, all intermediate steps can be analysed in a similar way.
\begin{prop}
If $\chi_R(z)$ is a rational function of the form $p(z)/p(-z)$ then we have $|\chi_R(\complexnum{i}\omega)|=1$ from which $|\zeta_R(\omega)|=1$ for all $\omega>0$. The schemes do not have diffusion at time $t_{n+R}$.
\end{prop}

\begin{rem}
For $K,R=1,2,3$, we have explicitly determined the amplification factor $\chi_R(z)$, case by case, as the quotient $p(z)/p(-z)$. We make the conjecture that this is always the case for any $K$ and $R$.
\hfill\smallBB
\end{rem}

\subsection{Phase error analysis}

Thanks to the explicit expressions of the coefficients of the structural equations, we get an analytical expression of the transfer functions $\chi_R(z)$ as rational functions for a set of situations given in Tables~\ref{table_stab_K1},~\ref{Table_stab_K2}, and~\ref{table_stab_K3} in~\ref{anexe::stability}. We then produce the polynomial expression as a function of $\omega$, which we report in Table~\ref{polynomial_omega} and use to compute the $\zeta_R(\omega)$ function.

\begin{table}[ht]
\centering
\caption{Spectral deviation -- polynomial expressions with respect to the local wavelength $\omega$ as function of $K$ and $R$.}
\label{polynomial_omega}
\renewcommand{\arraystretch}{1}
\begin{tabular}{@{}lll@{}}
\toprule
scheme && $p(\complexnum{i}\omega)=p(\complexnum{i}\omega;K,R)$\\
\midrule
\SK{1}{1} && $2+\complexnum{i}\omega$\\
\SK{1}{2} && $-\omega^2+3\complexnum{i}\omega+3$\\
\SK{1}{3} && $-3\complexnum{i}\omega^3 -11\omega^2+18\complexnum{i}\omega+12$\\
\SK{2}{1} && $-\omega^2+6\complexnum{i}\omega+12$\\
\SK{2}{2} && $\omega^4-9\complexnum{i}\omega^3-39\omega^2+90\complexnum{i}\omega+90$\\
\SK{2}{3} && $-3\omega^6+33\complexnum{i}\omega^5+193\omega^4-720\complexnum{i}\omega^3-1740\omega^2+2520\complexnum{i}\omega+1680$\\
\SK{3}{1} && $-\complexnum{i}\omega^3-12\omega^2+60\complexnum{i}\omega+120$\\
\SK{3}{2} && $-\omega^6+18\complexnum{i}\omega^5+165\omega^4-945\complexnum{i}\omega^3-3465\omega^2+7560\complexnum{i}\omega+7560$\\
\SK{3}{3} && $9\complexnum{i}\omega^9+198\omega^8-2355\complexnum{i}\omega^7-19075\omega^6+112770\complexnum{i}\omega^5$\\
{} && $\quad+\ 494760\omega^4-1587600\complexnum{i}\omega^3-3553200\omega^2+4989600\complexnum{i}\omega+3326400$\\
\bottomrule
\end{tabular}
\end{table}

The phase deviation (in radians) as a function of the local wavelength $\omega$ is plotted in Figure~\ref{figure::Dispersion} for $R=1,2,3$ with $K=1,2,3$. As expected, spectral resolution improves drastically with increasing $K$, whereas the impact of $R$ is less obvious. We observe that large deviations occur for $\omega\geqslant\pi$ when $R=3$. Particular attention is paid to the case with $R=2$ and $K=3$, where the phase deviation remains very small until $\omega=2\pi$. Indeed, only three points approximate a complete revolution with a tiny phase deviation, below $\pi/20$.

\begin{figure}[!ht]
\centering
\begin{tabular}{ccc}
\includegraphics[width=0.3\textwidth]{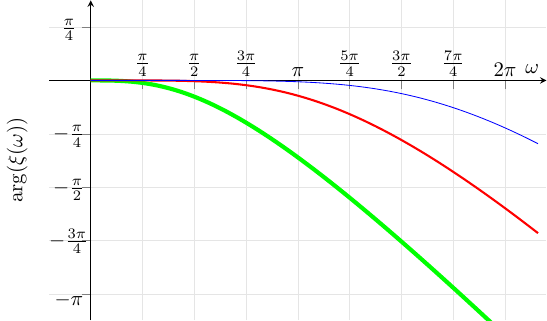}
& \includegraphics[width=0.3\textwidth]{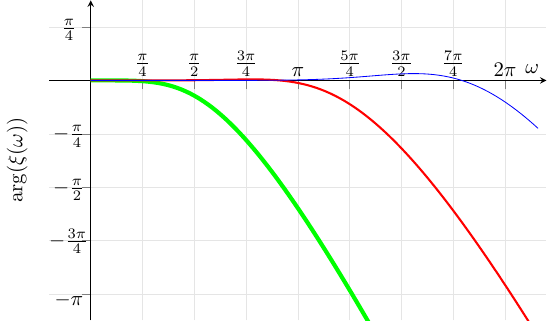}
& \includegraphics[width=0.3\textwidth]{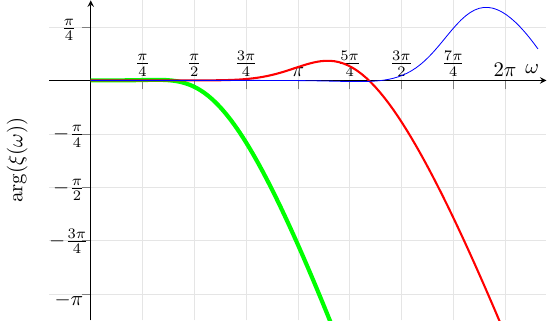}\\\\
(a) \SK{K}{1}. & (b) \SK{K}{2}. & (c) \SK{K}{3}.
\end{tabular}
\caption[]{Dispersion curves $\arg(\zeta_R)$ as function of $\omega$ with $K=1$ (\greenline), $K=2$ (\redline), and $K=3$ (\blueline).}
\label{figure::Dispersion}
\end{figure}

\subsection{Spectral resolution analysis with respect to $\Delta t$ and $K$}

Given $N$ and $\omega$, the impact of duplicating the number of points on the phase deviation is assessed. Indeed, taking $2N$ points provides the time parameter $\Delta t/2$ and, consequently, the local wavelength $\omega/2$, but the number of blocks is duplicated. Hence, the reference phase deviation $\zeta_R(\omega)$, corresponding to $N$ points, must be compared with $\zeta_R^2(\omega/2)$, which involves $2N$ points. More generally, for $\ell\,N$ points, the phase deviation $\zeta_R(\omega)$ must be compared with $\zeta_R^\ell(\omega/\ell)$, since $\ell$ blocks are necessary to reach the same final time.

\begin{figure}[!ht]
\centering
\begin{tabular}{@{}ccc@{}}
\includegraphics[width=0.3\textwidth,trim=0cm 0cm 0cm 0.5cm,clip=true]{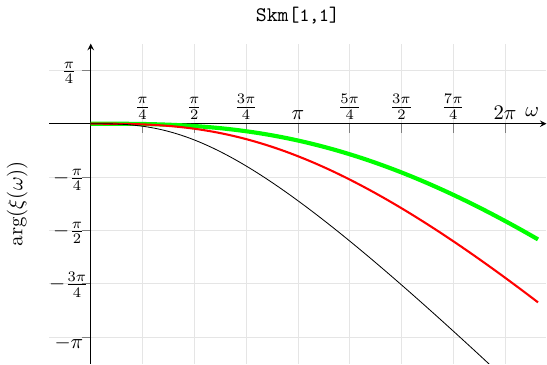}
& \includegraphics[width=0.3\textwidth,trim=0cm 0cm 0cm 0.5cm,clip=true]{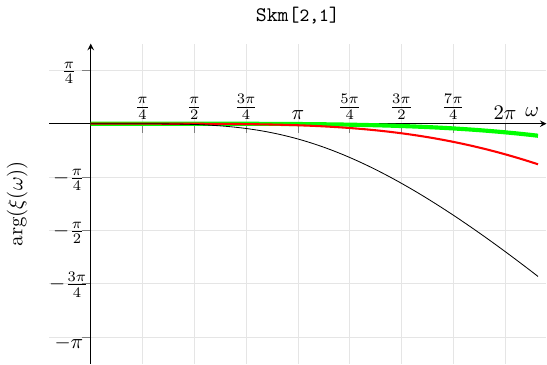}
& \includegraphics[width=0.3\textwidth,trim=0cm 0cm 0cm 0.5cm,clip=true]{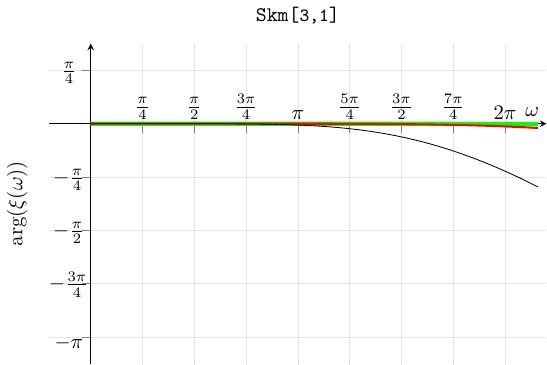}\\
(a) \SK{1}{1}. & (b) \SK{2}{1}. & (c) \SK{3}{1}.\\
\includegraphics[width=0.3\textwidth,trim=0cm 0cm 0cm 0.5cm,clip=true]{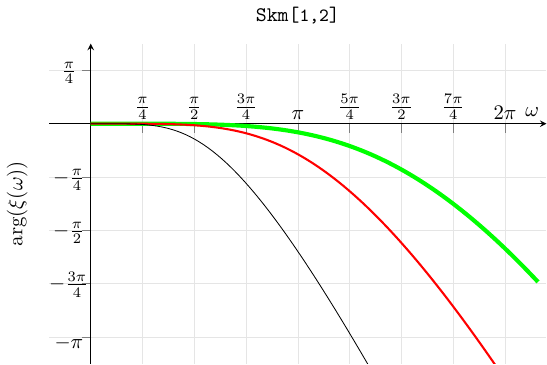}
& \includegraphics[width=0.3\textwidth,trim=0cm 0cm 0cm 0.5cm,clip=true]{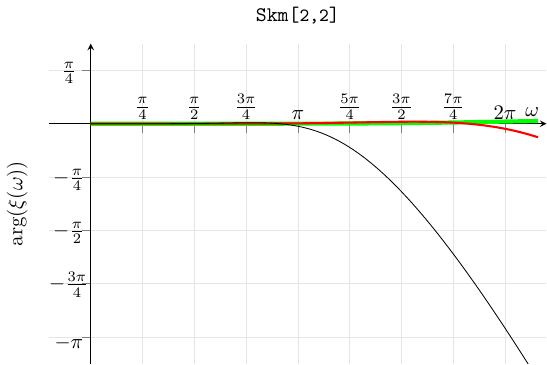}
& \includegraphics[width=0.3\textwidth,trim=0cm 0cm 0cm 0.5cm,clip=true]{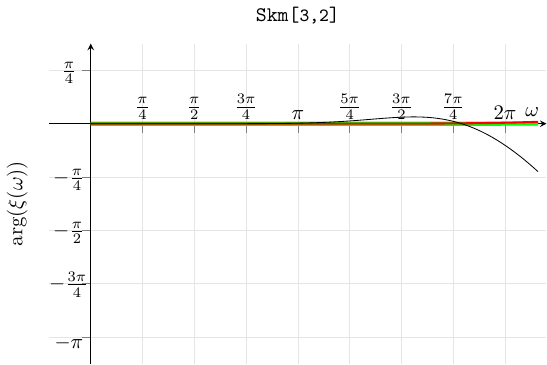}\\
(d) \SK{1}{2}. & (e) \SK{2}{2}. & (f) \SK{3}{2}.\\
\includegraphics[width=0.3\textwidth,trim=0cm 0cm 0cm 0.5cm,clip=true]{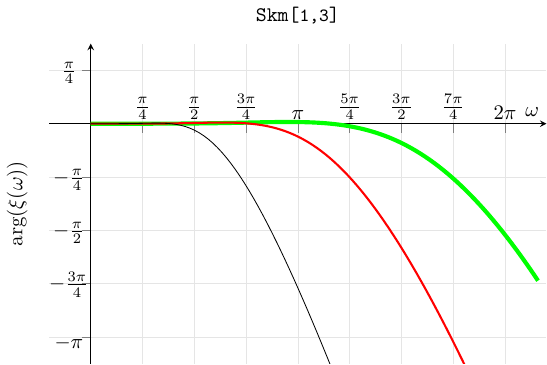}
& \includegraphics[width=0.3\textwidth,trim=0cm 0cm 0cm 0.5cm,clip=true]{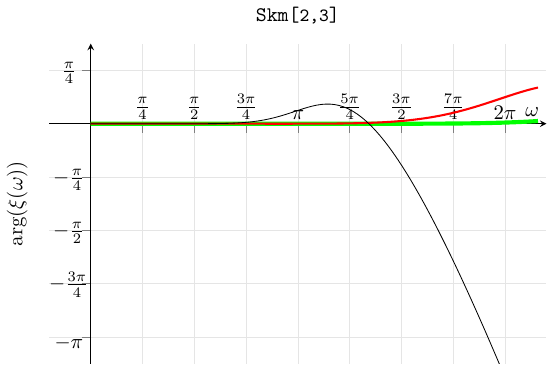}
& \includegraphics[width=0.3\textwidth,trim=0cm 0cm 0cm 0.5cm,clip=true]{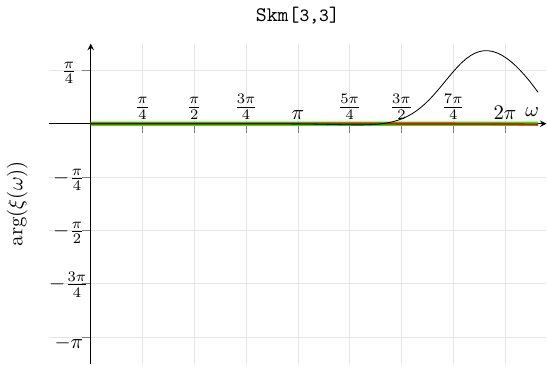}\\
(g) \SK{1}{3}. & (h) \SK{2}{3}. & (i) \SK{3}{3}.
\end{tabular}
\caption[]{Dispersion curves $\arg(\zeta_R(\omega))$ for $\Delta t$ (\blackline), $\arg(\zeta_R^2(\omega/2))$ for $\Delta t/2$ (\redline), and $\arg(\zeta_R^3(\omega/3))$ for $\Delta t/3$ (\greenline) as function of $\omega$ for schemes \SK{K}{R}, $K,R=1,2,3$.}
\label{fig:Dispersion_freq}
\end{figure}

The phase deviations $\zeta_R(\omega)$, $\zeta_R^2(\omega/2)$, and $\zeta_R^3(\omega/3)$ using the reference time step $\Delta t=T/N$ as a function of $K$ (columns) and $R$ (rows) are plotted in Figure~\ref{fig:Dispersion_freq}. First, as expected, a smaller time step strongly improves the spectral resolution by providing more time grid points. The trade-off between the argument reduction and the power increase of the phase deviation $\zeta_R^\ell(\omega/\ell)$ favors a larger $\ell$. Second, for fixed $\ell$, the phase error is drastically reduced with larger $K$. For example, with $\ell=1$, the phase error is almost contained in $\omega=\pi/4$ ($K=1$), $\omega=\pi/2$ ($K=2$), and $\omega=\pi$ ($K=3$). This last case indicates that we recovered subscale information, namely the presence of an extremum within the cell.

\subsection{Spectral resolution analysis with respect to $K$ and $R$} \label{sec:zeta_R_K_dependency}

Let $T=1$ and $\kappa\in\mathbb{N}$ be given. We fix the value of $N$ such that $J=N/R\in\mathbb{N}$, with $R=1,\ldots,4$. The aim is to compare the deviation of the numerical approximation in the final time $T$ for different choices of $K$ and $R$, while the time step $\Delta t=T/N$ is constant, independently of the two parameters. For $R=1$, the final deviation is ${\tt dev}(\omega;K,1)=\big (\zeta_R(\omega;K,1)\big )^N$ since $N$ blocks of size $R=1$ must be solved. For blocks of size $R$, the final deviation at the final time is given by ${\tt dev}(\omega;K,R)=\big (\zeta_R(\omega;K,R)\big )^{J}$.
Consequently, numerical simulations can be performed to assess the phase deviation relative to the analytic expression.

\begin{table}[ht]
\centering
\caption{Comparison between the analytical error calculated with $\arg(\zeta_R(\omega;K,R))$ and the numerical error obtained with the structural schemes $\SK{K}{R}$ for $N=36$, $T=1$, and $\kappa=1$ ($\omega=2\pi/36$).}
\label{Table::K1}
\begin{tabular}{@{}llrrrr@{}}
\toprule
K & err & R=1 & R=2 & R=3 & R=4\\
\midrule
\multirow{2}{*}{1}
&analytic &-1.59E-02 &-1.29E-04 & 7.07E-05 & 2.20E-06\\
&numeric &-1.59E-02 &-1.29E-04 & 7.07E-05 & 2.20E-06\\
\midrule
\multirow{2}{*}{2}
&analytic &-8.08E-06 & 1.87E-08 &-5.08E-11 & 1.61E-13\\
&numeric &-8.08E-06 & 1.87E-08 &-5.08E-11 & 1.61E-13\\
\midrule
\multirow{2}{*}{3}
&analytic &-1.76E-09 & 6.27E-15 & 2.61E-18 &-5.38E-23\\
&numeric &-1.76E-09 & 6.27E-15 & 2.61E-18 &-5.38E-23\\
\midrule
\multirow{2}{*}{4}
&analytic &-2.13E-13 &-8.88E-20 &-5.71E-26 &-5.46E-32\\
&numeric &-2.13E-13 &-8.88E-20 &-5.71E-26 &-5.46E-32\\
\bottomrule
\end{tabular}
\end{table}

\begin{table}[ht]
\centering
\caption{Comparison between the analytical error calculated with $\arg(\zeta_R(\omega;K,R))$ and the numerical error obtained with the structural schemes $\SK{K}{R}$ for $N=36$, $T=1$, and $\kappa=2$, ($\omega=2\pi/18$).}
\label{Table::K2}
\begin{tabular}{@{}llrrrr@{}}
\toprule
K & err & R=1 & R=2 & R=3 & R=4\\
\midrule
\multirow{2}{*}{1}
& analytic & -1.25E-01 &-4.03E-03 & 2.05E-03 & 2.64E-04\\
& numeric & -1.25E-01 &-4.03E-03 & 2.05E-03 & 2.64E-04\\
\midrule
\multirow{2}{*}{2}
& analytic & -2.57E-04 & 2.34E-06 &-2.46E-08 & 2.94E-10\\
& numeric & -2.57E-04 & 2.34E-06 &-2.46E-08 & 2.94E-10\\
\midrule
\multirow{2}{*}{3}
& analytic & -2.24E-07 & 1.27E-11 & 2.00E-14 &-6.72E-18\\
& numeric & -2.24E-07 & 1.27E-11 & 2.00E-14 &-6.72E-18\\
\midrule
\multirow{2}{*}{4}
& analytic & -1.09E-10 &-7.14E-16 &-7.06E-21 &-1.01E-25\\
& numeric & -1.09E-10 &-7.14E-16 &-7.06E-21 &-1.01E-25\\
\bottomrule
\end{tabular}
\end{table}

The deviation errors computed analytically and numerically for different values of $\kappa$ are reported in Tables~\ref{Table::K1} and~\ref{Table::K2}. We observe a perfect match between theory and simulation. For $\kappa=1$, the solution is a full revolution discretised with 36 points. The deviation decreases as $K$ and $R$ increase, in accordance with the scheme's convergence order. We note that doubling the frequency to $\kappa=2$, that is, 18 points per full revolution, strongly impacts accuracy, especially for the highest-order-accurate schemes. For example, one order of magnitude is lost for $K=1$ and $R=1$, while seven orders of magnitude are lost for $K=4$ and $R=4$.

To assess the spectral resolution in more challenging situations, we increase $\kappa$ to $4$, $8$, and $12$. The last case means that we have $9$ time grid points to compute two revolutions. The analytical errors of the phase deviation for these cases are reported in Table~\ref{Table::K345}. Clearly, when $K=1$, we do not resolve the sub-scaled variations with $\kappa=8,12$, and the solution is no longer admissible. The case $K=2$ provides a better approximation of the solution but requires a larger $R$. From $K=3$ and above, we managed to solve the sub-scaled information (at least the one associated with the highest frequency $\kappa=12$).

\begin{table}[ht]
\centering
\caption{Analytical error calculated as $\arg(\zeta_R(\omega;K,R))$ for the structural schemes $\SK{K}{R}$ for $N=36$, $T=1$ and three values of $\kappa$.}
\label{Table::K345}
\begin{tabular}{@{}llrrrr@{}}
\toprule
$\kappa$&K & R=1 & R=2 & R=3 & R=4\\
\midrule
\multirow{4}{*}{4}
& 1 &-9.52E-01 &-1.17E-01 & 4.23E-02 & 2.54E-02\\
& 2 &-8.05E-03 & 2.74E-04 &-9.92E-06 & 3.38E-07\\
& 3 &-2.83E-05 & 2.51E-08 & 1.21E-10 &-7.21E-13\\
& 4 &-5.50E-08 &-5.44E-12 &-7.19E-16 &-1.07E-19\\
\midrule
\multirow{4}{*}{8}
& 1 &-1.00E-01 &-2.56E-00 &-3.75E-01 & 7.52E-01\\
& 2 &-2.35E-01 & 2.41E-02 &-1.14E-03 &-4.72E-04\\
& 3 &-3.42E-03 & 4.39E-05 &-5.19E-08 &-3.53E-08\\
& 4 &-2.70E-05 &-3.22E-08 &-1.12E-11 & 2.78E-13\\
\midrule
\multirow{4}{*}{12}
& 1 & 1.66E-00 & 1.54E-00 &-2.61E-01 &-2.89E-00\\
& 2 & -1.52E-00 & 1.89E-01 & 6.68E-02 &-5.55E-02\\
& 3 & -5.30E-02 & 2.90E-03 &-1.71E-04 & 7.06E-06\\
& 4 & -9.65E-04 &-3.10E-06 & 5.90E-08 & 1.87E-09\\
\bottomrule
\end{tabular}
\end{table}

An interesting point to highlight is the trade-off between the total number of unknowns and accuracy. With $K$ derivatives, we handle $N(K+1)$ unknowns, and the number is invariant with respect to the parameter $R$. For example, with $\kappa=2$, comparable errors are obtained with the structural schemes $\SK{2}{4}$ and $\SK{4}{2}$, but with $5N$ unknowns for the first case and $3N$ unknowns for the latter. For $\kappa=2$, the structural schemes $\SK{1}{4}$ and $\SK{2}{1}$ have an accuracy comparable with $3N$ and $2N$ unknowns, respectively. Nevertheless, when increasing the frequency, there may be significant limitations on retrieving sub-scaled information if $K$ is too small, regardless of $R$. For instance, with $\kappa=12$, the structural schemes $\SK{2}{4}$ and $\SK{3}{1}$ have the same phase deviation error, but a larger $K$ (that is, increasing local data) is preferred to provide better insights on the sub-scaled information.

\subsection{Phase convergence}

To assess phase convergence, we solve the linear equation $\phi'(t)=-\complexnum{i}2\kappa\pi\phi(t)$ in $(0,1]$, and assess the deviation at the final time $T$, given as
\[
\zeta(T)=\frac{\phi_N}{\phi(T)}=\phi_N\exp(\complexnum{i}2\kappa\pi T),
\]
which represents the transfer function at the end of the simulation. $|\zeta(T)|$ is the amplitude variation while $\arg\zeta(T)$ is the phase shift of the numerical approximation at the final time. Note that if $\kappa\in\mathbb{Z}$ and $T\in\mathbb{Z}$, we simply have $|\zeta(T)|=|\phi_N|$ and $\arg\zeta(T)=\arg\phi_N$.

\begin{figure}[ht]
\centering
\begin{tabular}{cc}
\includegraphics[width=0.45\linewidth]{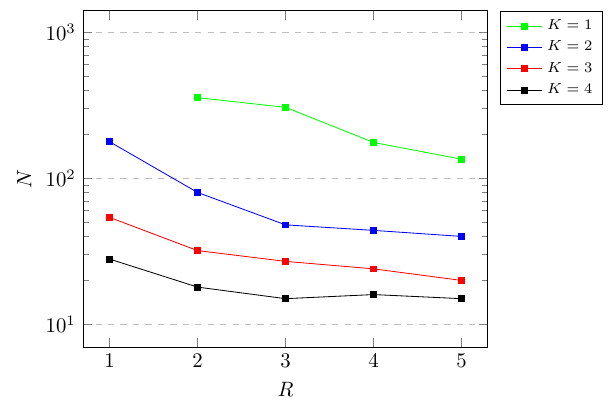}
& \includegraphics[width=0.45\linewidth]{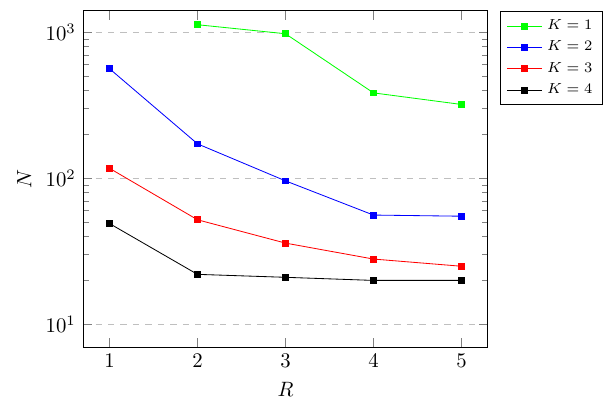}\\
(a) $\varepsilon$=1E-3. & (b) $\varepsilon$=1E-5.
\end{tabular}
\caption{Smallest values for $N$ to provide a phase error above the parameter $\varepsilon$.}
\label{fig:NvsR}
\end{figure}

The numerical test we propose aims to determine the smallest value of $N$ for which the scheme provides an error below a prescribed accuracy $\varepsilon$ as a function of the two parameters $K$ and $R$. To this end, we take $\kappa=10$, and by successive approximations, we determine the smallest $N$. The value of $N$ as a function of $R$ for the four curves ($K=1,\ldots,4$) with tolerances of $\varepsilon=10^{-3}$ and $\varepsilon=10^{-5}$ are plotted in Figure~\ref{fig:NvsR}. Except for the case $K=1$, a constant value $N$ is quickly reached depending on $R$. The parameter $\Delta t$ is mainly controlled by the spectral resolution, that is, the minimum number of points to resolve a full revolution.

\section{Conclusions}

We proposed and analysed a novel class of numerical schemes for ODEs based on the structural method. These schemes achieve extreme convergence orders and excellent spectral resolution while using compact stencils. We have established their stability and convergence. Several numerical benchmarks demonstrate that the method successfully balances running time and accuracy, resulting in superior numerical efficiency. The structural schemes analysed use the maximum number of physical equations. This leads to a novel multi-derivative Runge-Kutta formulation. Standard techniques often struggle to design and implement such methods, mainly because of the numerous non-linear order conditions. In contrast, our structural schemes determine coefficients directly from the kernel basis. This approach drastically reduces the design complexity and enables arbitrary convergence orders. An additional advantage is the accuracy of the intermediate time steps. Traditional Runge-Kutta methods treat sub-steps as auxiliary calculations. However, in the structural schemes, the values at $t_{n+r}$, $r\bint{1,R}$, achieve the same convergence order and can be used as genuine approximations. Finally, the structural method naturally connects with compact schemes commonly used in space discretisation. This opens the potential to combine structural equations in both time and space.

\section*{Acknowledgements}

The authors thank Professor Robert McLachlan for his important note on the connection between the structural method and the multi-derivative Runge-Kutta schemes.\\[0.4em]
S. Clain, M.T. Malheiro, and G.J. Machado acknowledge the financial support of the Portuguese Foundation for Science and Technology (FCT) through a national funding for projects IC\&DT with the reference 2023.16854.ICDT.\\[0.4em]
S. Clain acknowledge financial support by the Centre for Mathematics of the University of Coimbra (CMUC, https://doi.org/10.54499/UID/00324/2025) under the Portuguese Foundation for Science and Technology (FCT), Grants UID/00324/2025 and UID/PRR/00324/2025.\\[0.4em]
The research of M.T. Malheiro and G.J. Machado was financially supported by the Fundação para a Ciência e a Tecnologia (Portuguese Foundation for Science and Technology) under the scope of the project UID/00013/2025 {\tt https://doi.org/10.54499/UID/00013/2025}, Center of Mathematics of the University of Minho.\\[0.4em]
The research of R. Costa was partially financed by Portuguese Funds through FCT (Fundação para a Ciência e a Tecnologia) under the scope of the projects UIDB/05256/2025 and UIDP/05256/2025 {\tt https://doi.org/10.54499/UID/00013/2025}, Institute for Polymers and Composites of the University of Minho.\\[0.4em]
R. Costa also acknowledges the financial support of Funda\c c\~ao para a Ciência e a Tecnologia (Portuguese Foundation for Science and Technology) through the Contract-Program signed between the Funda\c c\~ao para a Ciência e a Tecnologia and the University of Minho, under the scope of the project 2023.09481.CEECIND/CP2841/CT0008, Individual Scientific Employment Stimulus Call -- 6th Edition (CEEC IND6ed).

\section*{Statement on the use of Artificial Intelligence Tools}

The document was reviewed and refined using Large Language Models to improve grammar, correct typographical errors, and enhance clarity. Any text or content generated by Artificial Intelligence is explicitly identified within the document and is utilized only to a limited extent. The authors retain full responsibility for the overall content and ideas presented.

\appendix
\section{Analytical structural equations}
\label{analytical_expression}

The structural schemes $\SK{K}{R}$ are based on the structural equations \SEE{K}{R}{S}(s), $s\bint{1,S}$, characterised by the basis $(\ba^s)_{s=1}^S$ of the kernel $\mathcal K^S\subset\mathbb{R}^M$, with $M=(K+1)(R+1)$ the size of the vectors. By construction, the structural equations are exact up to a polynomial degree $d=M-S$ and read
\[
\sum_{k=0}^K \sum_{r=0}^R a^s_{k,r}\,\Delta t^k\,\phi^{(k)}_{n+r}=0,\quad s\bint{1,K}.
\]
In practice, we use $S=R$ to design the structural method, and the structural schemes $\SK{K}{R}$ provide $R$ linearly independent equations.

Tables~\ref{tab::SE_R=1},~\ref{tab::SE_R=2}, and~\ref{tab::SE_R=3} present the analytical expressions of the coefficients $a^s_{k,r}$ for $R=1,2,3$, each with $K=1,2,3$. Additionally, the kernel method can be used to compute these coefficients numerically for any choice of $K$, $R$, and $S$.

\begin{table}[H]
\tablepolice
\centering
\caption{Structural equations coefficients for $R=1$.}
\label{tab::SE_R=1}
\begin{tabular}{@{}rrrrrrrrrrrr@{}}
\toprule
K & s && $a^{s}_{0,0}$ & $a^{s}_{1,0}$ & $a^{s}_{2,0}$ & $a^{s}_{3,0}$ && $a^{s}_{0,1}$ & $a^{s}_{1,1}$ & $a^{s}_{2,1}$ & $a^{s}_{3,1}$\\
\midrule
1 & 1 && -2 & -1 & & && 2 & -1 & &\\
\midrule
2 & 1 && -12 & -6 & -1 & && 12 & -6 & 1 &\\
\midrule
3 & 1 && 120 & 60 & 12 & 1 && -120 & 60 & -12 & 1\\
\bottomrule
\end{tabular}
\end{table}

\begin{table}[H]
\tablepolice
\centering
\caption{Structural equations coefficients for $R=2$.}
\label{tab::SE_R=2}
\begin{tabular}{@{}rrrrrrrrrrrrrrrrrr@{}}
\toprule
K && s & & $a^{s}_{0,0}$ & $a^{s}_{1,0}$ & $a^{s}_{2,0}$ & $a^{s}_{3,0}$ & \phantom{a} & $a^{s}_{0,1}$ & $a^{s}_{1,1}$ & $a^{s}_{2,1}$ & $a^{s}_{3,1}$ & \phantom{a} & $a^{s}_{0,2}$ & $a^{s}_{1,2}$ & $a^{s}_{2,2}$ & $a^{s}_{3,2}$\\
\midrule
\multirow{2}{*}{1}
&& 1 && -3 & -1 & & && 0 & -4 & & && 3 & -1 & &\\
&& 2 && 2 & 1 & & && -4 & 0 & & && 2 & -1 & &\\
\midrule
\multirow{2}{*}{2}
&& 1 && 24 & 9 & 1 & && -48 & 0 & -8 & && 24 & -9 & 1 &\\
&& 2 && -15 & -7 & -1 & && 0 & -16 & 0 & && 15 & -7 & 1 &\\
\midrule
\multirow{2}{*}{3}
&& 1 && 315 & 123 & 18 & 1 && 0 & 384 & 0 & 16 && -315 & 123 & -18 & 1\\
&& 2 && -192 & -87 & -15 & -1 && 384 & 0 & 48 & 0 && -192 & 87 & -15 & 1\\
\bottomrule
\end{tabular}
\end{table}

\begin{table}[H]
\tablepolice
\centering
\caption{Structural equations coefficients for $R=3$.}
\label{tab::SE_R=3}
\begin{tabular}{@{}rrrrrrrrrrrrrrrrrrrr@{}}
\toprule
K &&s& & $a^{s}_{0,0}$ & $a^{s}_{1,0}$ & $a^{s}_{2,0}$ & $a^{s}_{3,0}$ & $a^{s}_{0,1}$ & $a^{s}_{1,1}$ & $a^{s}_{2,1}$ & $a^{s}_{3,1}$ & $a^{s}_{0,2}$ & $a^{s}_{1,2}$ & $a^{s}_{2,2}$ & $a^{s}_{3,2}$ & $a^{s}_{0,3}$ & $a^{s}_{1,3}$ & $a^{s}_{2,3}$ & $a^{s}_{3,3}$\\
\midrule
\multirow{3}{*}{1}
&& 1& & 11 & 3 & & & 27 & 27 & & & -27 & 27 & & & -11 & 3 & &\\
&&2 && 19 & 6 & & & 0 & 27 & & & -27 & 0 & & & 8 & -3 & &\\
&& 3 && -13 & -6 & & & 27 & 0 & & & -27 & 0 & & & 13 & -6 & &\\
\midrule
\multirow{3}{*}{2}
&& 1 && -103 & -33 & -3 & & 729 & 243 & 81 & & -729 & 243 & -81 & & 103 & -33 & 3 &\\
&& 2 && 173 & 60 & 6 & & -486 & -81 & -81 & & 243 & -162 & 0 & & 70 & -27 & 3 &\\
&& 3 && -113 & -48 & -6 & & -81 & -162 & 0 & & 81 & -162 & 0 & & 113 & -48 & 6 &\\
\midrule
\multirow{3}{*}{3}
&& 1 && 4711 & 1599 & 198 & 9 & 45927 & 28431 & 4374 & 729 & -45927 & 28431 & -4374 & 729 & -4711 & 1599 & -198 & 9\\
&& 2 && -7823 & -2802 & -369 & -18 & -17496 & -19683 & -2187 & -729 & 28431 & -8748 & 2187 & 0 & -3112 & 1203 & -171 & 9\\
&& 3 && 5021 & 2064 & 315 & 18 & -19683 & -4374 & -2187 & 0 & 19683 & -4374 & 2187 & 0 & -5021 & 2064 & -315 & 18\\
\bottomrule
\end{tabular}
\end{table}

\section{Post-processing structural equations}
\label{appendix::postproc}

Assuming that we have performed the computation with a structural scheme $\SK{K}{R}$, post-processing structural equations enable computing higher derivatives of order $p$ with $p>K$ at the final time step $t_N$ using the approximate solution and derivatives from the backward time steps. We then propose an approximation for $\phi_N^{(p)}$ of convergence order $O$ using a relation in the form
\[
\phi_N^{(p)}=\sum_{i=0}^{I_{p}}\sum_{k=0}^K a^{p}_{ik}\,\Delta t^k\,\phi_{N-i}^{(k)},
\]
where the convergence order of the method and the coefficients are given in the following tables.

\subsection{Case $K=1$}

We assume that the approximations $\phi^{(k)}_{-i}$ are available at any point $T-i\Delta t$, $i\bint{0,I_d}$, and $k=0,1$. Table~\ref{tab::post_processing_K=1} provides the formula and the corresponding convergence order for the approximation $\phi_N^{(p)}$ for $p=2,3,4$.

\begin{table}[H]
\tablepolice
\centering
\caption{Coefficients of the post-processing structural equations with $K=1$ for $p=2,3,4$.}
\label{tab::post_processing_K=1}
\begin{tabular}{@{}rrrrl@{}}
\toprule
p & Ip & O && $\phi_N^{(p)}$\\
\midrule
2 & 1 & 2 &&
\specialcell{l}{$\begin{aligned}[t]
\phi_N^{(2)} &= 6\phi_{N-1}^{(0)}-6\phi_{N}^{(0)}\\
&\quad+\Delta t (2\phi_{N-1}^{(1)}+4\phi_{N}^{(1)})
\end{aligned}$}\\[3em]
2 & 2 & 4 &&
\specialcell{l}{$\begin{aligned}[t]
\phi_N^{(2)} &=\tfrac{7}{2}\phi_{N-2}^{(0)}+8\phi_{N-1}^{(0)}-\tfrac{23}{2}\phi_{N}^{(0)}\\
&\quad+\Delta t (\phi_{N-2}^{(1)}+8\phi_{N-1}^{(1)}+6\phi_{N}^{(1)})
\end{aligned}$}\\[3em]
2 & 3 & 6 &&
\specialcell{l}{$\begin{aligned}[t]
\phi_N^{(2)} &=\tfrac{8}{3}\phi_{N-3}^{(0)}+\tfrac{27}{2}\phi_{N-2}^{(0)}-\tfrac{97}{6}\phi_{N}^{(0)}\\
&\quad+\Delta t (\tfrac{2}{3}\phi_{N-3}^{(1)}+9\phi_{N-2}^{(1)}+18\phi_{N-1}^{(1)}+\tfrac{22}{3}\phi_{N}^{(1)})
\end{aligned}$}\\
\midrule
3 & 2 & 3 &&
\specialcell{l}{$\begin{aligned}[t]
\phi_N^{(3)} &=\tfrac{51}{2}\phi_{N-2}^{(0)}+24\phi_{N-1}^{(0)}-\tfrac{99}{2}\phi_{N}^{(0)}\\
&\quad+\Delta t (\tfrac{15}{2}\phi_{N-2}^{(1)}+48\phi_{N-1}^{(1)}+\tfrac{39}{2}\phi_{N}^{(1)})
\end{aligned}$}\\[3em]
3 & 3 & 5 &&
\specialcell{l}{$\begin{aligned}[t]
\phi_N^{(3)} &=\tfrac{238}{9}\phi_{N-3}^{(0)}+\tfrac{243}{2}\phi_{N-2}^{(0)}-54\phi_{N-1}^{(0)}-\tfrac{1691}{18}\phi_{N}^{(0)}\\
&\quad+\Delta t (\tfrac{20}{3}\phi_{N-3}^{(1)}+\tfrac{171}{2}\phi_{N-2}^{(1)}+144\phi_{N-1}^{(1)}+\tfrac{193}{6}\phi_{N}^{(1)})
\end{aligned}$}\\[3em]
3 & 4 & 7 &&
\specialcell{l}{$\begin{aligned}[t]
\phi_N^{(3)} &=\tfrac{1241}{48}\phi_{N-4}^{(0)}+\tfrac{2176}{9}\phi_{N-3}^{(0)}+171\phi_{N-2}^{(0)}-\tfrac{896}{3}\phi_{N-1}^{(0)}-\tfrac{20155}{144}\phi_{N}^{(0)}\\
&\quad+\Delta t (\tfrac{47}{8}\phi_{N-4}^{(1)}+\tfrac{368}{3}\phi_{N-3}^{(1)}+396\phi_{N-2}^{(1)}+304\phi_{N-1}^{(1)}+\tfrac{1045}{24}\phi_{N}^{(1)})
\end{aligned}$}\\
\midrule
4 & 2 & 2 &&
\specialcell{l}{$\begin{aligned}[t]
\phi_N^{(4)} &= 78\phi_{N-2}^{(0)}+24\phi_{N-1}^{(0)}-102\phi_{N}^{(0)}\\
&\quad+\Delta t (24\phi_{N-2}^{(1)}+120\phi_{N-1}^{(1)}+36\phi_{N}^{(1)})
\end{aligned}$}\\[3em]
4 & 3 & 4 &&
\specialcell{l}{$\begin{aligned}[t]
\phi_N^{(4)} &=\tfrac{400}{3}\phi_{N-3}^{(0)}+540\phi_{N-2}^{(0)}-360\phi_{N-1}^{(0)}-\tfrac{940}{3}\phi_{N}^{(0)}\\
&\quad+\Delta t (34\phi_{N-3}^{(1)}+408\phi_{N-2}^{(1)}+582\phi_{N-1}^{(1)}+96\phi_{N}^{(1)})
\end{aligned}$}\\[3em]
4 & 4 & 6 &&
\specialcell{l}{$\begin{aligned}[t]
\phi_N^{(4)} &=\tfrac{11879}{72}\phi_{N-4}^{(0)}+\tfrac{13328}{9}\phi_{N-3}^{(0)}+\tfrac{1659}{2}\phi_{N-2}^{(0)}-\tfrac{16912}{9}\phi_{N-1}^{(0)}-\tfrac{42931}{72}\phi_{N}^{(0)}\\
&\quad+\Delta t (\tfrac{113}{3}\phi_{N-4}^{(1)}+\tfrac{2296}{3}\phi_{N-3}^{(1)}+2343\phi_{N-2}^{(1)}+\tfrac{4712}{3}\phi_{N-1}^{(1)}+\tfrac{995}{6}\phi_{N}^{(1)})
\end{aligned}$}\\
\bottomrule
\end{tabular}
\end{table}

\subsection{Case $K=2$}

We proceed with $K=2$ and provide in Table~\ref{tab::post_processing_K=2} the post-processing structural equations that provide an approximation of $\phi_N^{(p)}$ for $p=3,4$.

\begin{table}[H]
\tablepolice
\centering
\caption{Coefficients of the post-processing structural equations with $K=2$ for $p=3,4$.}
\label{tab::post_processing_K=2}
\begin{tabular}{@{}rrrrl@{}}
\toprule
\texttt{p} & \texttt{Ip} & \texttt{O} && $\phi_N^{(p)}$\\
\midrule
3 & 2 & 6 &&
\specialcell{l}{$\begin{aligned}[t]
\phi_N^{(3)} &= 87\phi_{N-2}^{(0)}-384\phi_{N-1}^{(0)}+297\phi_{N}^{(0)}\\
&\quad+2\Delta t (15\phi_{N-2}^{(1)}-48\phi_{N-1}^{(1)}-72\phi_{N}^{(1)})\\
&\quad+\Delta t^2 (3\phi_{N-2}^{(2)}-48\phi_{N-1}^{(2)}+27\phi_{N}^{(2)})
\end{aligned}$}\\[4.5em]
\texttt{3} & \texttt{3} & \texttt{9} &&
\specialcell{l}{$\begin{aligned}[t]
\phi_N^{(3)} &=-\tfrac{344}{9}\phi_{N-3}^{(0)}+\tfrac{891}{2}\phi_{N-2}^{(0)}-648\phi_{N-1}^{(0)}+\tfrac{4333}{18}\phi_{N}^{(0)}\\
&\quad+\Delta t (-\tfrac{35}{3}\phi_{N-3}^{(1)}+162\phi_{N-2}^{(1)}+81\phi_{N-1}^{(1)}-103\phi_{N}^{(1)})\\
&\quad+\Delta t^2 (-\phi_{N-3}^{(2)}+\tfrac{81}{2}\phi_{N-2}^{(2)}-81\phi_{N-1}^{(2)}+\tfrac{33}{2}\phi_{N}^{(2)})
\end{aligned}$}\\[4.5em]
\texttt{3} & \texttt{4} & \texttt{12} &&
\specialcell{l}{$\begin{aligned}[t]
\phi_N^{(3)} &=\tfrac{559}{16}\phi_{N-4}^{(0)}-\tfrac{8768}{9}\phi_{N-3}^{(0)}+2592\phi_{N-2}^{(0)}-1984\phi_{N-1}^{(0)}+\tfrac{47705}{144}\phi_{N}^{(0)}\\
&\quad+\Delta t (\tfrac{39}{4}\phi_{N-4}^{(1)}-\tfrac{1088}{3}\phi_{N-3}^{(1)}+324\phi_{N-2}^{(1)}+576\phi_{N-1}^{(1)}-130\phi_{N}^{(1)})\\
&\quad+\Delta t^2 (\tfrac{3}{4}\phi_{N-4}^{(2)}-64\phi_{N-3}^{(2)}+324\phi_{N-2}^{(2)}-192\phi_{N-1}^{(2)}+\tfrac{75}{4}\phi_{N}^{(2)})
\end{aligned}$}\\
\midrule
\texttt{4} & \texttt{2} & \texttt{5} &&
\specialcell{l}{$\begin{aligned}[t]
\phi_N^{(4)} &=\tfrac{1359}{2}\phi_{N-2}^{(0)}-2304\phi_{N-1}^{(0)}+\tfrac{3249}{2}\phi_{N}^{(0)}\\
&\quad+\Delta t (237\phi_{N-2}^{(1)}-480\phi_{N-1}^{(1)}-702\phi_{N}^{(1)})\\
&\quad+\Delta t^2 (24\phi_{N-2}^{(2)}-336\phi_{N-1}^{(2)}+99\phi_{N}^{(2)})
\end{aligned}$}\\[4.5em]
\texttt{4} & \texttt{3} & \texttt{8} &&
\specialcell{l}{$\begin{aligned}[t]
\phi_N^{(4)} &=-\tfrac{7060}{9}\phi_{N-3}^{(0)}+\tfrac{17415}{2}\phi_{N-2}^{(0)}-11340\phi_{N-1}^{(0)}+\tfrac{61505}{18}\phi_{N}^{(0)}\\
&\quad+\Delta t (-\tfrac{2162}{9}\phi_{N-3}^{(1)}+3159\phi_{N-2}^{(1)}+2106\phi_{N-1}^{(1)}-\tfrac{11728}{9}\phi_{N}^{(1)})\\
&\quad+\Delta t^2 (-\tfrac{62}{3}\phi_{N-3}^{(2)}+810\phi_{N-2}^{(2)}-1458\phi_{N-1}^{(2)}+157\phi_{N}^{(2)})
\end{aligned}$}\\[4.5em]
\texttt{4} & \texttt{4} & \texttt{11} &&
\specialcell{l}{$\begin{aligned}[t]
\phi_N^{(4)} &=\tfrac{26751}{32}\phi_{N-4}^{(0)}-\tfrac{205888}{9}\phi_{N-3}^{(0)}+58968\phi_{N-2}^{(0)}-42432\phi_{N-1}^{(0)}+\tfrac{1585289}{288}\phi_{N}^{(0)}\\
&\quad+\Delta t (\tfrac{1869}{8}\phi_{N-4}^{(1)}-76736\phi_{N-3}^{(1)}+6804\phi_{N-2}^{(1)}+13632\phi_{N-1}^{(1)}-\tfrac{69295}{36}\phi_{N}^{(1)})\\
&\quad+\Delta t^2 (18\phi_{N-4}^{(2)}-\tfrac{4544}{3}\phi_{N-3}^{(2)}+7452\phi_{N-2}^{(2)}-4032\phi_{N-1}^{(2)}+\tfrac{835}{4}\phi_{N}^{(2)})
\end{aligned}$}\\
\bottomrule
\end{tabular}
\end{table}

\subsection{Case $K=3$}

Finally, for $K=3$ we provide in Table~\ref{tab::post_processing_K=2} the post-processing structural equations that provide an approximation of $\phi_N^{(4)}$.

\begin{table}[ht]
\tablepolice
\centering
\caption{Coefficients of the post-processing structural equations with $K=3$ for $p=4$.}
\label{tab::post_processing_K=3}
\begin{tabular}{@{}rrrrl@{}}
\toprule
\texttt{p} & \texttt{Ip} & \texttt{O} && $\phi_N^{(p)}$\\
\midrule
4 & 1 & 4 &&
\specialcell{l}{$\begin{aligned}[t]
\phi_N^{(4)} &= 840\phi_{N-1}^{(0)}-840\phi_{N}^{(0)}\\
&\quad+\Delta t (360\phi_{N-1}^{(1)}+480\phi_{N}^{(1)})\\
&\quad+\Delta t^2 (60\phi_{N-1}^{(2)}-120\phi_{N}^{(2)})\\
&\quad+\Delta t^3 (4\phi_{N-1}^{(3)}+16\phi_{N}^{(3)})
\end{aligned}$}\\[6.23em]
4 & 2 & 8 &&
\specialcell{l}{$\begin{aligned}[t]
\phi_N^{(4)} &= 1545\phi_{N-2}^{(0)}+1920\phi_{N-1}^{(0)}-\tfrac{5385}{2}\phi_{N}^{(0)}\\
&\quad+\Delta t (285\phi_{N-2}^{(1)}+1920\phi_{N-1}^{(1)}+1260\phi_{N}^{(1)})\\
&\quad+\Delta t^2 (39\phi_{N-2}^{(2)}+192\phi_{N-1}^{(2)}-246\phi_{N}^{(2)})\\
&\quad+\Delta t^3 (2\phi_{N-2}^{(3)}+64\phi_{N-1}^{(3)}+24\phi_{N}^{(3)})
\end{aligned}$}\\[6.23em]
4 & 3 & 12 &&
\specialcell{l}{$\begin{aligned}[t]
\phi_N^{(4)} &=\tfrac{21160}{27}\phi_{N-3}^{(0)}+\tfrac{27945}{2}\phi_{N-2}^{(0)}-9720\phi_{N-1}^{(0)}-\tfrac{271955}{54}\phi_{N}^{(0)}\\
&\quad+\Delta t (\tfrac{772}{3}\phi_{N-3}^{(1)}+7533\phi_{N-2}^{(1)}+10692\phi_{N-1}^{(1)}+\tfrac{18844}{9}\phi_{N}^{(1)})\\
&\quad+\Delta t^2 (\tfrac{92}{3}\phi_{N-3}^{(2)}+1215\phi_{N-2}^{(2)}-972\phi_{N-1}^{(2)}-\tfrac{1066}{3}\phi_{N}^{(2)})\\
&\quad+\Delta t^3 (\tfrac{4}{3}\phi_{N-3}^{(3)}+162\phi_{N-2}^{(3)}+324\phi_{N-1}^{(3)}+\tfrac{88}{3}\phi_{N}^{(3)})
\end{aligned}$}\\
\bottomrule
\end{tabular}
\end{table}

\section{Scheme stability data}
\label{anexe::stability}

In Section~\ref{subsec_Stab}, we established the stability condition of the structural schemes \SK{K}{R} with Theorem~\ref{TheorStab} if the transfer function reads $\chi_R(z)=p(z)/p(-z)$ and $p(z)$ is a Hurwitz polynomial in $\mathbb{C}$.
In this appendix, we first provide the expressions for $\chi_R(z)$, in Tables~\ref{table_stab_K1},~\ref{Table_stab_K2}, and~\ref{table_stab_K3} for $K,R=1,2,3$ obtained with the Symbolic Toolbox of MATLAB. We then check that the polynomials satisfy the Hurwitz property.

\subsection{Case $K=1$}

\begin{table}[ht]
\centering
\caption{Transfer functions of the structural schemes \SK{1}{R}.}
\label{table_stab_K1}
\begin{tabular}{@{}lll@{}}
\toprule
scheme & $R$ & $\chi_R(z)$\\
\midrule
\SK{1}{1} & $1$ & $\frac{2+z}{2-z}$\\[0.25cm]
\SK{1}{2} & $2$ & $\frac{z^2+3z+3}{z^2-3z+3}$\\[0.25cm]
\SK{1}{3} & $3$ & $\frac{3z^3+11z^2+18z+12}{-3z^3+11z^2-18z+12}$\\
\bottomrule
\end{tabular}
\end{table}


\begin{itemize}
\item \noindent \SK{1}{1}: unconditional second-order accurate scheme with $p(z)=z+2$ and its only zero ($z=-2$) is in the left half-plane.
\item \noindent \SK{1}{2}: unconditional fourth-order accurate scheme with $p(z)=z^2+3z+3$ and
\[
H_{12}=\begin{bmatrix}
3 & 0\\
1& 3\\
\end{bmatrix}.
\]
Its determinant is positive.
\item \noindent \SK{1}{3}: unconditional fourth-order accurate scheme with $p(z)=3z^3+11z^2+18z+12$ and
\[
H_{13}=\begin{bmatrix}
11 & 12 &0\\
3& 18 &0\\
0 &11 & 12
\end{bmatrix}.
\]
All its principal minors $m_i$ for $i\bint{1,3}$ of $\det H_{13}$ are strictly positive.
\end{itemize}

\subsection{Case $K=2$}

\begin{table}[H]
\centering
\caption{Transfer functions of the structural schemes \SK{2}{R}.}
\begin{tabular}{@{}lll@{}}
\toprule
scheme & $R$ & $\chi_R(z)$\\
\midrule
\SK{2}{1} & $1$ & $\frac{z^2+6z+12}{z^2-6z+12}$\\[0.25cm]
\SK{2}{2} & $2$ & $\frac{z^4+9z^3+39z^2+90z+90}{z^4-9z^3+39z^2-90z+90}$\\[0.25cm]
\SK{2}{3} & $3$ & $\frac{3z^6+33z^5+193z^4+720z^3+1740z^2+2520z+1680}{3z^6-33z^5+193z^4-720z^3+1740z^2-2520z+1680}$\\
\bottomrule
\end{tabular}\label{Table_stab_K2}
\end{table}

\begin{itemize}
\item \noindent \SK{2}{1}: unconditional fourth-order accurate scheme with $p(z)=z^2+6z+12$ and \[H_{21}=\begin{bmatrix}
6 & 0\\
1& 12\\
\end{bmatrix}.\] Its determinant is positive.

\item \noindent \SK{2}{2}: unconditional sixth-order accurate scheme with $p(z)=z^4+9z^3+39z^2+90z+90$ and \[H_{22}=\begin{bmatrix}
9 & 90 &0&0\\
1& 39 &90&0\\
0 &9 & 90 &0\\
0 &1& 39 &90
\end{bmatrix}.\] All principal minors $m_i$ for $i\bint{1,4}$ of $\det H_{22}$ are strictly positive.

\item \noindent \SK{2}{3}: unconditional eight-order accurate scheme with $p(z)=3z^6+33z^5+193z^4+720z^3+1740z^2+2520z+1680$ and
\[
H_{23}=\begin{bmatrix}
33 & 720 & 2520 & 0 & 0 & 0\\
3 & 193 & 1740 & 1680 & 0 & 0\\
0 & 33 & 720 & 2520 & 0 & 0\\
0 & 3 & 193 & 1740 & 1680 & 0\\
0 & 0 & 33 & 720 & 2520 & 0\\
0 & 0 & 3 & 193 & 1740 & 1680\\
\end{bmatrix}.\]
All principal minors $m_i$ for $i\bint{1,6}$ of $\det H_{23}$ are strictly positive.
\end{itemize}

\subsection{Case $K=3$}

\begin{table}[H]
\centering
\caption{Transfer functions of the structural schemes \SK{3}{R}.}
\begin{tabular}{@{}lll@{}}
\toprule
scheme & $R$ & $\chi_R(z)$\\
\midrule
\SK{3}{1} & $1$ & $\frac{z^3+12z^2+60z+120}{-z^3+12z^2-60z+120}$\\[0.25cm]
\SK{3}{2} & $2$ & $\frac{z^6+18z^5+165z^4+945z^3+3465z^2+7560z+7560}{z^6-18z^5+165z^4-945z^3+3465z^2-7560z+7560}$\\[0.25cm]
\SK{3}{3} & $3$ & $\frac{9z^9+198z^8+2355z^7+19075z^6+112770z^5+494760z^4+1587600z^3+3553200z^2+4989600z+3326400}{-9z^9+198z^8-2355z^7+19075z^6-112770z^5+494760z^4-1587600z^3+3553200z^2-4989600z+3326400}$\\
\bottomrule
\end{tabular}\label{table_stab_K3}
\end{table}

\begin{itemize}
\item \noindent \SK{3}{1}: unconditional sixth-order accurate scheme with $p(z)=z^3+12z^2+60z+120$ and \[H_{31}=\begin{bmatrix}
12 & 120 &0\\
1& 60 &0\\
0 &12 & 120
\end{bmatrix}.\] All principal minors $m_i$ for $i\bint{1,3}$ of $\det H_{31}$ are strictly positive.
\item \noindent \SK{3}{2}: unconditional tenth-order accurate scheme with $p(z)=z^6+18z^5+165z^4+945z^3+3465z^2+7560z+7560$ and \[H_{32}=\begin{bmatrix}
18 & 945 & 7560 & 0 & 0 & 0\\
1 & 165 & 3465 & 7560 & 0 & 0\\
0 & 18 & 945 & 7560 & 0 & 0\\
0 & 1 & 165 & 3465 & 7560 & 0\\
0 & 0 & 18 & 945 & 7560 & 0\\
0 & 0 & 1 & 165 & 3465 & 7560\\
\end{bmatrix}.\]
All principal minors $m_i$, for $i\bint{1,6}$ of $\det H_{32}$ are strictly positive.
\item \noindent \SK{3}{3}: unconditional twelfth-order accurate scheme with $p(z)=9z^9+198z^8+2355z^7+19075z^6+112770z^5+494760z^4+1587600z^3+3553200z^2+4989600z+3326400$ and
\[
H_{33}=\begin{bmatrix}
198 & 19075 & 494760 & 3553200 & 3326400 & 0 &0 &0 &0\\
9 & 2355 & 112770 & 15876000 & 4989600 & 0 &0 &0 &0\\
0 & 198 & 19075 & 494760 & 3553200 & 3326400 & 0 &0 &0\\
0 & 9 & 2355 & 112770 & 15876000 & 4989600 & 0 &0 &0\\
0 & 0 & 198 & 19075 & 494760 & 3553200 & 3326400 & 0 &0\\
0 & 0 & 9 & 2355 & 112770 & 15876000 & 4989600 & 0 &0\\
0& 0& 0& 198 & 19075 & 494760 & 3553200 & 3326400 & 0\\
0& 0& 0& 9 & 2355 & 112770 & 15876000 & 4989600 & 0\\
0& 0& 0& 0& 198 & 19075 & 494760 & 3553200 & 3326400\\
\end{bmatrix}.
\]
All principal minors $m_i$ for $i\bint{1,9}$ of $\det H_{33}$ are strictly positive.
\end{itemize}
Hence, all schemes are A-stable.

\section{Butcher Tableau for the collocation method} \label{app::colocation}
\subsection{The 4th-Order Method (2-Stage Gauss-Legendre)}
The 2-stage 4th-order Butcher tableau.
\begin{equation}
\begin{array}{c|cc}
\frac{1}{2} - \frac{\sqrt{3}}{6} & \frac{1}{4} & \frac{1}{4} - \frac{\sqrt{3}}{6} \\
\frac{1}{2} + \frac{\sqrt{3}}{6} & \frac{1}{4} + \frac{\sqrt{3}}{6} & \frac{1}{4} \\
\hline
& \frac{1}{2} & \frac{1}{2}
\end{array}
\end{equation}

\subsection{The 6th-Order Method (3-Stage Gauss-Legendre)}
The 3-stage 6th-order Butcher tableau
\begin{equation}
\begin{array}{c|ccc}
\frac{1}{2} - \frac{\sqrt{15}}{10} & \frac{5}{36} & \frac{2}{9} - \frac{\sqrt{15}}{15} & \frac{5}{36} - \frac{\sqrt{15}}{30} \\
\frac{1}{2} & \frac{5}{36} + \frac{\sqrt{15}}{24} & \frac{2}{9} & \frac{5}{36} - \frac{\sqrt{15}}{24} \\
\frac{1}{2} + \frac{\sqrt{15}}{10} & \frac{5}{36} + \frac{\sqrt{15}}{30} & \frac{2}{9} + \frac{\sqrt{15}}{15} & \frac{5}{36} \\
\hline
& \frac{5}{18} & \frac{4}{9} & \frac{5}{18}
\end{array}
\end{equation}

\subsection{8th-Order Gauss-Legendre Butcher Tableau}
The 4-stage, 8th-order Butcher Tableau with 32 digits of precision.
\begin{center}
\scriptsize
\begin{tabular}{ccc}
\textbf{Stage Nodes ($c_i$)} &\phantom{mmmm} &\textbf{Weights ($b_i$)} \\[0.5em]
$
\begin{aligned}
c_1 &= 0.06943184420297371238854284555813 \\
c_2 &= 0.33000947820757186759865892120042 \\
c_3 &= 0.66999052179242813240134107879958 \\
c_4 &= 0.93056815579702628761145715444187
\end{aligned}
$
&
\phantom{mmmm}
&
$
\begin{aligned}
b_1 &= 0.17392742256872692868653197461040 \\
b_2 &= 0.32607257743127307131346802538960 \\
b_3 &= 0.32607257743127307131346802538960 \\
b_4 &= 0.17392742256872692868653197461040
\end{aligned}
$
\end{tabular}
\end{center}
\centerline{\scriptsize  {\bf Stage Matrix Entries} ($a_{ij}$)}
\begin{center}
\scriptsize 
\begin{minipage}{0.35\textwidth}
\begin{align*}
           &\hskip6em \textbf{Row 1}\\
    a_{11} &=  0.08696371128436346434326598730520 \\
    a_{12} &= -0.02619985959089069670008594246830 \\
    a_{13} &=  0.01509374026369527376785640232490 \\
    a_{14} &= -0.00642574775419432902249360160367
\end{align*}
\begin{align*}
           &\hskip6em \textbf{Row 2}\\
    a_{21} &= 0.18738361514757530663415174092415 \\
    a_{22} &= 0.16303628871563653565673401269480 \\
    a_{23} &= -0.02677937984187063467657150117387 \\
    a_{24} &= 0.00636915418623065998434466875535
\end{align*}
\end{minipage}
\hskip 1em
\begin{minipage}{0.35\textwidth}
\begin{align*}
           &\hskip6em \textbf{Row 3}\\
    a_{31} &= 0.16755826838249626870218730585505 \\
    a_{32} &= 0.35285195727314370598990176843438 \\
    a_{33} &= 0.16303628871563653565673401269480 \\
    a_{34} &= -0.01345599258210340356560943180470
\end{align*}
\begin{align*}
           &\hskip6em \textbf{Row 4}\\
    a_{41} &= 0.18035317032292125770902557621407 \\
    a_{42} &= 0.31097883716757779754561162306470 \\
    a_{43} &= 0.35227243702216376801355403061452 \\
    a_{44} &= 0.08696371128436346434326598730520
\end{align*}
\end{minipage}
\end{center}

\subsection*{10th-Order Gauss-Legendre Butcher Tableau}
The 5-stage, 10th-order Butcher Tableau with 32 digits of precision.
\begin{center}
\scriptsize
\begin{tabular}{ccc}
\textbf{Stage Nodes ($c_i$)} &\phantom{mmmm} &\textbf{Weights ($b_i$)} \\[0.5em]
$
\begin{aligned}
c_1 &= 0.04691007703066800360118656085031 \\
c_2 &= 0.23076534494715845448184278821844 \\
c_3 &= 0.50000000000000000000000000000000 \\
c_4 &= 0.76923465505284154551815721178156 \\
c_5 &= 0.95308992296933199639881343914969
\end{aligned}
$
&
\phantom{mmmm}
&
$
\begin{aligned}
b_1 &= 0.11846344252809454375713202035994 \\
b_2 &= 0.23931433529974125961272126693111 \\
b_3 &= 0.28444444444444444444444444444444  \\
b_4 &= 0.23931433529974125961272126693111 \\
b_5 &= 0.11846344252809454375713202035994
\end{aligned}
$
\end{tabular}
\end{center}

\begin{center}
\scriptsize 
\textbf{ Stage Matrix Entries ($a_{ij}$)}\\
\begin{minipage}[t]{0.45\textwidth}
\begin{align*}
           &\hskip6em  \textbf{Row 1}\\
    a_{11} &=  0.05923172126404727187856601017997 \\
    a_{12} &= -0.01783063595561081525091763781744 \\
    a_{13} &=  0.00913988019488390772702758169116 \\
    a_{14} &= -0.00494576326190820521316694605917 \\
    a_{15} &=  0.00131487478925584445967755285579
\end{align*}
\begin{align*}
           &\hskip6em  \textbf{Row 2}\\
    a_{21} &=  0.12467364402636901844577884784116 \\
    a_{22} &=  0.11965716764987062980636063346555 \\
    a_{23} &= -0.01831885918739908856658934526647 \\
    a_{24} &=  0.00739976735508892694553259969188 \\
    a_{25} &= -0.00184637489677103214923994751368
\end{align*}
\begin{align*}
           &\hskip6em  \textbf{Row 3}\\
    a_{31} &=  0.12151608752230752538964344402283 \\
    a_{32} &=  0.26127548981650392306774641973053 \\
    a_{33} &=  0.14222222222222222222222222222222  \\
    a_{34} &= -0.02195982511624518268046768666164 \\
    a_{35} &=  0.00294602555521151199341883713653
\end{align*}
\end{minipage}
\hskip 1em
\begin{minipage}[t]{0.45\textwidth}
\begin{align*}
           &\hskip6em  \textbf{Row 4}\\
    a_{41} &=  0.12030981742486557590637196787362 \\
    a_{42} &=  0.23191456794265233266718866723923 \\
    a_{43} &=  0.30276330363184353298642084661853 \\
    a_{44} &=  0.11965716764987062980636063346555 \\
    a_{45} &= -0.00621020159639052585250444372982
\end{align*}
\begin{align*}
           &\hskip6em  \textbf{Row 5}\\
    a_{51} &=  0.11714856773883869929745446750415 \\
    a_{52} &=  0.24426009856164946482588821299028 \\
    a_{53} &=  0.27530456424956053671741686275328 \\
    a_{54} &=  0.25714497125535207486363890474855 \\
    a_{55} &=  0.05923172126404727187856601017997
\end{align*}
\end{minipage}
\end{center}

\subsubsection{12th-Order Gauss-Legendre Butcher Tableau}
The 6-stage, 12th-order Butcher Tableau with 32 digits of precision
\begin{center}
\scriptsize
\begin{tabular}{ccc}
\textbf{Stage Nodes ($c_i$)} &\phantom{mmmm} &\textbf{Weights ($b_i$)} \\[0.5em]
$
\begin{aligned}
    c_1 &= 0.03376524289842398612197577545465 \\
    c_2 &= 0.16939530676686774311894236047318 \\
    c_3 &= 0.38069040695840154571932338148970 \\
    c_4 &= 0.61930959304159845428067661851030 \\
    c_5 &= 0.83060469323313225688105763952682 \\
    c_6 &= 0.96623475710157601387802422454535
\end{aligned}
$
&
\phantom{mmmm}
&
$
\begin{aligned}
    b_1 &= 0.08566224618958517252277579178949 \\
    b_2 &= 0.18038078652406930335017409240366 \\
    b_3 &= 0.23395696728634552412705011580685 \\
    b_4 &= 0.23395696728634552412705011580685 \\
    b_5 &= 0.18038078652406930335017409240366 \\
    b_6 &= 0.08566224618958517252277579178949
\end{aligned}
$
\end{tabular}
\end{center}

\begin{center}
\scriptsize 
\textbf{ Stage Matrix Entries ($a_{ij}$)}\\
\begin{minipage}[t]{0.45\textwidth}
\begin{align*}
           &\hskip6em  \textbf{Row 1}\\
    a_{11} &=  0.04283112309479258626138789589475 \\
    a_{12} &= -0.01255762837335661011195655554625 \\
    a_{13} &=  0.00557476646535555986877903822818 \\
    a_{14} &= -0.00310860551065116744839818815124 \\
    a_{15} &=  0.00127116790938676239162985331562 \\
    a_{16} &= -0.00024558068710314483946626828641
\end{align*}
\begin{align*}
           &\hskip6em  \textbf{Row 2}\\
    a_{21} &=  0.08865487711929944062145889369324 \\
    a_{22} &=  0.09019039326203465167508704620183 \\
    a_{23} &= -0.01235650742136195759187383216892 \\
    a_{24} &=  0.00539168913926839352726749069150 \\
    a_{25} &= -0.00188612187315147571380907481363 \\
    a_{26} &=  0.00029097654077869060081183686916
\end{align*}
\begin{align*}
           &\hskip6em  \textbf{Row 3}\\
    a_{31} &=  0.08728956977264879555234249112444 \\
    a_{32} &=  0.19159930768407428399587442111893 \\
    a_{33} &=  0.11697848364317276206352505790343 \\
    a_{34} &= -0.01869871029272891350179965416045 \\
    a_{35} &=  0.00441617417531778152575239920194 \\
    a_{36} &= -0.00059441802408316391637133369859
\end{align*}
\end{minipage}
\hskip 1em
\begin{minipage}[t]{0.45\textwidth}
\begin{align*}
           &\hskip6em  \textbf{Row 4}\\
    a_{41} &=  0.08625666550248192070640445809090 \\
    a_{42} &=  0.17596461234875152182442169320172 \\
    a_{43} &=  0.25265567757907443762884976996730 \\
    a_{44} &=  0.11697848364317276206352505790343 \\
    a_{45} &= -0.01121852115999424888493132714526 \\
    a_{46} &=  0.00180479109015978161730628292887
\end{align*}
\begin{align*}
           &\hskip6em  \textbf{Row 5}\\
    a_{51} &=  0.08537126964880648192196395492033 \\
    a_{52} &=  0.18226690839722077906398316721729 \\
    a_{53} &=  0.22856527814707713060232532463428 \\
    a_{54} &=  0.24631347470770748171638210344583 \\
    a_{55} &=  0.09019039326203465167508704620183 \\
    a_{56} &= -0.00223594892742918868840224021234
\end{align*}
\begin{align*}
           &\hskip6em  \textbf{Row 6}\\
    a_{61} &=  0.08541666550248192070640445809090 \\
    a_{62} &=  0.17910961861468254095854394571928 \\
    a_{63} &=  0.22954079937695876173542026249123 \\
    a_{64} &=  0.23706557280011504936496464871261 \\
    a_{65} &=  0.19293841489742591346213064794991 \\
    a_{66} &=  0.04283112309479258626138789589475
\end{align*}
\end{minipage}
\end{center}

\bibliographystyle{elsarticle-num}
\bibliography{bibliography}
\end{document}